\documentclass{amsart}%
\usepackage{hyperref}
\usepackage{amssymb}
\usepackage{amsmath}
\usepackage{amsthm}
\usepackage{amsfonts}
\usepackage{graphicx}%
\@ifundefined{pdfoutput} { \DeclareGraphicsRule{.wmf}{bmp}{}{}
\DeclareGraphicsRule{.jpg}{bmp}{}{}
\DeclareGraphicsRule{.png}{bmp}{}{}
\DeclareGraphicsRule{.cdr}{bmp}{}{}
\DeclareGraphicsRule{.gif}{bmp}{}{} } { \ifnum\pdfoutput=0\relax
\DeclareGraphicsRule{.wmf}{bmp}{}{}
\DeclareGraphicsRule{.jpg}{bmp}{}{}
\DeclareGraphicsRule{.png}{bmp}{}{}
\DeclareGraphicsRule{.cdr}{bmp}{}{}
\DeclareGraphicsRule{.gif}{bmp}{}{} \fi \ifnum\pdfoutput=1\relax
\DeclareGraphicsRule{.eps}{pdf}{}{}
\DeclareGraphicsRule{.wmf}{jpg}{}{} \fi } \makeatother
\newtheorem{thm}{Theorem}[section]
\newtheorem{cor}[thm]{Corollary}
\newtheorem{ex}[thm]{Example}

\newtheorem{lem}[thm]{Lemma}
\newtheorem{nota}[thm]{Notation}
\newtheorem{prop}[thm]{Proposition}

\theoremstyle{definition}
\newtheorem*{acknowledgement}{Acknowledgement}
\newtheorem{df}[thm]{Definition}

\newtheorem{rem}[thm]{Remark}

\numberwithin{equation}{section}
\begin{document}

\def\H{\mathcal{H}}
\def\P{\mathcal{P}}
\def\WC{W_{\mathbb C}}
\def\HC{H_{\mathbb C}}

\def\hfootnote#1{}

\title[Square integrable holomorphic functions]{Square integrable holomorphic functions
on infinite-dimensional Heisenberg type groups}
\author[Driver]{Bruce K. Driver{$^{\dagger }$}}
\thanks{\footnotemark {$^\dagger$}This research was supported in part by NSF Grant DMS-0504608 and the Miller Institute at the University of California, at Berkeley.}
\address{Department of Mathematics, 0112\\
University of California, San Diego \\
La Jolla, CA 92093-0112 } \email{driver{@}euclid.ucsd.edu}
\author[Gordina]{Maria Gordina{$^{*}$}}
\thanks{\footnotemark {$*$}Research was supported in part by NSF Grant DMS-0706784 and the Humboldt Foundation Research Fellowship.}
\address{Department of Mathematics\\
University of Connecticut\\
Storrs, CT 06269, U.S.A. } \email{gordina@math.uconn.edu}

\keywords{Heisenberg group, holomorphic, heat kernel,
quasi-invariance, Taylor map} \subjclass{Primary; 35K05,43A15
Secondary; 58G32}


\date{\today \ \emph{File:\jobname{.tex}}}

\begin{abstract}
We introduce a class of non-commutative, complex,
infinite-dimensional Heisenberg like Lie groups based on an abstract
Wiener space. The holomorphic functions which are also square
integrable with respect to a heat kernel measure $\mu$ on these
groups are studied. In particular, we establish a unitary
equivalence between the square integrable holomorphic functions and
a certain completion of the universal enveloping algebra of the
``Lie algebra'' of this class of groups. Using quasi-invariance of
the heat kernel measure, we also construct a skeleton map which
characterizes globally defined functions from the
$L^{2}\left(\nu\right)$-closure of holomorphic polynomials by their
values on the Cameron-Martin subgroup.

\end{abstract}

\maketitle

\tableofcontents


\renewcommand{\contentsname}{Table of Contents}

\section{Introduction\label{s.h1}}

The aim of this paper is to study spaces of holomorphic functions on
an infinite-dimensional Heisenberg like group based on a complex
abstract Wiener space. In particular, we prove Taylor, skeleton, and
holomorphic chaos isomorphism theorems. The tools we use come from
properties of heat kernel measures on such groups which have been
constructed and studied in \cite{DG07b}. We will state the main
results of our paper and then conclude this introduction with a
brief discussion of how our results relate to the existing
literature.

\subsection{Statements of the main results\label{s.h1.1}}

\subsubsection{The Heisenberg like groups and heat kernel
measures\label{s.h1.1.1}}

The basic input to our theory is a complex abstract Wiener space,
$\left( W,H,\mu\right)$,  as in Notation \ref{n.h2.4} which is
equipped with a continuous skew-symmetric bi-linear form
$\omega:W\times W\rightarrow \mathbf{C}$ as in Notation
\ref{n.h3.1}. Here and throughout this paper, $\mathbf{C}$ is a
finite dimensional complex inner product space. The space,
$G:=W\times\mathbf{C}$,  becomes an infinite-dimensional
\textquotedblleft Heisenberg like\textquotedblright\ group when
equipped with the following multiplication rule
\begin{equation}
\left(  w_{1},c_{1}\right)  \cdot\left(  w_{2},c_{2}\right)  =\left(
w_{1}+w_{2},c_{1}+c_{2}+\frac{1}{2}\omega\left(  w_{1},w_{2}\right)  \right)
. \label{e.h1.1}%
\end{equation}
A typical example of such a group is the Heisenberg group of a
symplectic vector space, but in our setting we have an additional
structure of an abstract Wiener space to carry out the heat kernel
measure analysis.

The group $G$ contains the \emph{Cameron--Martin} group, $G_{CM}%
:=H\times\mathbf{C}$,  as a subgroup. The Lie algebras of $G$ and
$G_{CM}$ will be denoted by $\mathfrak{g}$ and $\mathfrak{g}_{CM}$
respectively which, as sets, may be identified with $G$ and $G_{CM}$
respectively --- see Definition \ref{d.h3.2}, Notation \ref{n.h3.3},
and Proposition \ref{p.h3.5} for more details.

Let $b\left(  t\right)  =\left(  B\left(  t\right), B_{0}\left(
t\right) \right)$ be a Brownian motion on $\mathfrak{g}$ associated
to the natural Hilbertian structure on $\mathfrak{g}_{CM}$ as
described in Eq. (\ref{e.h4.1}). The Brownian motion $\left\{
g\left(  t\right)  \right\}_{t\geq0}$ on $G$ is then the solution to
the stochastic differential equation,
\begin{equation}
dg\left(  t\right)=g\left(  t\right)  \circ db\left(  t\right)
\text{ with
}g\left(  0\right)=\mathbf{e}=\left(  0, 0\right). \label{e.h1.2}%
\end{equation}
The explicit solution to Eq. (\ref{e.h1.2}) may be found in Eq. (\ref{e.h4.2}%
). For each $T>0$ we let $\nu_{T}:=\operatorname*{Law}\left( g\left(
T\right)  \right)$ be the \emph{heat kernel} measure on $G$ at time
$T$ as explained in Definitions \ref{d.h4.1} and \ref{d.h4.2}.
Analogous to the abstract Wiener space setting, $\nu_{T}$ is left
(right) quasi-invariant by an element, $h\in G$,  iff $h\in G_{CM}$,
while $\nu_{T}\left(  G_{CM}\right)=0$, see Theorem \ref{t.h4.5},
Proposition \ref{p.h4.6}, and \cite[Proposition 6.3]{DG07b}.

In addition to the above infinite-dimensional structures we will
need corresponding finite dimensional approximations. These
approximations will be indexed by $\operatorname*{Proj}\left(
W\right)$ which we now define.

\begin{nota}
\label{n.h1.1}Let $\operatorname*{Proj}\left(  W\right)$denote the
collection of finite rank continuous linear maps, $P:W\rightarrow
H$,  such that $P|_{H}$ is an orthogonal projection. (Explicitly,
$P$ must be as in Eq. \eqref{e.h2.17} below.) Further, let
$G_{P}:=PW\times\mathbf{C}$ (a subgroup of $G_{CM})$ and
$\pi_{P}:G\rightarrow G_{P}$ be the projection map defined by
$\pi_{P}\left(  w, c\right):=\left(  Pw, c\right)$. \end{nota}

To each $P\in\operatorname*{Proj}\left(  W\right)$, $G_{P}$ is a
finite dimensional Lie group. The Brownian motions and heat kernel
measures, $\left\{  \nu_{t}^{P}\right\}_{t>0}$,  on $G_{P}$ are
constructed similarly to those on $G$--see Definition \ref{d.h4.10}.
We will use $\left\{  \left( G_{P}, \nu_{T}^{P}\right)
\right\}_{P\in\operatorname*{Proj}\left( W\right) }$ as finite
dimensional approximations to $\left(  G, \nu_{T}\right)$.
\subsubsection{The Taylor isomorphism theorem\label{s.h1.1.2}}

The Taylor map, $\mathcal{T}_{T}$,  is a unitary map relating the
\textquotedblleft square integrable\textquotedblright\ holomorphic
functions on $G_{CM}$ with the collection of their derivatives at
$\mathbf{e\in}G_{CM}$. Before we can state this theorem we need to
introduce the two Hilbert spaces involved.

In what follows, $\mathcal{H}\left(  G_{CM}\right)$ and
$\mathcal{H}\left( G\right)$ will denote the space of holomorphic
functions on $G_{CM}$ and $G$ respectively. (See Section \ref{s.h5}
for the properties of these function spaces which are used
throughout this paper.) We also let $\mathbf{T:=T} \left(
\mathfrak{g}_{CM}\right)$ be the algebraic tensor algebra over
$\mathfrak{g}_{CM}$, $\mathbf{T}^{\prime}$ be its algebraic dual,
$J$ be the two-sided ideal in $\mathbf{T}$ generated by
\begin{equation}
\{h\otimes k-k\otimes h-[h,k]:h,k\in\mathfrak{g}_{CM}\}, \label{e.h1.3}%
\end{equation}
and $J^{0}=\{\alpha\in\mathbf{T}^{\prime}:\alpha\left(  J\right)
=0\}$ be the backwards annihilator of $J$--see Notation
\ref{n.h6.1}. Given $f\in\mathcal{H}\left(  G\right)$ we let
$\alpha:=\mathcal{T}f$ denote the element of $J^{0}$ defined by
$\left\langle \alpha, 1\right\rangle =f\left( \mathbf{e}\right)$ and
\[
\left\langle \alpha, h_{1}\otimes\dots\otimes h_{n}\right\rangle
:=\left( \tilde{h}_{1}\dots\tilde{h}_{n}f\right)  \left(
\mathbf{e}\right)
\]
where $h_{i}\in\mathfrak{g}_{CM}$ and $\tilde{h}_{i}$ is the left
invariant vector field on $G_{CM}$ agreeing with $h_{i}$ at
$\mathbf{e}$--see Proposition \ref{p.h3.5} and Definition
\ref{d.h6.2}. We call $\mathcal{T}$ the Taylor map since
$\mathcal{T}f\in J^{0}\left(  \mathfrak{g}_{CM}\right)$ encodes all
of the derivatives of $f$ at $\mathbf{e}$.

\begin{df}
[$L^{2}$--holomorphic functions on $G_{CM}$]\label{d.h1.2} For
$T>0$,  let
\begin{align}
\left\Vert f\right\Vert_{\mathcal{H}_{T}^{2}\left(  G_{CM}\right)  }
&
=\sup_{P\in\operatorname*{Proj}\left(  W\right)  }\left\Vert f|_{G_{P}%
}\right\Vert_{L^{2}\left(  G_{P},\nu_{T}^{P}\right)  }\text{ for all }%
f\in\mathcal{H}\left(  G_{CM}\right),\text{ and }\label{e.h1.4}\\
\mathcal{H}_{T}^{2}\left(  G_{CM}\right)   &  :=\left\{  f\in\mathcal{H}%
\left(  G_{CM}\right)  :\left\Vert
f\right\Vert_{\mathcal{H}_{T}^{2}\left(
G_{CM}\right)  }<\infty\right\}. \label{e.h1.5}%
\end{align}

\end{df}

In Corollary \ref{c.h6.6} below, we will see that
$\mathcal{H}_{T}^{2}\left( G_{CM}\right)$ is not empty and in fact
contains the space of holomorphic cylinder polynomials $\left(
\mathcal{P}_{CM}\right)$on $G_{CM}$ described in Eq. (\ref{e.h1.7})
below. Despite the fact that $\nu_{T}\left( G_{CM}\right)=0$,
$\mathcal{H}_{T}^{2}\left(  G_{CM}\right)$ should roughly be thought
of as the $\nu_{T}$--square integrable holomorphic functions on
$G_{CM}$.
\begin{df}
[Non-commutative Fock space]\label{d.h1.3}Let $T>0$ and
\[
\left\Vert \alpha\right\Vert_{J_{T}^{0}\left(
\mathfrak{g}_{CM}\right)
}^{2}:=\sum_{n=0}^{\infty}\frac{T^{n}}{n!}\sum_{h_{1},\dots,h_{n}\in
S}\left\vert \left\langle \alpha,h_{1}\otimes\dots\otimes
h_{n}\right\rangle
\right\vert ^{2}\text{ for all }\alpha\in J^{0}\left(  \mathfrak{g}%
_{CM}\right),
\]
where $S$ is any orthonormal basis for $\mathfrak{g}_{CM}$. The
\emph{non-commutative Fock space} is defined as
\[
J_{T}^{0}\left(  \mathfrak{g}_{CM}\right) :=\left\{  \alpha\in
J^{0}\left( \mathfrak{g}_{CM}\right) :\left\Vert \alpha\right\Vert
_{J_{T}^{0}\left( \mathfrak{g}_{CM}\right) }^{2}<\infty\right\}.
\]

\end{df}

It is easy to see that $\left\Vert \cdot\right\Vert_{J_{T}^{0}\left(
\mathfrak{g}_{CM}\right) }$ is a Hilbertian norm on $J_{T}^{0}\left(
\mathfrak{g}_{CM}\right)$--see Definition \ref{d.h6.4} and Eq.
(\ref{e.h6.8}). For a detailed introduction to such Fock spaces we
refer to \cite{GrossMall}.

\begin{rem}
\label{r.h1.4} When $\omega=0$,  $G\left(  \omega\right)$ is
commutative and the Fock space, $J_{T}^{0}\left(
\mathfrak{g}_{CM}\right)$, becomes the standard commutative bosonic
Fock space of symmetric tensors over $\mathfrak{g}_{CM}^{\ast}$.
\end{rem}

The following theorem is proved in Section \ref{s.h6}--see Theorem
\ref{t.h6.10}.

\begin{thm}
[The Taylor isomorphism]\label{t.h1.5}For all $T>0$,
$\mathcal{T}\left( \mathcal{H}_{T}^{2}\left(  G_{CM}\right)  \right)
\subset J_{T}^{0}\left( \mathfrak{g}_{CM}\right)$ and the linear
map,
\begin{equation}
\mathcal{T}_{T}:=\mathcal{T}|_{\mathcal{H}_{T}^{2}\left(  G_{CM}\right)
}:\mathcal{H}_{T}^{2}\left(  G_{CM}\right)  \rightarrow J_{T}^{0}\left(
\mathfrak{g}_{CM}\right), \label{e.h1.6}%
\end{equation}
is unitary.
\end{thm}

Associated to this theorem is an analogue of Bargmann's pointwise bounds which
appear in Theorem \ref{t.h6.11} below.

\subsubsection{The skeleton isomorphism theorem\label{s.h1.1.3}}
Similarly to how it has been done on a complex abstract Wiener space
by H.~Sugita in \cite{Sugita1994a,Sugita1994b}, the quasi-invariance
of the heat kernel measure $\nu_{T}$ allows us to define the
skeleton map from $L^{p}\left(  G, \nu_{T}\right)$ to a space of
functions on the Cameron-Martin subgroup $G_{CM}$, a set of
$\nu_{T}$-measure $0$.
\begin{df}
\label{d.h1.6}A \textbf{holomorphic cylinder polynomial on }$G$ is a
holomorphic cylinder function (see Definition \ref{d.h4.3}) of the form,
$f=F\circ\pi_{P}:G\rightarrow\mathbb{C}$,  where $P\in\operatorname*{Proj}%
\left(  W\right)$ and $F:PW\times\mathbf{C}\mathbb{\rightarrow C}$
is a holomorphic polynomial. The space of holomorphic cylinder
polynomials will be denoted by $\mathcal{P}$. \end{df}

The \textquotedblleft Gaussian\textquotedblright\ heat kernel bounds
in Theorem \ref{t.h4.11} easily imply that $\mathcal{P}\subset
L^{p}\left( \nu_{T}\right)$ for all $p<\infty$--see Corollary
\ref{c.h5.10}.

\begin{df}
[Holomorphic $L^{p}$--functions]\label{d.h1.7} For $T>0$ and
$1\leqslant p<\infty$,  let $\mathcal{H}_{T}^{p}\left(  G\right)$
denote the $L^{p}\left(  \nu_{T}\right)$~--~closure of
$\mathcal{P}\subset L^{p}\left( \nu_{T}\right)$. \end{df}

From Corollary \ref{c.h4.8} below, if $T>0$,  $p\in(1,\infty]$,
$f\in L^{p}\left(  G,\nu_{T}\right)$,  and $h\in G_{CM}$,  then
$\int_{G}\left\vert f\left(  h\cdot g\right)  \right\vert
d\nu_{T}\left(  g\right)  <\infty$. Thus, if
$f\in\mathcal{H}_{T}^{2}\left(  G\right)$ we may define the
\emph{skeleton map }(see Definition \ref{d.h4.7}) by%
\[
\left(  S_{T}f\right)  \left(  h\right)  :=\int_{G}f\left(  h\cdot
g\right) d\nu_{T}\left(  g\right).
\]
It is shown in Theorem \ref{t.h5.12} that $S_{T}\left(  \mathcal{H}_{T}%
^{2}\left(  G\right)  \right)  \subset\mathcal{H}_{T}^{2}\left(
G_{CM}\right)$ for all $T>0$.
\begin{thm}
[The skeleton isomorphism]\label{t.h1.8}For each $T>0$,  the
skeleton map,
$S_{T}:\mathcal{H}_{T}^{2}\left(  G\right)  \rightarrow\mathcal{H}_{T}%
^{2}\left(  G_{CM}\right)$,  is unitary.
\end{thm}

Following Sugita's results \cite{Sugita1994a, Sugita1994b} in the case of an abstract Wiener space, we call $S_{T}%
|_{\mathcal{H}_{T}^{2}\left(  G \right)}$ the skeleton map since it
characterizes $f\in\mathcal{H}_{T}^{2}\left(  G\right)$ by its
\textquotedblleft values\textquotedblright, $S_{T}f$, on $G_{CM}$.
Sugita would refer to $G_{CM}$ as the skeleton of $G\left(
\omega\right)$ owing to the fact that $\nu_{T}\left( G_{CM}\right)
=0$ as we show in Proposition \ref{p.h4.6}.

Theorem \ref{t.h1.8} is proved in Section \ref{s.h8} and relies on
two key density results from Section \ref{s.h7}. The first is Lemma
\ref{l.h7.3} (an infinite-dimensional version of \cite[Lemma
3.5]{D-G-SC07b}) which states that the finite rank tensors (see
Definition \ref{d.h7.2}) are dense inside of $J_{T}^{0}\left(
\mathfrak{g}_{CM}\right)$. The second is Theorem \ref{t.h7.1} which
states that
\begin{equation}
\mathcal{P}_{CM}:=\left\{  p|_{G_{CM}}:p\in\mathcal{P}\right\}  \label{e.h1.7}%
\end{equation}
is a dense subspace of $\mathcal{H}_{T}^{2}\left(  G_{CM}\right)$.
Matt Cecil \cite{Cecil2008} has modified the arguments presented in
Section \ref{s.h7} to cover the situation of path groups over graded
nilpotent Lie groups. Cecil's arguments are necessarily much more
involved because his Lie groups have nilpotency of arbitrary step.

\subsubsection{The holomorphic chaos expansion\label{s.h1.1.4}}

So far we have produced (for each $T>0$) two unitary isomorphisms,
the skeleton map $S_{T}$ and the Taylor isomorphism
$\mathcal{T}_{T}$,
\[
\mathcal{H}_{T}^{2}\left(  G\right)  \overset{S_{T}}{\underset{\cong
}{\longrightarrow}}\mathcal{H}_{T}^{2}\left(  G_{CM}\right)
\overset
{\mathcal{T}_{T}}{\underset{\cong}{\longrightarrow}}J_{T}^{0}\left(
\mathfrak{g}_{CM}\right).
\]
The next theorem gives an explicit formula for $\left(
\mathcal{T}_{T}\circ S_{T}\right)^{-1}:J_{T}^{0}\left(
\mathfrak{g}_{CM}\right)  \rightarrow \mathcal{H}_{T}^{2}\left(
G\right)$.
\begin{thm}
[The holomorphic chaos expansion]\label{t.h1.9} If
$f\in\mathcal{H}_{T}^{2}\left(  G\right)$ and
$\alpha_{f}:=\mathcal{T}_{T}S_{T}f$, then
\begin{equation}
f\left(  g\left(  T\right)  \right)  =\sum_{n=0}^{\infty}\left\langle
\alpha_{f},\int_{0\leq s_{1}\leq s_{2}\leq\dots\leq s_{n}\leq T}db\left(
s_{1}\right)  \otimes\dots\otimes db\left(  s_{n}\right)  \right\rangle
\label{e.h1.8}%
\end{equation}
where $b\left(  t\right)$ and $g\left(  t\right)$ are related as in
Eq. (\ref{e.h1.2}) or equivalently as in Eq. (\ref{e.h4.2}).
\end{thm}

This result is proved in Section \ref{s.h9} and in particular, see
Theorem \ref{t.h9.10}. The precise meaning of the right hand side of
Eq. (\ref{e.h1.8}) is also described there.

\subsection{Discussion\label{s.h1.2}}

As we noticed in Remark \ref{r.h1.4} when the form $\omega\equiv0$
the Fock space $J_{T}^{0}\left( \mathfrak{g}_{CM}\right)$ is the
standard commutative bosonic Fock space \cite{Fock1928}. In this
case the Taylor map is one of three isomorphisms between different
representations of a Fock space, one other being the Segal-Bargmann
transform. The history of the latter is described in
\cite{GrossMall} beginning with works of V.~Bargmann
\cite{Bargmann61} and I.~Segal in \cite{Segal1960a}. For other
relevant results see \cite{HallSengupta1998, Driver1998c}.

To put our results into perspective, recall that the classical Segal-Bargmann
space is the Hilbert space of holomorphic functions on $\mathbb{C}^{n}$ that
are square-integrable with respect to the Gaussian measure $d\mu_{n}%
(z)=\pi^{-n}e^{-|z|^{2}}dz$, where $dz$ is the $2n$--dimensional
Lebesgue measure. One of the features of functions in the
Segal-Bargmann space is that they satisfy the pointwise bounds
$\left\vert f(z)\right\vert \leqslant\Vert
f\Vert_{L^{2}(\mu_{n})}\exp(|z|^{2}/2)$ (compare with Theorem
\ref{t.h6.11}). As it is described in \cite{GrossMall}, if
$\mathbb{C}^{n}$ is replaced by an infinite-dimensional complex
Hilbert space $H$, one of the first difficulties is to find a
suitable version of the Gaussian measure. It can be achieved, but
only on a certain extension $W$ of $H$, which leads one to consider
the complex abstract Wiener space setting. From H.~Sugita's
\cite{Sugita1994a,Sugita1994b} work on holomorphic functions over a
complex abstract Wiener space, it is known that the pointwise bounds
control only the values of the holomorphic functions on $H$. This
difficulty explains, in part, the need to consider two function
spaces: one is of holomorphic functions on $H$ (or $G_{CM}$ in our
case) versus the square-integrable (weakly) holomorphic functions on
$W$ (or $G$ in our case).

The Taylor map has also been studied in other non-commutative
infinite-dimensional settings. M. Gordina \cite{Gordina2000a,
Gordina2000b, Gordina2002} considered the Taylor isomorphism in the
context of Hilbert-Schmidt groups, while M. Cecil \cite{Cecil2008}
considered the Taylor isomorphism for path groups over stratified
Lie groups. The nilpotentcy of the Heisenberg like groups studied in
this paper allow us to give a more complete description of the
square integrable holomorphic function spaces than was possible in
\cite{Gordina2000a, Gordina2000b, Gordina2002} for the
Hilbert-Schmidt groups.

Complex analysis in infinite dimensions in a somewhat different
setting has been studied by L. Lempert (e.g.\cite{Lempert2004}), and
for more results on Gaussian-like measures on infinite-dimensional
curved spaces see papers by D.~Pickrell (e.g.\cite{Pickrell1987,
Pickrell2000}). For another view of different representations of
Fock space, one can look at results in the field of white noise, as
presented in the book by N. Obata \cite{ObataBook}. The map between
an $L^{2}$-space and a space of symmetric tensors sometimes is
called the Segal isomorphism as in \cite{Kondratiev1980, KLPSW1996}.
For more background on this and related topics see \cite{HKPSbook}.

\section{Complex abstract Wiener spaces\label{s.h2}}

Suppose that $W$ is a complex separable Banach space and
$\mathcal{B}_{W}$ is the Borel $\sigma$--algebra on $W$. Let
$W_{\operatorname{Re}}$ denote $W$ thought of as a real Banach
space. For $\lambda\in\mathbb{C}$,  let $M_{\lambda}:W\rightarrow W$
be the operation of multiplication by $\lambda$.
\begin{df}
\label{d.h2.1} A measure $\mu$ on $(W,\mathcal{B}_{W})$ is called a
(mean zero, non-degenerate) \textbf{Gaussian measure} provided that
its characteristic functional is given by
\begin{equation}
\hat{\mu}(u):=\int_{W}e^{iu\left(  w\right)  }d\mu\left(  w\right)
=e^{-\frac{1}{2}q(u,u)}\text{ for all }u\in W_{\operatorname{Re}}^{\ast},
\label{e.h2.1}%
\end{equation}
where $q=q_{\mu}:W_{\operatorname{Re}}^{\ast}\times W_{\operatorname{Re}%
}^{\ast}\rightarrow\mathbb{R}$ is an inner product on $W_{\operatorname{Re}%
}^{\ast}$. If in addition, $\mu$ is invariant under multiplication
by $i$, that is, $\mu\circ M_{i}^{-1}=\mu$,  we say that $\mu$ is a
\textbf{complex Gaussian measure }on $W$.
\end{df}

\begin{rem}
\label{r.h2.2} Suppose $W=\mathbb{C}^{d}$ and let us write $w\in W$
as $w=x+iy$ with $x,y\in\mathbb{R}^{d}$. Then the most general
Gaussian measure on $W$ is of the form
\[
d\mu\left(  w\right)  =\frac{1}{Z}\exp\left(  -\frac{1}{2}Q\left[
\begin{array}
[c]{c}%
x\\
y
\end{array}
\right]  \cdot\left[
\begin{array}
[c]{c}%
x\\
y
\end{array}
\right]  \right)  dx \,  dy
\]
where $Q$ is a real positive definite $2d \times 2d$ matrix and $Z$
is a normalization constant. The matrix $Q$ may be written in
$2\times2$ block form as
\[
Q=\left[
\begin{array}
[c]{cc}%
A & B\\
B^{\operatorname{tr}} & C
\end{array}
\right].
\]
A simple exercise shows $\mu=\mu\circ M_{i}^{-1}$ iff $B=0$ and
$A=C$. Thus the general complex Gaussian measure on $\mathbb{C}^{d}$
is of the form
\begin{align*}
d\mu\left(  w\right)   &  =\frac{1}{Z}\exp\left(  -\frac{1}{2}\left(  Ax\cdot
x+Ay\cdot y\right)  \right)  dx\, dy\\
&  =\frac{1}{Z}\exp\left(  -\frac{1}{2}Aw\cdot\bar{w}\right)  dx\,
dy,
\end{align*}
where $A$ is a real positive definite matrix.
\end{rem}

Given a complex Gaussian measure $\mu$ as in Definition \ref{d.h2.1}, let
\begin{equation}
\left\Vert w\right\Vert_{H}:=\sup\limits_{u\in
W_{\operatorname{Re}}^{\ast }\setminus\left\{  0\right\}
}\frac{|u(w)|}{\sqrt{q(u,u)}}\text{ for all
}w\in W, \label{e.h2.2}%
\end{equation}
and define the \textbf{Cameron-Martin subspace}, $H\subset W$, by
\begin{equation}
H=\left\{  h\in W:\left\Vert h\right\Vert_{H}<\infty\right\}.
\label{e.h2.3}%
\end{equation}
The following theorem summarizes some of the standard properties of
the triple $\left(  W, H, \mu\right)$.

\begin{thm}
\label{t.h2.3}Let $\left(  W, H, \mu\right)$ be as above, where
$\mu$ is a complex Gaussian measure on $\left(  W,
\mathcal{B}_{W}\right)$. Then

\begin{enumerate}
\item $H$ is a dense complex subspace of $W$.
\item There exists a unique inner product, $\left\langle \cdot,\cdot
\right\rangle_{H}$,  on $H$ such that $\left\Vert h\right\Vert_{H}%
^{2}=\left\langle h,h\right\rangle $ for all $h\in H$. Moreover,
with this inner product $H$ is a complete separable complex Hilbert
space.

\item There exists $C<\infty$ such that%
\begin{equation}
\left\Vert h\right\Vert_{W}\leqslant C\left\Vert h\right\Vert
_{H}\text{ for
any }h\in H. \label{e.h2.4}%
\end{equation}

\item If $\left\{  e_{j}\right\}_{j=1}^{\infty}$ is an orthonormal basis for
$H$ and $u, v\in W_{\operatorname{Re}}^{\ast}$,  then%
\begin{equation}
q\left(  u, v\right)  =\left\langle u, v\right\rangle_{H_{\operatorname{Re}%
}^{\ast}}=\sum_{j=1}^{\infty}\left[  u\left(  e_{j}\right)  v\left(
e_{j}\right) + u\left(  ie_{j}\right)  v\left(  ie_{j}\right)
\right].
\label{e.h2.5}%
\end{equation}

\item $\mu\circ M_{\lambda}^{-1}=\mu$ for all $\lambda\in\mathbb{C}$ with
$\left\vert \lambda\right\vert =1$. \end{enumerate}
\end{thm}

\begin{proof}
We will begin with the proof of item 5. From Eq. (\ref{e.h2.1}), the
invariance of $\mu$ under multiplication by $i$ ( $\mu\circ
M_{i}^{-1}=\mu$ ) is equivalent to assuming that $q\left(  u\circ
M_{i},u\circ M_{i}\right) =q\left(  u, u\right)$ for all $u\in
W_{\operatorname{Re}}^{\ast}$. By polarization, we may further
conclude that
\begin{equation}
q\left(  u\circ M_{i}, v\circ M_{i}\right)=q\left(  u, v\right)
\text{ for
all }u, v\in W_{\operatorname{Re}}^{\ast}. \label{e.h2.6}%
\end{equation}
Taking $v=u\circ M_{i}$ in this identity then shows that $q\left(
u\circ M_{i}, -u\right)  =q\left(  u, u\circ M_{i}\right)$ and hence
that
\begin{equation}
q\left(  u,u\circ M_{i}\right) =0\text{ for any }u\in W_{\operatorname{Re}%
}^{\ast}. \label{e.h2.7}%
\end{equation}
Therefore if $\lambda=a+ib$ with $a, b\in\mathbb{R}$, we see that
\begin{align}
q\left(  u\circ M_{\lambda},u\circ M_{\lambda}\right)   &  =q\left(
au+bu\circ M_{i}, au+bu\circ M_{i}\right) \nonumber\\
&  =\left(  a^{2}+b^{2}\right)  q\left(  u, u\right)  =\left\vert
\lambda\right\vert ^{2}q\left(  u, u\right), \label{e.h2.8}%
\end{align}
from which it follows that $q\left(  u\circ M_{\lambda}, u\circ
M_{\lambda }\right)  =q\left(  u, u\right)  $ for all $u\in
W_{\operatorname{Re}}^{\ast}$ and $\left\vert \lambda\right\vert
=1$. Coupling this observation with Eq. (\ref{e.h2.1}) implies
$\mu\circ M_{\lambda}^{-1}=\mu$ for all $\left\vert
\lambda\right\vert =1$. If $\left\vert \lambda\right\vert =1$,  from
Eqs. (\ref{e.h2.2}) and
(\ref{e.h2.8}), it follows that%
\begin{align*}
\left\Vert \lambda w\right\Vert_{H}  &  =\sup_{u\in W_{\operatorname{Re}%
}^{\ast}\setminus\left\{  0\right\}  }\frac{\left\vert u\left(  \lambda
w\right)  \right\vert }{\sqrt{q\left(  u,u\right)  }}=\sup_{u\in
W_{\operatorname{Re}}^{\ast}\setminus\left\{  0\right\}  }\frac{\left\vert
u\circ M_{\lambda}\left(  w\right)  \right\vert }{\sqrt{q\left(  u\circ
M_{\lambda},u\circ M_{\lambda}\right)  }}\\
&  =\sup_{u\in W_{\operatorname{Re}}^{\ast}\setminus\left\{  0\right\}  }%
\frac{\left\vert u\left(  w\right)  \right\vert }{\sqrt{q\left(
u,u\right) }}=\left\Vert w\right\Vert_{H}\text{ for all }w\in W.
\end{align*}
In particular, if $\left\Vert h\right\Vert_{H}<\infty$ and
$\left\vert \lambda\right\vert =1$,  then $\left\Vert \lambda
h\right\Vert_{H}=\left\Vert h\right\Vert_{H}<\infty$ and hence
$\lambda H\subset H$ which shows that $H$ is a complex subspace of
$W$. From \cite[Theorem 2.3]{DG07b} summarizing some well-known
properties of Gaussian measures, we know that item 3. holds, $H$ is
a dense subspace of $W_{\operatorname{Re}}$,  and there exists a
unique real Hilbertian inner product, $\left\langle
\cdot,\cdot\right\rangle_{H_{\operatorname{Re}}}$,  on $H$ such that
$\left\Vert h\right\Vert_{H}^{2}=\left\langle h, h\right\rangle
_{H_{\operatorname{Re}}}$ for all $h\in H$. Polarizing the identity
$\left\Vert \lambda h\right\Vert_{H}=\left\Vert h\right\Vert_{H}$
implies
$\left\langle \lambda h,\lambda k\right\rangle_{H_{\operatorname{Re}}%
}=\left\langle h, k\right\rangle_{H_{\operatorname{Re}}}$ for all
$h,k\in H$. Taking $\lambda=i$ and $k=-ih$ then shows $\left\langle
ih,h\right\rangle_{\operatorname{Re}}=\left\langle
h,-ih\right\rangle_{\operatorname{Re}}$, and hence that
$\left\langle ih,h\right\rangle_{\operatorname{Re}}=0$ for all $h\in
H$. Using this information it is a simple matter to check that
\begin{equation}
\left\langle h, k\right\rangle_{H}:=\left\langle h, k\right\rangle
_{H_{\operatorname{Re}}}+i\left\langle h, ik\right\rangle
_{H_{\operatorname{Re}}}\text{ for all }h, k\in H, \label{e.h2.9}%
\end{equation}
is the unique complex inner product on $H$ such that $\operatorname{Re}%
\left\langle \cdot, \cdot\right\rangle_{H}=\left\langle \cdot,\cdot
\right\rangle_{H_{\operatorname{Re}}}$. 

So it only remains to prove Eq. (\ref{e.h2.5}). For a proof of the
first equality in Eq. (\ref{e.h2.5}), see \cite[Theorem 2.3]{DG07b}.
To prove the second equality in this equation, it suffices to
observe that $\left\{ e_{j},ie_{j}\right\}_{j=1}^{\infty}$ is an
orthonormal basis for $\left( H_{\operatorname{Re}},\left\langle
\cdot,\cdot\right\rangle
_{H_{\operatorname{Re}}}\right)$ and therefore,%
\[
\left\langle u, v\right\rangle_{H_{\operatorname{Re}}^{\ast}}=\sum
_{j=1}^{\infty}\left[  u\left(  e_{j}\right)  v\left(  e_{j}\right)
+u\left( ie_{j}\right)  v\left(  ie_{j}\right)  \right]  \text{ for
any }u, v\in H_{\operatorname{Re}}^{\ast}.
\]

\end{proof}

\begin{nota}
\label{n.h2.4}The triple, $\left(  W, H, \mu\right)$, appearing in
Theorem \ref{t.h2.3} will be called a \textbf{complex abstract
Wiener space}. (Notice that there is redundancy in this notation
since $\mu$ is determined by $H$, and $H$ is determined by $\mu$.)
\end{nota}

\begin{lem}
\label{l.h2.5}Suppose that $u, v\in W_{\operatorname{Re}}^{\ast}$
and $a, b\in\mathbb{C}$, then
\begin{equation}
\int_{W}e^{au+bv}d\mu=\exp\left(  \frac{1}{2}\left(  a^{2}q\left(
u,u\right) +b^{2}q\left(  v, v\right)  +2abq\left(  u, v\right)
\right)  \right).
\label{e.h2.10}%
\end{equation}

\end{lem}

\begin{proof}
Equation (\ref{e.h2.10}) is easily verified when both $a$ and $b$
are real. This suffices to complete the proof, since both sides of
Eq. (\ref{e.h2.10}) are analytic functions of $a,b\in\mathbb{C}$.
\end{proof}

\begin{lem}
\label{l.h2.6}Let $\left(  W,H,\mu\right)  $ be a complex abstract
Wiener space, then for any $\varphi\in W^{\ast}$,  we have
\begin{equation}
\int_{W}e^{\varphi\left(  w\right)  }d\mu\left(  w\right)  =1=\int
_{W}e^{\overline{\varphi\left(  w\right)  }}d\mu\left(  w\right),
\label{e.h2.11}%
\end{equation}%
\begin{equation}
\int_{W}\left\vert \operatorname{Re}\varphi\left(  w\right)  \right\vert
^{2}d\mu\left(  w\right)  =\int_{W}\left\vert \operatorname{Im}\varphi\left(
w\right)  \right\vert ^{2}d\mu\left(  w\right)  =\left\Vert \varphi\right\Vert
_{H^{\ast}}^{2}, \label{e.h2.12}%
\end{equation}
and%
\begin{equation}
\int_{W}\left\vert \varphi\left(  w\right)  \right\vert ^{2}d\mu\left(
w\right)  =2\left\Vert \varphi\right\Vert_{H^{\ast}}^{2}. \label{e.h2.13}%
\end{equation}
More generally, if $\mathbf{C}$ is another complex Hilbert space and
$\varphi\in L\left(  W,\mathbf{C}\right)$,  then
\begin{equation}
\int_{W}\left\Vert \varphi\left(  w\right)  \right\Vert_{\mathbf{C}}^{2}%
d\mu\left(  w\right)  =2\left\Vert \varphi\right\Vert_{H^{\ast}%
\otimes\mathbf{C}}^{2}. \label{e.h2.14}%
\end{equation}

\end{lem}

\begin{proof}
If $u=\operatorname{Re}\varphi$,  then $\varphi\left(  w\right)
=u\left( w\right)  -iu\left(  iw\right)$. Therefore by Eqs.
(\ref{e.h2.6}),
(\ref{e.h2.7}), and (\ref{e.h2.10}),%
\begin{align*}
\int_{W}e^{\varphi}d\mu &  =\int_{W}e^{u-iu\circ M_{i}}d\mu\\
&  =\exp\left(  \frac{1}{2}\left(  q\left(  u,u\right)  -q\left(  u\circ
M_{i},u\circ M_{i}\right)  -2iq\left(  u,u\circ M_{i}\right)  \right)
\right)  =1.
\end{align*}
Taking the complex conjugation of this identity shows
$\int_{W}e^{\overline {\varphi\left(  w\right)  }}d\mu\left(
w\right)  =1$. Also using Lemma
\ref{l.h2.5}, we have%
\begin{align*}
\int_{W}\left\vert \operatorname{Re}\varphi\left(  w\right)  \right\vert
^{2}d\mu\left(  w\right)   &  =q\left(  u, u\right)  \text{ and}\\
\int_{W}\left\vert \operatorname{Im}\varphi\left(  w\right)
\right\vert ^{2}d\mu\left(  w\right)   &  =\int_{W}\left\vert
u\left(  iw\right) \right\vert ^{2}d\mu\left(  w\right)  =q\left(
u\circ M_{i}, u\circ M_{i}\right)  =q\left(  u, u\right).
\end{align*}
To evaluate $q\left(  u,u\right)$,  let $\left\{  e_{k}\right\}
_{k=1}^{\infty}$ be an orthonormal basis for $H$ so that $\left\{
e_{k},ie_{k}\right\}_{k=1}^{\infty}$ is an orthonormal basis for
$\left( H_{\operatorname{Re}},\operatorname{Re}\left\langle
\cdot,\cdot\right\rangle_{H}\right)$. Then by Eq. (\ref{e.h2.5}),
\[
q\left(  u, u\right)  =\sum_{k=1}^{\infty}\left[  \left\vert u\left(
e_{k}\right)  \right\vert ^{2}+\left\vert u\left(  ie_{k}\right)
\right\vert ^{2}\right]  =\sum_{k=1}^{\infty}\left\vert
\varphi\left(  e_{k}\right) \right\vert ^{2}=\left\Vert
\varphi\right\Vert_{H^{\ast}}^{2}.
\]
To prove Eq. (\ref{e.h2.14}), apply \cite[Eq. (2.13)]{DG07b} to find
\begin{align*}
\int_{W}\left\Vert \varphi\left(  w\right)  \right\Vert_{\mathbf{C}}^{2}%
d\mu\left(  w\right)   &  =\sum_{k=1}^{\infty}\left[  \left\Vert
\varphi\left(  e_{k}\right)  \right\Vert_{\mathbf{C}}^{2}+\left\Vert
\varphi\left(  ie_{k}\right)  \right\Vert_{\mathbf{C}}^{2}\right] \\
&  =2\sum_{k=1}^{\infty}\left\Vert \varphi\left(  e_{k}\right)  \right\Vert
_{\mathbf{C}}^{2}=2\left\Vert \varphi\right\Vert_{H^{\ast}\otimes\mathbf{C}%
}^{2}.
\end{align*}

\end{proof}

\begin{rem}
[Heat kernel interpretation of Lemma \ref{l.h2.6}]\label{r.h2.7}The measure
$\mu$ formally satisfies
\[
\int_{W}f\left(  w\right)  d\mu\left(  w\right)  =\left(  e^{\frac{1}{2}%
\Delta_{H_{\operatorname{Re}}}}f\right)  \left(  0\right),
\]
where
$\Delta_{H_{\operatorname{Re}}}=\sum_{j=1}^{\infty}\partial_{e_{j}}^{2}$
and $\left\{  e_{j}\right\}_{j=1}^{\infty}$ is an orthonormal basis
for $H_{\operatorname{Re}}$. If $f$ is holomorphic or
anti-holomorphic, then $f$ is harmonic and therefore
\[
\int_{W}f\left(  w\right)  d\mu\left(  w\right)  =\left(  e^{\frac{1}{2}%
\Delta_{H_{\operatorname{Re}}}}f\right)  \left(  0\right)  =f\left(  0\right)
.
\]
Applying this identity to $f\left(  w\right)  =e^{\varphi\left(  w\right)  }$
or $f\left(  w\right)  =\overline{e^{\varphi\left(  w\right)  }}$ with
$\varphi\in W^{\ast}$ gives Eq. \eqref{e.h2.11}. If $u\in W_{\operatorname{Re}%
}^{\ast}$,  we have
\begin{align*}
\int_{W}u^{2}\left(  w\right)  d\mu\left(  w\right)   &  =\left(  e^{\frac
{1}{2}\Delta_{H_{\operatorname{Re}}}}u^{2}\right)  \left(  0\right)
=\sum_{n=0}^{\infty}\frac{1}{2^{n}n!}\left(  \Delta_{H_{\operatorname{Re}}%
}^{n}u^{2}\right)  \left(  0\right) \\
&  =\frac{1}{2}\left(  \Delta_{H_{\operatorname{Re}}}u^{2}\right)  \left(
0\right)  =\frac{1}{2}\sum_{j=1}^{\infty}\left(  \partial_{e_{j}}^{2}%
u^{2}\right)  \left(  0\right) \\
&  =\sum_{j=1}^{\infty}u\left(  e_{j}\right)^{2}=\left\Vert
u\right\Vert _{H_{\operatorname{Re}}}^{2}.
\end{align*}
Eqs. \eqref{e.h2.12} and \eqref{e.h2.13} now follow easily from this identity.
\end{rem}

\subsection{The structure of the projections\label{s.h2.1}}

Let $i:H\rightarrow W$ be the inclusion map and
$i^{\ast}:W^{\ast}\rightarrow H^{\ast}$ be its transpose, i.e.
$i^{\ast}\ell:=\ell\circ i$ for all $\ell\in W^{\ast}$. Also let
\begin{equation}
H_{\ast}:=\left\{  h\in H:\left\langle \cdot,h\right\rangle_{H}%
\in\operatorname*{Ran}\left(  i^{\ast}\right)  \subset H^{\ast}\right\}
\label{e.h2.15}%
\end{equation}
or in other words, $h\in H$ is in $H_{\ast}$ iff $\left\langle \cdot
, h\right\rangle_{H}\in H^{\ast}$ extends to a continuous linear
functional on $W$. (We will continue to denote the continuous
extension of $\left\langle \cdot,h\right\rangle_{H}$ to $W$ by
$\left\langle \cdot,h\right\rangle_{H}.)$ Because $H$ is a dense
subspace of $W$,  $i^{\ast}$ is injective, and because $i$ is
injective, $i^{\ast}$ has a dense range. Since $h\in
H\rightarrow\left\langle \cdot,h\right\rangle_{H}\in H^{\ast}$ is a
conjugate linear isometric isomorphism, it follows from the above
comments that $H_{\ast}\ni h\rightarrow\left\langle
\cdot,h\right\rangle_{H}\in W^{\ast}$ is a conjugate linear
isomorphism too, and that $H_{\ast}$ is a dense subspace of
$H$.

\begin{lem}
\label{l.h2.8} There is a one to one correspondence between
$\operatorname*{Proj}\left(  W\right)$ (see Notation \ref{n.h1.1})
and the collection of finite rank orthogonal projections, $P$,  on
$H$ such that $PH\subset H_{\ast}$. \end{lem}

\begin{proof}
If $P\in\operatorname*{Proj}\left(  W\right)$ and $u\in PW\subset
H$,  then,
because $P|_{H}$ is an orthogonal projection, we have%
\begin{equation}
\left\langle Ph,u\right\rangle_{H}=\left\langle h,Pu\right\rangle
_{H}=\left\langle h,u\right\rangle_{H}\text{ for all }h\in H. \label{e.h2.16}%
\end{equation}
Since $P:W\rightarrow H$ is continuous, it follows that $u\in
H_{\ast}$,  i.e. $PW\subset H_{\ast}$.

Conversely, suppose that $P:H\rightarrow H$ is a finite rank
orthogonal projection such that $PH\subset H_{\ast}$. Let $\left\{
e_{j}\right\}_{j=1}^{n}$ be an orthonormal basis for $PH$ and
$\ell_{j}\in W^{\ast}$ such that $\ell_{j}|_{H}=\left\langle \cdot,
e_{j}\right\rangle_{H}$. Then we may extend $P$ uniquely to a
continuous operator from $W$ to $H$ (still denoted by
$P)$ by letting%
\begin{equation}
Pw:=\sum_{j=1}^{n}\ell_{j}\left(  w\right)  e_{j}=\sum_{j=1}^{n}\left\langle
w,e_{j}\right\rangle_{H}e_{j}\text{ for all }w\in W. \label{e.h2.17}%
\end{equation}
From \cite[Eq. 3.43]{DG07b}, there exists $C=C\left(  P\right)  <\infty$ such
that%
\begin{equation}
\left\Vert Pw\right\Vert_{H}\leqslant C\left\Vert w\right\Vert
_{W}\text{ for
all }w\in W. \label{e.h2.18}%
\end{equation}

\end{proof}

\section{Complex Heisenberg like groups\label{s.h3}}

In this section we review the infinite-dimensional Heisenberg like
groups and Lie algebras which were introduced in \cite[Section
3]{DG07b}.

\begin{nota}
\label{n.h3.1}Let $\left(  W, H, \mu\right)$ be a complex abstract
Wiener space, $\mathbf{C}$ be a complex finite dimensional inner
product space, and $\omega: W\times W\rightarrow\mathbf{C}$ be a
continuous skew symmetric bilinear quadratic form on $W$. Further,
let
\begin{equation}
\left\Vert \omega\right\Vert_{0}:=\sup\left\{  \left\Vert
\omega\left( w_{1},w_{2}\right)  \right\Vert
_{\mathbf{C}}:w_{1},w_{2}\in W\text{ with }\left\Vert
w_{1}\right\Vert_{W}=\left\Vert w_{2}\right\Vert_{W}=1\right\}
\label{e.h3.1}%
\end{equation}
be the uniform norm on $\omega$ which is finite by the assumed
continuity of $\omega$. \end{nota}

\begin{df}
\label{d.h3.2}Let $\mathfrak{g}$ denote $W\times\mathbf{C}$ when
thought of as a Lie algebra with the Lie bracket operation given by
\begin{equation}
\left[  \left(  A,a\right),\left(  B,b\right)  \right]  :=\left(
0,\omega\left(  A,B\right)  \right). \label{e.h3.2}%
\end{equation}
Let $G=G\left(  \omega\right)$ denote $W\times\mathbf{C}$ when
thought of as a group with the multiplication law given by
\begin{equation}
g_{1}g_{2}=g_{1}+g_{2}+\frac{1}{2}\left[  g_{1},g_{2}\right]  \text{ for any
}g_{1},g_{2}\in G \label{e.h3.3}%
\end{equation}
or equivalently by Eq. (\ref{e.h1.1}).
\end{df}

It is easily verified that $\mathfrak{g}$ is a Lie algebra and $G$
is a group. The identity of $G$ is the zero element,
$\mathbf{e:}=\left(  0,0\right)$.
\begin{nota}
\label{n.h3.3}Let $\mathfrak{g}_{CM}$ denote $H\times\mathbf{C}$
when viewed as a Lie subalgebra of $\mathfrak{g}$ and $G_{CM}$
denote $H\times\mathbf{C}$ when viewed as a subgroup of $G=G\left(
\omega\right)$. We will refer to $\mathfrak{g}_{CM}$ ($G_{CM})$ as
the \textbf{Cameron--Martin subalgebra (subgroup) }of $\mathfrak{g}$
$\left(  G\right)$. (For explicit examples of such $\left(
W,H,\mathbf{C}, \omega\right)$,  see \cite{DG07b}.)
\end{nota}

We equip $G=\mathfrak{g}=W\times\mathbf{C}$ with the Banach space
norm
\begin{equation}
\left\Vert \left(  w, c\right)  \right\Vert
_{\mathfrak{g}}:=\left\Vert
w\right\Vert_{W}+\left\Vert c\right\Vert_{\mathbf{C}} \label{e.h3.4}%
\end{equation}
and $G_{CM}=\mathfrak{g}_{CM}=H\times\mathbf{C\ }$with the Hilbert space inner
product,%
\begin{equation}
\left\langle \left(  A, a\right), \left(  B, b\right) \right\rangle
_{\mathfrak{g}_{CM}}:=\left\langle A, B\right\rangle
_{H}+\left\langle
a, b\right\rangle_{\mathbf{C}}. \label{e.h3.5}%
\end{equation}
The associate Hilbertian norm is given by
\begin{equation}
\left\Vert \left(  A,\delta\right)  \right\Vert_{\mathfrak{g}_{CM}}%
:=\sqrt{\left\Vert A\right\Vert_{H}^{2}+\left\Vert \delta\right\Vert
_{\mathbf{C}}^{2}}. \label{e.h3.6}%
\end{equation}
As was shown in \cite[Lemma 3.3]{DG07b}, these Banach space topologies on
$W\times\mathbf{C}$ and $H\times\mathbf{C}$ make $G$ and $G_{CM}$ into
topological groups.

\begin{nota}
[Linear differentials]\label{n.h3.4} Suppose
$f:G\rightarrow\mathbb{C}$,  is a Frech\'{e}t smooth function. For
$g\in G$ and $h, k\in\mathfrak{g}$ let
\[
f^{\prime}\left(  g\right)  h:=\partial_{h}f\left(  g\right)  =\frac{d}%
{dt}\Big|_{0}f\left(  g+th\right)
\]
and%
\[
f^{\prime\prime}\left(  g\right)  \left(  h\otimes k\right)
:=\partial_{h}\partial_{k}f\left(  g\right).
\]
(Here and in the sequel a prime on a symbol will be used to denote its
derivative or differential.)
\end{nota}

As $G$ itself is a vector space, the tangent space, $T_{g}G$,  to
$G$ at $g$ is naturally isomorphic to $G$. Indeed, if $v, g\in G$,
then we may define a tangent vector $v_{g}\in T_{g}G$ by
$v_{g}f=f^{\prime}\left(  g\right)v$ for all Frech\'{e}t smooth
functions $f:G\rightarrow\mathbb{C}$. We will identify
$\mathfrak{g}$ with $T_{\mathbf{e}}G$ and $\mathfrak{g}_{CM}$ with
$T_{\mathbf{e}}G_{CM}$. Recall that as sets $\mathfrak{g}=G$ and
$\mathfrak{g}_{CM}=G_{CM}$. For $g\in G$,  let
\thinspace$l_{g}:G\rightarrow G$ be the left translation by $g$. For
$h\in\mathfrak{g}$,  let $\tilde{h}$ be the \textbf{left invariant
vector field} on $G$ such that $\tilde{h}\left(  g\right)  =h$ when
$g=\mathbf{e}$. More precisely, if $\sigma\left(  t\right)  \in G$
is any smooth curve such that $\sigma\left(  0\right)  =\mathbf{e}$
and $\dot{\sigma }\left(  0\right) =h$ (e.g. $\sigma\left( t\right)
=th)$,  then
\begin{equation}
\tilde{h}\left(  g\right)=,
l_{g\ast}h:=\frac{d}{dt}\left|_{0}\right.g\cdot
\sigma\left(  t\right). \label{e.h3.7}%
\end{equation}
As usual, we view $\tilde{h}$ as a first order differential operator
acting on smooth functions, $f:G\rightarrow\mathbb{C}$,  by
\begin{equation}
\left(  \tilde{h}f\right)  \left(  g\right)  =\frac{d}{dt}\Big|_{0}f\left(
g\cdot\sigma\left(  t\right)  \right). \label{e.h3.8}%
\end{equation}
The proof of the following easy proposition may be found in \cite[Proposition
3.7]{DG07b}.

\begin{prop}
\label{p.h3.5}Let $f:G\rightarrow\mathbb{C}$ be a smooth function,
$h=(A,a)\in\mathfrak{g}$ and $g=\left(  w, c\right)  \in G$. Then%
\begin{equation}
\widetilde{h}\left(  g\right)  :=, l_{g\ast}h=\left(  A, a+\frac{1}{2}%
\omega\left(  w, A\right)  \right)  \text{ for any }g=\left( w,
c\right)  \in
G \label{e.h3.9}%
\end{equation}
and, in particular,
\begin{equation}
\widetilde{\left(  A, a\right)  }f\left(  g\right) =f^{\prime}\left(
g\right)  \left(  A, a+\frac{1}{2}\omega\left( w, A\right)  \right).
\label{e.h3.10}%
\end{equation}
If $h,k\in\mathfrak{g}$,  then%
\begin{equation}
\left(  \tilde{h}\tilde{k}f-\tilde{k}\tilde{h}f\right)  =\widetilde{\left[
h,k\right]  }f. \label{e.h3.11}%
\end{equation}
The one parameter group in $G$,  $e^{th}$,  determined by $h=\left(
A, a\right) \in\mathfrak{g}$,  is given by $e^{th}=th=t\left( A,
\delta\right)$. \end{prop}

\section{Brownian motion and heat kernel measures\label{s.h4}}

This section will closely follow \cite[Section 4]{DG07b} except for
the introduction of a certain factor of $1/2$ into the formalism
which will simplify later formulas. Let $\left\{  b\left(  t\right)
=\left(  B\left( t\right), B_{0}\left(  t\right)  \right)  \right\}
_{t\geqslant0}$ be a Brownian motion on
$\mathfrak{g}=W\times\mathbf{C}$ with the variance determined by
\begin{equation}
\mathbb{E}\left[  \operatorname{Re}\left\langle b\left(  s\right) ,
h\right\rangle
_{\mathfrak{g}_{CM}}\cdot\operatorname{Re}\left\langle
b\left(  t\right), k\right\rangle_{\mathfrak{g}_{CM}}\right]  =\frac{1}%
{2}\operatorname{Re}\left\langle h, k\right\rangle
_{\mathfrak{g}_{CM}}s\wedge
t \label{e.h4.1}%
\end{equation}
for all $s, t\in\lbrack0,\infty)$,  $h=\left(  A, a\right)$,  and $k\mathbf{:=}%
\left(  C, c\right)$, where $A,C\in H_{\ast}$ and
$a,c\in\mathbf{C}$. (Recall the definition of $H_{\ast}$ from Eq.
(\ref{e.h2.15}).)

\begin{df}
\label{d.h4.1} The associated \textbf{Brownian motion} on $G$
starting at $\mathbf{e}=\left(  0, 0\right)  \in G$ is defined to be
the process
\begin{equation}
g\left(  t\right)  =\left(  B\left(  t\right), B_{0}\left(  t\right)
+\frac{1}{2}\int_{0}^{t}\omega\left(  B\left(  \tau\right), dB\left(
\tau\right)  \right)  \right). \label{e.h4.2}%
\end{equation}
More generally, if $h\in G$,  we let $g_{h}\left(  t\right) :=h\cdot
g\left( t\right)$, the Brownian motion on $G$ starting at $h$.
\end{df}

\begin{df}
\label{d.h4.2} Let $\mathcal{B}_{G}$ be the Borel $\sigma$--algebra
on $G$ and for any $T>0$,  let
$\nu_{T}:\mathcal{B}_{G}\rightarrow\left[  0, 1\right]$ be the
distribution of $g\left(  T\right)$. We will call $\nu_{T}$ the
\textbf{heat kernel measure} on $G$. \end{df}

To be more explicit, the measure $\nu_{T}$ is the unique measure on
$\left( G,\mathcal{B}_{G}\right)$ such that
\begin{equation}
\nu_{T}\left(  f\right):=\int_{G}fd\nu_{T}=\mathbb{E}\left[ f\left(
g\left(  T\right)  \right)  \right]  \label{e.h4.3}%
\end{equation}
for all bounded measurable functions $f:G\rightarrow\mathbb{C}$. Our
next goal is to describe the generator of the process $\left\{
g_{h}\left( t\right)  \right\}_{t\geqslant0}$.
\begin{df}
\label{d.h4.3} A function $f:G\rightarrow\mathbb{C}$ is said to be a
\textbf{cylinder function} if it may be written as $f=F\circ\pi_{P}$
for some $P\in\operatorname*{Proj}\left(  W\right)$ and some
function $F:G_{P} \rightarrow\mathbb{C}$, where $G_{P}$ is defined
as in Notation \ref{n.h1.1}. We say that $f$ is a
\textbf{holomorphic }(\textbf{smooth})\textbf{ }cylinder function if
$F:G_{P}\rightarrow \mathbb{C}$ is holomorphic (smooth). We will
denote the space of holomorphic (analytic) cylinder functions by
$\mathcal{A}$. \end{df}

\begin{prop}
[Generator of $g_{h}$]\label{p.h4.4}If $f:G\rightarrow\mathbb{C}$ is a smooth
cylinder function, let%
\begin{equation}
Lf:=\sum_{j=1}^{\infty}\left[  \widetilde{\left(  e_{j},0\right)  }%
^{2}+\widetilde{\left(  ie_{j},0\right)  }^{2}\right]
f+\sum_{j=1}^{d}\left[ \widetilde{\left(  0, f_{j}\right)
}^{2}+\widetilde{\left(  0, if_{j}\right)
}^{2}\right]  f, \label{e.h4.4}%
\end{equation}
where $\left\{  e_{j}\right\}_{j=1}^{\infty}$ and $\left\{
f_{j}\right\}_{j=1}^{d}$ are complex orthonormal bases for $\left(
H, \left\langle \cdot, \cdot\right\rangle_{H}\right)$ and $\left(
\mathbf{C},\left\langle \cdot,\cdot\right\rangle
_{\mathbf{C}}\right)$ respectively. Then $Lf$ is well defined, i.e.
the sums in Eq. (\ref{e.h4.4}) are convergent and independent of the
choice of bases. Moreover, for all $h\in G$,  $\frac{1}{4}L$ is the
generator for $\left\{  g_{h}\left(  t\right)  \right\}_{t\geqslant
0}$. More precisely,
\begin{equation}
M_{t}^{f}:=f\left(  g_{h}\left(  t\right)  \right)  -\frac{1}{4}\int_{0}%
^{t}Lf\left(  g_{h}\left(  \tau\right)  \right)  d\tau\label{e.h4.5}%
\end{equation}
is a local martingale for any smooth cylinder function,
$f:G\rightarrow \mathbb{C}$. \end{prop}

\begin{proof}
After bearing in mind the factor of $1/2$ used in defining the
Brownian motion $b\left( t \right)$  in Eq. (\ref{e.h4.1}), this
proposition becomes a direct consequence of Proposition 3.29 and
Theorem 4.4 of \cite{DG07b}. Indeed, the Brownian motions in this
paper are equal in distribution to the Brownian motions used in
\cite{DG07b} after making the time change, $t\rightarrow t/2$. It is
this time change that is responsible for the $1/4$ factor (rather
than $1/2)$ in Eq. (\ref{e.h4.5}).
\end{proof}

\subsection{Heat kernel quasi-invariance properties\label{s.h4.1}}

In this subsection we are going to recall one of the key theorems
from \cite{DG07b}. We first need a little more notation.

Let $C_{CM}^{1}$ denote the collection of $C^{1}$-paths, $g:\left[
0,1\right]  \rightarrow G_{CM}$. The length of $g$ is defined as
\begin{equation}
\ell_{G_{CM}}\left(  g\right)  =\int_{0}^{1}\left\Vert
\,l_{g^{-1}\left( s\right)  \ast}g^{\prime}\left(  s\right)
\right\Vert_{\mathfrak{g}_{CM}}ds.
\label{e.h4.6}%
\end{equation}
As usual, the Riemannian distance between $x,y\in G_{CM}$ is defined as%
\begin{equation}
d_{G_{CM}}(x,y)=\inf\left\{  \ell_{G_{CM}}\left(  g\right)  :g\in C_{CM}%
^{1}~\ni~g\left(  0\right)  =x\text{ and }g\left(  1\right)
=y\right\}.
\label{e.h4.7}%
\end{equation}
Let us also recall the definition of $k\left(  \omega\right)$ from
\cite[Eq. 7.6]{DG07b};
\begin{align}
k\left(  \omega\right)   &  =-\frac{1}{2}\sup_{\left\Vert
A\right\Vert_{H_{\operatorname{Re}}}=1}\left\Vert \omega\left(
\cdot, A\right)
\right\Vert_{H_{\operatorname{Re}}^{\ast}\otimes\mathbf{C}_{\operatorname{Re}%
}}^{2}\nonumber\\
&  =-\sup_{\left\Vert A\right\Vert_{H}=1}\left\Vert \omega\left(
\cdot, A\right)  \right\Vert
_{H^{\ast}\otimes\mathbf{C}}^{2}\geqslant
-\left\Vert \omega\right\Vert_{H^{\ast}\otimes H^{\ast}\otimes C}^{2}%
>-\infty, \label{e.h4.8}%
\end{align}
wherein we have used \cite[Lemma 3.17]{DG07b} in the second equality. It is
known by Fernique's or Skhorohod's theorem that $\left\Vert \omega\right\Vert
_{2}^{2}=\left\Vert \omega\right\Vert_{H^{\ast}\otimes H^{\ast}\otimes C}%
^{2}<\infty$,  see \cite[Proposition 3.14]{DG07b} for details.

\begin{thm}
\label{t.h4.5}For all $h\in G_{CM}$ and $T>0$,  the measures,
$\nu_{T}\circ $\thinspace$l_{h}^{-1}$ and $\nu_{T}\circ r_{h}^{-1}$,
are absolutely
continuous relative to $\nu_{T}$. Let $Z_{h}^{l}:=\frac{d\left(  \nu_{T}%
\circ\, l_{h}^{-1}\right)  }{d\nu_{T}}$ and
$Z_{h}^{r}:=\frac{d\left(  \nu_{T}\circ r_{h}^{-1}\right)
}{d\nu_{T}}$ be the respective Randon-Nikodym
derivatives, $k\left(  \omega\right)$ is given in Eq. (\ref{e.h4.8}), and%
\[
c\left(  t\right)  :=\frac{t}{e^{t}-1} \text{ for any }t\in\mathbb{R}%
\]
with the convention that $c\left(  0\right)  =1$. Then for all
$1\leqslant p<\infty$,  $Z_{h}^{l}$ and $Z_{h}^{r}\ $ are both in
\thinspace$L^{p}\left( \nu_{T}\right)$ and satisfy the estimate
\begin{equation}
\left\Vert Z_{h}^{\ast}\right\Vert_{L^{p}\left(  \nu_{T}\right)  }%
\leqslant\exp\left(  \frac{c\left(  k\left(  \omega\right)
T/2\right) \left(  p-1\right)  }{T}d_{G_{CM}}^{2}\left(  \mathbf{e},
h\right)  \right),
\label{e.h4.9}%
\end{equation}
where $\ast=l$ or $\ast=r$. \end{thm}

\begin{proof}
This is \cite[Theorem 8.1]{DG07b} (also see \cite[Corollary
7.3]{DG07b}) with the modification that $T$ should be replaced by
$T/2$. This is again due to the fact that the Brownian motions in
this paper are equal in distribution to those in \cite{DG07b} after
making the time change, $t\rightarrow t/2$. \end{proof}

It might be enlightening to note here that we call $G_{CM}$ the Cameron-Martin
subgroup not only because it is constructed from the Cameron-Martin subspace,
$H$, but also because it has properties similar to $H$. In particular, the
following statement holds.

\begin{prop}
\label{p.h4.6}The heat kernel measure does not charge $G_{CM}$, i.e.
$\nu_{T}\left(  G_{CM}\right)  =0$. \end{prop}

\begin{proof}
Note that for a bounded measurable function $f: W\times C
\to\mathbb{C}$ that depends only on the the first component in
$W\times C$, that is, $f \left(  w, c\right)  =f\left(  w\right)$ we
have

\[
\int_{G} f\left(  w\right)  d\nu_{T}\left(  w, c \right)  =
\mathbb{E} \left[ f \left(  B \left(  T\right)  \right)  \right]
=\int_{W}f\left(  w\right) d\mu_{T}\left(  w\right).
\]
Note that for the projection $\pi: W \times C \to W$, $\pi\left(  w, c
\right)  =w$ we have $\pi_{\ast}\nu_{T}=\mu_{T}$ and therefore%

\[
\nu_{T}\left(  G_{CM} \right)  =\nu_{T}\left(  \pi^{-1}\left(  H
\right) \right)  =\pi_{\ast}\nu_{T}\left(  H \right)  =\mu_{T}\left(
H \right) =0.
\]

\end{proof}

For later purposes, we would like to introduce the heat operator,
$S_{T}:=e^{TL/4}$,  acting on $L^{p}\left(  G, \nu_{T}\right)$. To
motivate our definition, suppose $f:G\rightarrow\mathbb{C}$ is a
smooth cylinder function and suppose we can make sense of $u\left(
t,y\right)  =\left( e^{\left(  T-t\right)  L/4}f\right)  \left(
y\right)$. Then working formally, by It\^{o}'s formula, Eq.
(\ref{e.h4.5}), and the left invariance of $L$,  we expect $u\left(
t, hg\left(  t\right)  \right)$ to be a martingale for $0\leqslant
t\leqslant T$ and in particular,
\begin{equation}
\mathbb{E}\left[  f\left(  hg\left(  T\right)  \right)  \right]
=\mathbb{E}\left[  u\left(  T, hg\left(  T\right)  \right)  \right]
=\mathbb{E}\left[  u\left(  0, hg\left(  0\right)  \right)  \right]
=\left(
e^{TL/4}f\right)  \left(  h\right). \label{e.h4.10}%
\end{equation}

\begin{df}
\label{d.h4.7}For $T>0$,  $p\in(1,\infty]$,  and $f\in L^{p}\left(
G,\nu_{T}\right)$,  let $S_{T}f:G_{CM}\rightarrow\mathbb{C}$ be
defined by
\begin{equation}
\left(  S_{T}f\right)  \left(  h\right)  =\int_{G}f\left(  h\cdot g\right)
d\nu_{T}\left(  g\right)  =\mathbb{E}\left[  f\left(  hg\left(  T\right)
\right)  \right]. \label{e.h4.11}%
\end{equation}

\end{df}

The following result is a simple corollary of Theorem \ref{t.h4.5}
and H\"{o}lder's inequality along with the observation that
$p^{\prime}-1=\left( p-1\right)^{-1}$, where $p^{\prime}$ is the
conjugate exponent to $p\in(1,\infty]$.
\begin{cor}
\label{c.h4.8}If $p>1$,  $T>0$,  $f\in L^{p}\left(  G,\nu_{T}\right)
$,  $h\in G_{CM}$,  and
\begin{equation}
Z_{h}^{l}\in L^{\infty-}\left(  \nu_{T}\right)  :=\cap_{1\leqslant q<\infty
}L^{q}\left(  \nu_{T}\right)  \label{e.h4.12}%
\end{equation}
is as in Theorem \ref{t.h4.5}, then $S_{T}f$ is well defined and may be
computed as
\begin{equation}
\left(  S_{T}f\right)  \left(  h\right)  =
\int_{G}f\left(  g\right)  Z_{h}^{l}\left(  g\right)  d\nu_{T}\left(  g\right). \label{e.h4.13}%
\end{equation}
Moreover, we have the following pointwise \textquotedblleft
Gaussian\textquotedblright\ bounds
\begin{equation}
\left\vert \left(  S_{T}f\right)  \left(  h\right)  \right\vert
\leqslant \left\Vert f\right\Vert_{L^{p}\left(  \nu_{T}\right)
}\exp\left( \frac{c\left(  k\left(  \omega\right)  T/2\right)
}{T\left(  p-1\right)
}d_{G_{CM}}^{2}\left(  \mathbf{e}, h\right)  \right). \label{e.h4.14}%
\end{equation}

\end{cor}

We will see later that when $f$ is \textquotedblleft
holomorphic\textquotedblright\ and $p=2,$\ the above estimate in Eq.
(\ref{e.h4.14}) may be improved to%
\begin{equation}
\left\vert \left(  S_{T}f\right)  \left(  h\right)  \right\vert
\leqslant \left\Vert f\right\Vert_{L^{2}\left(  \nu_{T}\right)
}\exp\left(  \frac {1}{2T}d_{G_{CM}}^{2}\left(  \mathbf{e}, h\right)
\right)  \text{ for any } h
\in G_{CM}. \label{e.h4.15}%
\end{equation}
This bound is a variant of Bargmann's pointwise bounds (see
\cite[Eq. (1.7)]{Bargmann61} and \cite[Eq. (5.4)]{Driver1997c}).

\begin{lem}
\label{l.h4.9}Let $T>0$ and suppose that $f:G\rightarrow\mathbb{C}$
is a continuous and in $L^{p}\left(  \nu_{T}\right)  $ for some
$p>1$. Then $S_{T}f:G_{CM}\rightarrow\mathbb{C}$ is continuous.
\end{lem}

\begin{proof}
For $q\in\left(  1, p\right)$ and $h\in G_{CM}$ we have by
H\"{o}lder's
inequality and Theorem \ref{t.h4.5} that%
\begin{align}
\mathbb{E}\left\vert f\left(  hg\left(  T\right)  \right)
\right\vert ^{q} &  =v_{T}\left(  \left\vert f\right\vert
^{q}Z_{h}^{l}\right)  \leqslant \left\Vert f\right\Vert_{L^{p}\left(
\nu_{T}\right)  }^{q/p}\cdot\left\Vert Z_{h_{n}}^{l}\right\Vert
_{L^{\frac{p}{p-q}}\left(  \nu_{T}\right)
}\nonumber\\
&  \leqslant\left\Vert f\right\Vert_{L^{p}\left(  \nu_{T}\right)  }^{q/p}%
\exp\left(  \frac{c\left(  k\left(  \omega\right)  T/2\right)  q}{T\left(
p-q\right)  }d_{G_{CM}}^{2}\left(  \mathbf{e},h\right)  \right)
\label{e.h4.16}%
\end{align}
Hence if $\left\{  h_{n}\right\}_{n=1}^{\infty}\subset G_{CM}$ is a
sequence converging to $h\in G_{CM}$, it follows that
\begin{equation}
\sup_{n}\mathbb{E}\left\vert f\left(  h_{n}g\left(  T\right) \right)
\right\vert ^{q}\leqslant\left\Vert f\right\Vert_{L^{p}\left(
\nu_{T}\right)  }^{q/p}\exp\left( \frac{c\left(  k\left(
\omega\right) T/2\right)  q}{T\left( p-q\right)
}\sup_{n}d_{G_{CM}}^{2}\left( \mathbf{e}, h_{n}\right)  \right)
<\infty\label{e.h4.17},
\end{equation}
which implies that $\left\{  f\left(  h_{n}g\left(  T\right) \right)
\right\}_{n=1}^{\infty}$ is uniformly integrable. Since  by
continuity of $f$,  $\lim_{n\rightarrow\infty}f\left( h_{n}g\left(
T\right)  \right) =f\left(  hg\left(  T\right) \right)$,  we may
pass to the limit under the expectation to find
\[
\lim_{n\rightarrow\infty}S_{T}f\left(  h_{n}\right)
=\lim_{n\rightarrow \infty}\mathbb{E}f\left(  h_{n}g\left(  T\right)
\right)  =\mathbb{E}\left[ f\left(  hg\left(  T\right)  \right)
\right]  =S_{T}f\left(  h\right).
\]

\end{proof}

\subsection{Finite dimensional approximations\label{s.h4.2}}

\begin{nota}
\label{d.h4.10}For each $P\in\operatorname*{Proj}\left(  W\right) $,
let $g_{P}\left(  t\right)$ denote the $G_{P}$--valued Brownian
motion defined by
\begin{equation}
g_{P}\left(  t\right)  =\left(  PB\left(  t\right),B_{0}\left(
t\right) +\frac{1}{2}\int_{0}^{t}\omega\left(  PB\left(  \tau\right)
,dPB\left(
\tau\right)  \right)  \right). \label{e.h4.18}%
\end{equation}
Also, for any $t>0$,  let $\nu_{t}^{P}:=\operatorname*{Law}\left(
g_{P}\left( t\right)  \right)$ be the corresponding heat kernel
measure on $G_{P}$. \end{nota}

The following Theorem is a restatement of \cite[Theorem 4.16]{DG07b}.

\begin{thm}
[Integrated heat kernel bounds]\label{t.h4.11}Suppose that $\rho
^{2}:G\rightarrow\lbrack0,\infty)$ be defined as
\begin{equation}
\rho^{2}\left(  w,c\right)  :=\left\Vert
w\right\Vert_{W}^{2}+\left\Vert
c\right\Vert_{\mathbf{C}}. \label{e.h4.19}%
\end{equation}
Then there exists a $\delta>0$ such that for all
$\varepsilon\in\left( 0, \delta\right)$ and $T>0$
\begin{equation}
\sup_{P\in\operatorname*{Proj}\left(  W\right)  }\mathbb{E}\left[
e^{\frac{\varepsilon}{T}\rho^{2}\left(  g_{P}\left(  T\right)  \right)
}\right]  <\infty\text{ and }\int_{G}e^{\frac{\varepsilon}{T}\rho^{2}\left(
g\right)  }d\nu_{T}\left(  g\right)  <\infty. \label{e.h4.20}%
\end{equation}

\end{thm}

\begin{prop}
\label{p.h4.12}Let $P_{n}\in\operatorname*{Proj}\left(  W\right)$
such that
$P_{n}|_{H}\uparrow I_{H}$ on $H$ and let $g_{n}\left(  T\right)  :=g_{P_{n}%
}\left(  T\right)$. Further suppose that $\delta>0$ is as in Theorem
\ref{t.h4.11}, $p\in\lbrack1,\infty)$,  and
$f:G\rightarrow\mathbb{C}$ is a continuous function such that
\begin{equation}
\left\vert f\left(  g\right)  \right\vert \leqslant Ce^{\varepsilon\rho
^{2}\left(  g\right)  /\left(  pT\right)  }\text{ for all }g\in G
\label{e.h4.21}%
\end{equation}
for some $\varepsilon\in\left(  0,\delta\right)$. Then $f\in
L^{p}\left(
\nu_{T}\right)$ and for all $h\in G$ we have%
\begin{equation}
\lim_{n\rightarrow\infty}\mathbb{E}\left\vert f\left(  hg\left(  T\right)
\right)  -f\left(  hg_{n}\left(  T\right)  \right)  \right\vert ^{p}=0,
\label{e.h4.22}%
\end{equation}
and%
\begin{equation}
\lim_{n\rightarrow\infty}\mathbb{E}\left\vert f\left(  g\left(  T\right)
h\right)  -f\left(  g_{n}\left(  T\right)  h\right)  \right\vert ^{p}=0.
\label{e.h4.23}%
\end{equation}

\end{prop}

\begin{proof}
If $q\in\left(  p,\infty\right)$ is sufficiently close to $p$ so
that
$qp^{-1}\varepsilon<\delta$,  then%
\[
\sup_{n}\mathbb{E}\left\vert f\left(  g_{n}\left(  T\right)  \right)
\right\vert ^{q}\leqslant C^{q}\sup_{n}\mathbb{E}\left[  e^{p^{-1}%
q\varepsilon\rho^{2}\left(  g\right)  /T}\right]
\]
which is finite by Theorem \ref{t.h4.11}. This shows that $\left\{
\left\vert f\left(  g_{n}\left(  T\right)  \right)  \right\vert
^{p}\right\}_{n=1}^{\infty}$ is uniformly integrable. As a
consequence of \cite[Lemma 4.7]{DG07b} and the continuity of $f$, we
also know that $f\left( g_{n}\left(  T\right)  \right) \rightarrow
f\left(  g\left(  T\right) \right)$ in probability as
$n\rightarrow\infty$. Thus we have shown Eqs. (\ref{e.h4.22}) and
(\ref{e.h4.23}) hold when $h=\mathbf{e}=0$. Now suppose that
$g=\left(  w,c\right)$ and $h=\left(  A,a\right)$ are in
$G$. Then for all $\alpha>0,$%
\begin{align}
\rho^{2}\left(  gh\right)   &  =\left\Vert
w+A\right\Vert_{W}^{2}+\left\Vert
a+c+\frac{1}{2}\omega\left(  w,A\right)  \right\Vert_{\mathbf{C}}\nonumber\\
&  \leqslant\left\Vert w\right\Vert_{W}^{2}+\left\Vert A\right\Vert_{W}%
^{2}+2\left\Vert A\right\Vert_{W}\left\Vert w\right\Vert
_{W}+\left\Vert
a\right\Vert_{\mathbf{C}}+\left\Vert c\right\Vert_{\mathbf{C}}+\frac{1}%
{2}\left\Vert \omega\left(  w,A\right)  \right\Vert_{\mathbf{C}}\nonumber\\
&  \leqslant\rho^{2}\left(  g\right)  +\rho^{2}\left(  h\right)  +C\left\Vert
A\right\Vert_{W}\left\Vert w\right\Vert_{W}\label{e.h4.24}\\
&  \leqslant\rho^{2}\left(  g\right)  +\rho^{2}\left(  h\right)  +\frac{C}%
{2}\left[  \alpha^{-1}\left\Vert
A\right\Vert_{W}^{2}+\alpha\left\Vert
w\right\Vert_{W}^{2}\right] \nonumber\\
&  \leqslant\left(  1+\frac{C\alpha}{2}\right)  \rho^{2}\left(
g\right) +\left(  1+\frac{C}{2\alpha}\right)  \rho^{2}\left(
h\right),\nonumber
\end{align}
where $C:=\left(  2+\frac{1}{2}\left\Vert \omega\right\Vert
_{0}\right)$. As Eq. (\ref{e.h4.24}) is invariant under
interchanging $g$ and $h$ the same bound also hold for
$\rho^{2}\left(  hg\right)$. By choosing $\alpha>0$ sufficiently
small so that $\left(  1+\frac{C\alpha}{2}\right)  \varepsilon
<\delta$,  we see that $g\rightarrow f\left(  gh\right)$ and
$g\rightarrow f\left(  hg\right)$ satisfy the same type of bound as
in Eq. (\ref{e.h4.21}) for $\,g\rightarrow f\left(  g\right)$.
Therefore, by the first paragraph, we have now verified Eqs.
(\ref{e.h4.22}) and (\ref{e.h4.23}) hold for any $h\in G$.
\end{proof}

\section{Holomorphic functions on $G$ and $G_{CM}$\label{s.h5}}

We will begin with a short summary of the results about holomorphic functions
on Banach spaces that will be needed in this paper.

\subsection{Holomorphic functions on Banach spaces\label{s.h5.1}}

Let $X$ and $Y$ be two complex Banach space and for $a\in X$ and $\delta>0$
let
\[
B_{X}\left(  a,\delta\right)  :=\left\{  x\in X:\left\Vert x-a\right\Vert
_{X}<\delta\right\}
\]
be the open ball in $X$ with center $a$ and radius $\delta$.
\begin{df}
[{Hille and Phillips \cite[Definition 3.17.2, p. 112.]{HP74}}]\label{d.h5.1}%
Let $\mathcal{D}$ be an open subset of $X$. A function $u:\mathcal{D}%
\rightarrow Y\ $is said to be \textbf{holomorphic (or analytic)} if the
following two conditions hold.

\begin{enumerate}
\item $u$ is locally bounded, namely for all $a\in\mathcal{D}$ there exists
an $r_{a}>0$ such that
\[
M_{a}:=\sup\left\{  \left\Vert u\left(  x\right)
\right\Vert_{Y}:x\in B_{X}\left(  a,r_{a}\right)  \right\}  <\infty.
\]

\item The function $u$ is complex G\^{a}teaux differentiable on $\mathcal{D}$,  i.e. for each $a\in\mathcal{D}$ and $h\in X$,  the function
$\lambda\rightarrow u\left(  a+\lambda h\right)$ is complex
differentiable at $\lambda =0\in\mathbb{C}$. \end{enumerate}

(Holomorphic and analytic will be considered to be synonymous terms for the
purposes of this paper.)
\end{df}

The next theorem gathers together a number of basic properties of
holomorphic functions which may be found in \cite{HP74}. (Also see
\cite{Herve89}.) One of the key ingredients to all of these results
is Hartog's theorem, see \cite[Theorem 3.15.1]{HP74}.


\begin{thm}
\label{t.h5.2}If $u:\mathcal{D}\rightarrow Y$ is holomorphic, then
there exists a function
$u^{\prime}:\mathcal{D}\rightarrow\operatorname{Hom}\left(
X,Y\right)$, the space of bounded \textbf{complex} linear operators
from $X$ to $Y$, satisfying

\begin{enumerate}
\item If $a\in\mathcal{D}$,  $x\in B_{X}\left(  a,r_{a}/2\right)$,  and $h\in
B_{X}\left(  0,r_{a}/2\right)$,  then
\begin{equation}
\left\Vert u\left(  x+h\right)  -u\left(  x\right)  -u^{\prime}\left(
x\right)  h\right\Vert_{Y}\leqslant\frac{4M_{a}}{r_{a}\left(  r_{a}%
-2\left\Vert h\right\Vert_{X}\right)  }\left\Vert h\right\Vert
_{X}^{2}.
\label{e.h5.1}%
\end{equation}
In particular, $u$ is continuous and Frech\'{e}t differentiable on
$\mathcal{D}$.
\item The function $u^{\prime}:\mathcal{D}\rightarrow\operatorname{Hom}\left(
X,Y\right)$ is holomorphic.
\end{enumerate}
\end{thm}

\begin{rem}
\label{r.h5.3}By applying Theorem \ref{t.h5.2} repeatedly, it
follows that any holomorphic function, $u:\mathcal{D}\rightarrow Y$
is Frech\'{e}t differentiable to all orders and each of the
Frech\'{e}t differentials are again holomorphic functions on
$\mathcal{D}$. \end{rem}

\begin{proof}
By \cite[Theorem 26.3.2 on p. 766.]{HP74}, for each
$a\in\mathcal{D}$ there is a linear operator, $u^{\prime}\left(
a\right): X\rightarrow Y$ such that $du\left(  a+\lambda h\right)
/d\lambda|_{\lambda=0}=u^{\prime}\left( a\right)  h$. The Cauchy
estimate in Theorem 3.16.3 (with $n=1$) of \cite{HP74} implies that
if $a\in\mathcal{D}$,  $x\in B_{X}\left( a,r_{a}/2\right)$ and $h\in
B_{X}\left(  0,r_{a}/2\right)$ (so that $x+h\in B_{X}\left( a,
r_{a}\right)  )$,  then $\left\Vert u^{\prime}\left( x\right)
h\right\Vert_{Y}\leqslant M_{a}$. It follows from this estimate
that%
\begin{equation}
\sup\left\{  \left\Vert u^{\prime}\left(  x\right)  \right\Vert
_{\operatorname{Hom}\left(  X,Y\right)  }:x\in B_{X}\left(  a,r_{a}/2\right)
\right\}  \leqslant2M_{a}/r_{a}. \label{e.h5.2}%
\end{equation}
and hence that
$u^{\prime}:\mathcal{D}\rightarrow\operatorname{Hom}\left(
X,Y\right)$ is a locally bounded function. The estimate in Eq.
(\ref{e.h5.1}) appears in the proof of the Theorem 3.17.1 in
\cite{HP74} which completes the proof of item 1.

To prove item 2. we must show $u^{\prime}$ is G\^{a}teaux
differentiable on $\mathcal{D}$. We will in fact show more, namely,
that $u^{\prime}$ is Frech\'{e}t differentiable on $\mathcal{D}$.
Given $h\in X$,  let $F_{h}:\mathcal{D}\rightarrow Y$ be defined by
$F_{h}\left(  x\right) :=u^{\prime}\left(  x\right)  h$. According
to \cite[Theorem 26.3.6]{HP74}, $F_{h}$ is holomorphic on
$\mathcal{D}$ as well. Moreover, if $a\in\mathcal{D}$ and $x\in
B\left(  a,r_{a}/2\right)$ we have by Eq. (\ref{e.h5.2}) that
\[
\left\Vert F_{h}\left(  x\right)
\right\Vert_{Y}\leqslant2M_{a}\left\Vert h\right\Vert_{X}/r_{a}.
\]
So applying the estimate in Eq. (\ref{e.h5.1}) to $F_{h}$,  we learn that%
\begin{equation}
\left\Vert F_{h}\left(  x+k\right)  -F_{h}\left(  x\right)
-F_{h}^{\prime }\left(  x\right)
k\right\Vert_{Y}\leqslant\frac{4\left(  2M_{a}\left\Vert
h\right\Vert_{X}/r_{a}\right)  }{\frac{r_{a}}{2}\left(  \frac{r_{a}}%
{2}-2\left\Vert k\right\Vert_{X}\right)  }\cdot\left\Vert
k\right\Vert
_{X}^{2} \label{e.h5.3}%
\end{equation}
for $x\in B\left(  a,r_{a}/4\right)$ and $\left\Vert k\right\Vert_{X}%
<r_{a}/4$, where
\[
F_{h}^{\prime}\left(  x\right)
k=\frac{d}{d\lambda}\left|_{0}F_{h}\left( x+\lambda k\right)\right.
=\frac{d}{d\lambda}|_{0}u^{\prime}\left( x+\lambda k\right)
h=:\left(  \delta^{2}u\right)  \left( x;h,k\right).
\]
Again by \cite[Theorem 26.3.6]{HP74}, for each fixed
$x\in\mathcal{D}$,  $\left(  \delta^{2}u\right)  \left(
x;h,k\right) $ is a continuous symmetric bilinear form in $\left(
h,k\right) \in X\times X$. Taking the supremum of Eq. (\ref{e.h5.3})
over those $h\in X$ with $\left\Vert h\right\Vert_{X}=1$,  we may conclude that%
\begin{align*}
\big\Vert u^{\prime}\left(  x+k\right)  -u^{\prime}\left(  x\right)   &
-\delta^{2}u\left(  x;\cdot,k\right)  \big\Vert_{\operatorname{Hom}\left(
X,Y\right)  }\\
&  =\sup_{\left\Vert h\right\Vert_{X}=1}\left\Vert F_{h}\left(
x+k\right)
-F_{h}\left(  x\right)  -F_{h}^{\prime}\left(  x\right)  k\right\Vert_{Y}\\
&  \leqslant\frac{4\left(  2M_{a}/r_{a}\right)
}{\frac{r_{a}}{2}\left( \frac{r_{a}}{2}-2\left\Vert
k\right\Vert_{X}\right)  } \left\Vert k\right\Vert_{X}^{2}.
\end{align*}
This estimate shows $u^{\prime}$ is Frech\'{e}t differentiable with
$u^{\prime\prime}\left(  x\right)  \in\operatorname{Hom}\left(
X,\operatorname{Hom}\left(  X,Y\right)  \right)$ being given by
$u^{\prime\prime}\left(  x\right)  k=\left(  \delta^{2}u\right)
\left( x;\cdot,k\right)  \in\operatorname{Hom}\left(  X,Y\right)$
for all $k\in X$ and $x\in\mathcal{D}$. \end{proof}

\subsection{Holomorphic functions on $G$ and $G_{CM}$\label{s.h5.2}}

For the purposes of this section, let $G_{0}=G$ and $\mathfrak{g}%
_{0}=\mathfrak{g}$ or $G_{0}=G_{CM}$ and
$\mathfrak{g}_{0}=\mathfrak{g}_{CM}$. Also for $g,h\in\mathfrak{g}$,
let (as usual) $ad_{g}h:=\left[  g,h\right]$.
\begin{lem}
\label{l.h5.4}For each $g\in G_{0}$,
\thinspace$l_{g}:G_{0}\rightarrow G_{0}$
is holomorphic in the $\left\Vert \cdot\right\Vert_{\mathfrak{g}_{0}}%
$--topology. Moreover, a function $u:G_{0}\rightarrow\mathbb{C}$ defined in a
neighborhood of $g\in G_{0}$ is G\^{a}teaux (Frech\'{e}t) differentiable at
$g$ iff $u\circ$\thinspace$l_{g}$ is G\^{a}teaux (Frech\'{e}t) differentiable
at $0$. In addition, if $u$ is Frech\'{e}t differentiable at $g$,  then%
\begin{equation}
\left(  u\circ\,l_{g}\right)^{\prime}\left(  0\right)
h=u^{\prime}\left(
g\right)  \left(  h+\frac{1}{2}\left[  g,h\right]  \right). \label{e.h5.4}%
\end{equation}
(See \cite[Theorem 5.7]{GrossMall} for an analogous result in the context of
path groups.)
\end{lem}

\begin{proof}
Since
\[
\,l_{g}\left(  h\right)  =gh=g+h+\frac{1}{2}\left[  g,h\right]  =g+\left(
Id_{\mathfrak{g}_{0}}+\frac{1}{2}ad_{g}\right)  h,
\]
it is easy to see that \thinspace$l_{g}$ is holomorphic and $l_{g}^{\prime}$
is the constant function equal to $Id_{\mathfrak{g}_{0}}+\frac{1}{2}ad_{g}%
\in\operatorname*{End}\left(  \mathfrak{g}_{0}\right)$. Using
$ad_{g}^{2}=0$ or the fact that $l_{g}^{-1}=l_{g^{-1}}$,  we see
that $l_{g}^{\prime}$ is invertible and that
\[
l_{g}^{\prime}{}^{-1}=\left(  Id_{\mathfrak{g}_{0}}+\frac{1}{2}ad_{g}\right)
^{-1}=Id_{\mathfrak{g}_{0}}-\frac{1}{2}ad_{g}.
\]
These observations along with the chain rule imply the Frech\'{e}t
differentiability statements of the lemma and the identity in Eq.
(\ref{e.h5.4}).

If $u$ is G\^{a}teaux differentiable at $g$, $h\in\mathfrak{g}_{0}$,
and $k:=h+\frac{1}{2}\left[  g,h\right]$,  then
\[
\frac{d}{d\lambda}|_{0}u\circ\,l_{g}\left(  \lambda h\right)  =\frac
{d}{d\lambda}|_{0}u\left(  g\cdot\left(  \lambda h\right)  \right)  =\frac
{d}{d\lambda}|_{0}u\left(  g+\lambda k\right)
\]
and the existence of $\frac{d}{d\lambda}|_{0}u\left(  g+\lambda
k\right)$ implies the existence of
$\frac{d}{d\lambda}|_{0}u\circ$\thinspace $l_{g}\left(  \lambda
h\right)$. Conversely, if $u\circ$\thinspace$l_{g}$ is G\^{a}teaux
differentiable at $0$,  $h\in\mathfrak{g}_{0}$,  and
\[
k:=h-\frac{1}{2}\left[  g,h\right]  =\left(  Id_{\mathfrak{g}_{0}}+\frac{1}%
{2}ad_{g}\right)^{-1}h,
\]
then
\[
\,l_{g}\left(  \lambda k\right)  =g+\lambda\left(  Id_{\mathfrak{g}_{0}}%
+\frac{1}{2}ad_{g}\right)  k=g+\lambda h.
\]
So the existence of $\frac{d}{d\lambda}|_{0}\left(  u\circ\,l_{g}\right)
\left(  \lambda k\right)$ implies the existence of $\frac{d}{d\lambda}%
|_{0}u\left(  g+\lambda h\right)$. \end{proof}

\begin{cor}
\label{c.h5.5}A function $u:G_{0}\rightarrow\mathbb{C}$ is
holomorphic iff it is locally bounded and $h\rightarrow u\left(
ge^{h}\right)  =u\left(  g\cdot h\right)$ is G\^{a}teaux
(Frech\'{e}t) differentiable at $0$ for all $g\in G_{0}$. Moreover,
if $u$ is holomorphic and $h\in\mathfrak{g}_{0}$,  then
\[
\left(  \tilde{h}u\right)  \left(  g\right)  =\frac{d}{d\lambda}|_{0}u\left(
ge^{\lambda h}\right)  =u^{\prime}\left(  g\right)  \left(  h+\left[
g,h\right]  \right)
\]
is holomorphic as well.
\end{cor}

\begin{nota}
\label{n.h5.6}The space of globally defined holomorphic functions on $G$ and
$G_{CM}$ will be denoted by $\mathcal{H}\left(  G\right)$ and $\mathcal{H}%
\left(  G_{CM}\right)$ respectively.
\end{nota}

Notice that the space $\mathcal{A}$ of holomorphic cylinder
functions as described in Definition \ref{d.h4.3} is contained in
$\mathcal{H}\left( G\right)$. Also observe that a simple induction
argument using Corollary
\ref{c.h5.5} allows us to conclude that $\tilde{h}_{1}\dots\tilde{h}_{n}%
u\in\mathcal{H}\left(  G_{0}\right)$ for all $u\in\mathcal{H}\left(
G_{0}\right)$ and $h_{1},\dots,h_{n}\in\mathfrak{g}_{0}$.
\begin{prop}
\label{p.h5.7}If $f\in\mathcal{H}\left(  G\right)$ and
$h\in\mathfrak{g}$,  then $\widetilde{ih}f=i\tilde{h}f$,
$\widetilde{ih}\bar{f}=-i\tilde{h}\bar {f}$,  \begin{align} \left[
\left(  \widetilde{ih}\right)^{2}+\tilde{h}^{2}\right]  f  &
=0,\text{ and}\label{e.h5.5}\\
\left(  \tilde{h}^{2}+\widetilde{ih}^{2}\right)  \left\vert f\right\vert ^{2}
&  =4\left\vert \tilde{h}f\right\vert ^{2}. \label{e.h5.6}%
\end{align}

\end{prop}

\begin{proof}
The first assertions are directly related to the definition of $f$
being holomorphic. Using the identity $\widetilde{ih}f=i\tilde{h}f$
twice implies Eq. (\ref{e.h5.5}). Eq.(\ref{e.h5.6}) is a consequence
of summing the following two identities
\[
\tilde{h}^{2}\left\vert f\right\vert ^{2}=\tilde{h}\left(  f\cdot\bar
{f}\right)  =\tilde{h}^{2}f\cdot\bar{f}+f\cdot\tilde{h}^{2}\bar{f}+2\tilde
{h}f\cdot\tilde{h}\bar{f}%
\]
and%
\begin{align*}
\widetilde{ih}^{2}\left\vert f\right\vert ^{2}  &  =\widetilde{ih}\left(
f\cdot\bar{f}\right)  =\widetilde{ih}^{2}f\cdot\bar{f}+f\cdot\widetilde
{ih}^{2}\bar{f}+2\widetilde{ih}f\cdot\widetilde{ih}\bar{f}\\
&  =-\widetilde{h}^{2}f\cdot\bar{f}-f\cdot\widetilde{h}^{2}\bar{f}%
+2\widetilde{h}f\cdot\widetilde{h}\bar{f},
\end{align*}
and using $\widetilde{h}\bar{f}=\overline{\widetilde{h}f}$.
\end{proof}

\begin{cor}
\label{c.h5.8}Let $L$ be as in Proposition \ref{p.h4.4}. Suppose that
$f:G\rightarrow\mathbb{C}$ is a holomorphic cylinder function (i.e.
$f\in\mathcal{A})$,  then $Lf=0$ and%
\begin{equation}
L\left\vert f\right\vert ^{2}=\sum_{h\in\Gamma}\left\vert \tilde
{h}f\right\vert ^{2}, \label{e.h5.7}%
\end{equation}
where $\Gamma$ is an orthonormal basis for $\mathfrak{g}_{CM}$ of the form
\begin{equation}
\Gamma=\Gamma_{e}\cup\Gamma_{f}=\left\{  \left(  e_{j},0\right)
\right\}_{j=1}^{\infty}\cup\left\{  \left(  0,f_{j}\right)
\right\}_{j=1}^{d}
\label{e.h5.8}%
\end{equation}
with $\left\{  e_{j}\right\}_{j=1}^{\infty}$ and $\left\{
f_{j}\right\}_{j=1}^{d}$ being complex orthonormal bases for $H$ and
$\mathbf{C}$ respectively.
\end{cor}

\begin{proof}
These assertions follow directly form Eqs. (\ref{e.h4.4}), (\ref{e.h5.5}), and
(\ref{e.h5.6}).
\end{proof}

Formally, if $f:G\rightarrow\mathbb{C}$ is a holomorphic function, then
$e^{TL/4}f=f$ and therefore we should expect $S_{T}f=f|_{G_{CM}}$ where
$S_{T}$ is defined in Definition \ref{d.h4.7}. Theorem \ref{t.h5.9} below is a
precise version of this heuristic.

\begin{thm}
\label{t.h5.9}Suppose $p\in\left(  1,\infty\right)$ and
$f:G\rightarrow
\mathbb{C}$ is a continuous function such that $f|_{G_{CM}}\in\mathcal{H}%
\left(  G_{CM}\right)$ and there exists
$P_{n}\in\operatorname*{Proj}\left( W\right)$ such that
$P_{n}|_{H}\uparrow I_{H}$, then
\begin{equation}
\left\Vert f\right\Vert_{L^{p}\left(  \nu_{T}\right)  }\leqslant\sup
_{n}\left\Vert f\right\Vert_{L^{p}\left(
G_{P_{n}},\nu_{T}^{P_{n}}\right)
}. \label{e.h5.9}%
\end{equation}
If we further assume that%
\begin{equation}
\sup_{n}\left\Vert f\right\Vert_{L^{p}\left(  G_{P_{n}},\nu_{T}^{P_{n}%
}\right)  }<\infty, \label{e.h5.10}%
\end{equation}
then $f\in L^{p}\left(  \nu_{T}\right)$,  $S_{T}f=f|_{G_{CM}}$, and
$f$ satisfies the Gaussian bounds
\begin{equation}
\left\vert f\left(  h\right)  \right\vert \leqslant\left\Vert f\right\Vert
_{L^{p}\left(  \nu_{T}\right)  }\exp\left(  \frac{c\left(  k\left(
\omega\right)  T/2\right)  }{T\left(  p-1\right)  }d_{G_{CM}}^{2}\left(
\mathbf{e},h\right)  \right)  \text{ for any } h\in G_{CM}. \label{e.h5.11}%
\end{equation}

\end{thm}

\begin{proof}
According to \cite[Lemma 4.7]{DG07b}, by passing to a subsequence if
necessary, we may assume that $g_{P_{n}}\left(  T\right) \rightarrow
g\left( T\right)$ almost surely. Hence an application of Fatou's
lemma implies Eq. (\ref{e.h5.9}). In particular, if we assume Eq.
(\ref{e.h5.10}) holds, then $f\in L^{p}\left( \nu_{T}\right)$ and so
$S_{T}f$ is well defined.

Now suppose that $P\in\operatorname*{Proj}\left(  W\right)$ and
$h\in G_{P}$. Working exactly as in the proof of Lemma \ref{l.h4.9},
we find for any
$q\in\left(  1,p\right)$ that%
\begin{equation}
\mathbb{E}\left\vert f\left(  hg_{P}\left(  T\right)  \right)
\right\vert ^{q}\leqslant\left\Vert f\right\Vert_{L^{p}\left(
G_{P},\nu_{T}^{P}\right) }^{q/p}\exp\left(  \frac{c\left(
k_{P}\left(  \omega\right)  T/2\right) q}{T\left(  q-p\right)
}d_{G_{P}}^{2}\left(  \mathbf{e},h\right)  \right),
\label{e.h5.12}%
\end{equation}
where $d_{G_{P}}\left(  \cdot,\cdot\right)$ is the Riemannian
distance on
$G_{P}$ and (see \cite[Eq. (\ref{e.h5.13})]{DG07b}),%
\begin{equation}
k_{P}\left(  \omega\right)  :=-\frac{1}{2}\sup\left\{  \left\Vert
\omega\left(  \cdot,A\right)  \right\Vert_{\left(  PH\right)^{\ast}%
\otimes\mathbf{C}}^{2}:A\in PH,\ \left\Vert A\right\Vert
_{PH}=1\right\}.
\label{e.h5.13}%
\end{equation}
Observe that $k_{P}\left(  \omega\right)  \geqslant k\left(
\omega\right)$ and therefore, as $c$ is a decreasing function,
$c\left(  k\left( \omega\right)  \right)  \geqslant c\left(
k_{P}\left(  \omega\right) \right)$. Let $m\in\mathbb{N}$ be given
and $h\in G_{P_{m}}$. Then for $n\geqslant m$ we
have from Eq. (\ref{e.h5.12}) that%
\begin{align*}
\mathbb{E}\left\vert f\left(  hg_{P_{n}}\left(  T\right)  \right)
\right\vert ^{q}  &  \leqslant\left\Vert f\right\Vert_{L^{p}\left(
G_{P_{n}},\nu_{T}^{P_{n}}\right)  }^{q/p}\exp\left(  \frac{c\left(
k_{P_{n}}\left( \omega\right)  T/2\right)  q}{T\left(  q-p\right)
}d_{G_{P_{n}}}^{2}\left(
\mathbf{e},h\right)  \right) \\
&  \leqslant\left\Vert f\right\Vert_{L^{p}\left(  G_{P_{n}},\nu_{T}^{P_{n}%
}\right)  }^{q/p}\exp\left(  \frac{c\left(  k\left(  \omega\right)
T/2\right)  q}{T\left(  q-p\right)  }d_{G_{P_{m}}}^{2}\left(  \mathbf{e}%
,h\right)  \right)
\end{align*}
wherein in the last inequality we have used $c\left(  k\left(
\omega\right) \right)  \geqslant c\left(  k_{P}\left(  \omega\right)
\right)$ and the fact that $d_{G_{P_{n}}}^{2}\left(
\mathbf{e},h\right)$ is decreasing in $n\geqslant m$. Hence it
follows that $\sup_{n\geqslant m}\mathbb{E}\left\vert f\left(
hg_{P_{n}}\left(  T\right)  \right)  \right\vert ^{q}<\infty$ and
thus that $\left\{  f\left(  hg_{P_{n}}\left(  T\right)  \right)
\right\}_{n\geqslant m}$ is uniformly integrable. Therefore,
\begin{equation}
S_{T}f\left(  h\right)  =\mathbb{E}f\left(  hg\left(  T\right)  \right)
=\lim_{n\rightarrow\infty}\mathbb{E}f\left(  hg_{P_{n}}\left(  T\right)
\right)  =\lim_{n\rightarrow\infty}\int_{G_{P_{n}}}f\left(  hx\right)
d\nu_{T}^{P_{n}}\left(  x\right). \label{e.h5.14}%
\end{equation}

On the other hand by \cite[Lemma 4.8]{DG07b} (with $T$ replaced by
$T/2$ because of our normalization in Eq. (\ref{e.h4.1})),
$\nu_{T}^{P_{n}}$ is the heat kernel measure on $G_{P_{n}}$ based at
$\mathbf{e}\in G_{P_{n}}$,  i.e. $\nu_{T}^{P_{n}}\left(  dx\right)
=p_{T/2}^{P_{n}}\left(  e, x\right) dx$, where $dx$ is the
Riemannian volume measure (equal to a Haar measure) on $G_{P_{n}}$
and $p_{T}^{P_{n}}\left(  x,y\right)$ is the heat kernel on
$G_{P_{n}}$. Since $f|_{G_{P_{n}}}$ is holomorphic, the previous
observations allow us to apply \cite[Proposition 1.8]{DG07a} to
conclude that
\begin{equation}
\int_{G_{P_{n}}}f\left(  hx\right)  d\nu_{T}^{P_{n}}\left(  x\right)
=f\left(  \mathbf{e}\right)  \text{ for all }n\geqslant m. \label{e.h5.15}%
\end{equation}
As $m\in\mathbb{N}$ was arbitrary, combining Eqs. (\ref{e.h5.14})
and (\ref{e.h5.15}) implies that $S_{T}f\left(  h\right)  =f\left(
h\right)$ for all $h\in G_{0}:=\cup_{m\in\mathbb{N}}G_{P_{m}}$.
Recall from Lemma \ref{l.h4.9} that
$S_{T}f:G_{CM}\rightarrow\mathbb{C}$ is continuous and from the
proof of \cite[Theorem 8.1]{DG07b} that $G_{0}$ is a dense subgroup
of $G_{CM}$. Therefore we may conclude that in fact $S_{T}f\left(
h\right) =f\left(  h\right)$ for all $h\in G_{CM}$. The Gaussian
bound now follows immediately from Corollary \ref{c.h4.8}.
\end{proof}

\begin{cor}
\label{c.h5.10}Suppose that $\delta>0$ is as in Theorem
\ref{t.h4.11} and $f:G\rightarrow\mathbb{C}$ is a continuous
function such that $f|_{G_{CM}}$ is holomorphic and $\left\vert
f\right\vert \leqslant Ce^{\varepsilon\rho ^{2}/\left(  pT\right) }$
for some $\varepsilon\in\lbrack0,\delta)$. Then $f\in L^{p}\left(
\nu_{T}\right)$,  $S_{T}f=f$,  and the Gaussian bounds in Eq.
(\ref{e.h5.11}) hold.
\end{cor}

\begin{proof}
By Theorem \ref{t.h4.11}, the given function $f$ verifies Eq.
(\ref{e.h5.10}) for any choice of $\left\{  P_{n}\right\}
_{n=1}^{\infty}\subset \operatorname*{Proj}\left(  W\right)$ with
$P_{n}|_{H}\uparrow P$ strongly as $n\uparrow\infty$. Hence Theorem
\ref{t.h5.9} is applicable.
\end{proof}

As a simple consequence of Corollary \ref{c.h5.10}, we know that
$\mathcal{P}\subset L^{p}\left(  \nu_{T}\right)$ (see Definition
\ref{d.h1.6}) and that $\left(  S_{T}p\right)  \left(  h\right)
=p\left( h\right)$ for all $h\in G_{CM}$ and $p\in\mathcal{P}$.
\begin{nota}
\label{n.h5.11}For $T>0$ and $1\leqslant p<\infty$,  let
$\mathcal{A}_{T}^{p}$ and $\mathcal{H}_{T}^{p}\left(  G\right)$
denote the $L^{p}\left(  \nu_{T}\right)$~--~closure of
$\mathcal{A}\cap L^{p}\left(  \nu_{T}\right)$ and $\mathcal{P}$,
where $\mathcal{A}$ and $\mathcal{P}$ denote the holomorphic
cylinder functions (see Definition \ref{d.h4.3}) and holomorphic
cylinder polynomials on $G$ respectively.
\end{nota}

\begin{thm}
\label{t.h5.12}For all $T>0$ and $p\in\left(  1,\infty\right)$, we
have $S_{T}\left(  \mathcal{H}_{T}^{p}\left(  G\right)  \right)
\subset \mathcal{H}\left(  G_{CM}\right)$. \end{thm}

\begin{proof}
Let $f\in\mathcal{H}_{T}^{p}\left(  G\right)$ and
$p_{n}\in\mathcal{P}$ such that $\lim_{n\rightarrow\infty}\left\Vert
f-p_{n}\right\Vert_{L^{p}\left( \nu_{T}\right)  }=0$. If $h\in
G_{CM}$,  then by Corollary \ref{c.h4.8}
\begin{align*}
\left\vert S_{T}f\left(  h\right)  -p_{n}\left(  h\right)  \right\vert  &
=\left\vert S_{T}\left(  f-p_{n}\right)  \left(  h\right)  \right\vert \\
&  \leqslant\left\Vert f-p_{n}\right\Vert_{L^{p}\left(  \nu_{T}\right)  }%
\exp\left(  \frac{c\left(  k\left(  \omega\right)  T/2\right)
}{T\left( p-1\right)  }d_{G_{CM}}^{2}\left(  \mathbf{e},h\right)
\right).
\end{align*}
This shows that $S_{T}f$ is the limit of
$p_{n}|_{G_{CM}}\in\mathcal{H}\left( G_{CM}\right)$ with the limit
being uniform over any bounded subset of $h$'s contained in
$G_{CM}$. This is sufficient to show that $S_{T}f\in
\mathcal{H}\left(  G_{CM}\right)$ via an application of
\cite[Theorem 3.18.1]{HP74}.
\end{proof}

\begin{rem}
\label{r.h5.13} It seems reasonable to conjecture that $\mathcal{A}_{T}%
^{2}=\mathcal{H}_{T}^{2}\left(  G\right)$,  nevertheless we do not
know if these two spaces are equal! We also do not know if
$S_{T}f=f$ for every $f\in\mathcal{A}\cap L^{2}\left(
\nu_{T}\right)$. However, Theorem \ref{t.h5.9} does show that
$S_{T}f=f$ for all $f\in\mathcal{A}\cap
_{P\in\operatorname*{Proj}\left(  W\right)  }L^{p}\left(
\nu_{T}^{P}\right)$ with $L^{p}\left(  \nu_{T}^{P}\right)$--norms of
$f$ being bounded.
\end{rem}

\section{The Taylor isomorphism theorem\label{s.h6}}

The main purpose of this section is to prove the Taylor isomorphism
Theorem \ref{t.h1.5} (or Theorem \ref{t.h6.10}). We begin with the
formal development of the algebraic setup. In what follows below for
 a vector space $V$ we will denote the algebraic dual to $V$ by
$V^{\prime}$. If $V$ happens to be a normed space, we will let
$V^{\ast}$ denote the topological dual of $V$.
\subsection{A non-commutative Fock space\label{s.h6.1}}

\begin{nota}
\label{n.h6.1}For $n\in\mathbb{N}$ let $\mathfrak{g}_{CM}^{\otimes
n}$ denote the $n$--fold algebraic tensor product of $\mathfrak{\
g}_{CM}$ with itself, and by convention let $\mathfrak{g}_{CM}^{\otimes0}:=\mathbb{C}$. Also let%
\[
\mathbf{T:=T}\left(  \mathfrak{g}_{CM}\right)  =\mathbb{C}\oplus
\mathfrak{g}_{CM}\oplus\mathfrak{g}_{CM}^{\otimes2}\oplus\mathfrak{g}%
_{CM}^{\otimes3}\oplus\dots
\]
be the algebraic tensor algebra over $\mathfrak{g}_{CM}$,
$\mathbf{T}^{\prime }$ be its algebraic dual, and $J$ be the
two-sided ideal in $\mathbf{T}$ generated by the elements in Eq.
(\ref{e.h1.3}). The backwards annihilator of
$J$ is%
\begin{equation}
J^{0}=\{\alpha\in\mathbf{T}^{\prime}:\alpha\left(  J\right)  =0\}.
\label{e.h6.1}%
\end{equation}
For any $\alpha\in\mathbf{T}^{\prime}$ and $n\in\mathbb{N\cup}\left\{
0\right\}$,  we let $\alpha_{n}:=\alpha|_{\mathfrak{g}_{CM}^{\otimes n}}%
\in\left(  \mathfrak{g}_{CM}^{\otimes n}\right)^{\prime}$.
\end{nota}

After the next definition we will be able to give numerous examples
of elements in $J^{0}$.
\begin{df}
[Left differentials]\label{d.h6.2}For $f\in\mathcal{H}\left(
G_{CM}\right) $,  $n\in\mathbb{N\cup}\left\{  0\right\}$,  and $g\in
G_{CM}$,  define $\hat{f}_{n}\left(  g\right)  :=D^{n}f\left(
g\right)  \in\left( \mathfrak{g}_{CM}^{\otimes n}\right)^{\prime}$
by
\begin{align}
&  \left(  D^{0}f\right)  \left(  g\right)  =f\left(  g\right)  \text{
and}\nonumber\\
&  \left\langle D^{n}f\left(  g\right), h_{1}\otimes\dots\otimes
h_{n}\right\rangle =\left(  \tilde{h}_{1}\dots\tilde{h}_{n}f\right)
\left(
g\right)  \label{e.h6.2}%
\end{align}
for all and $h_{1},...,h_{n}\in\mathfrak{g}_{0}$,  where
$\tilde{h}f$ is given as in Eq. (\ref{e.h3.8}) or Eq.
(\ref{e.h3.10}). We will write $Df$ for
$D^{1}f$ and $\hat{f}\left(  g\right)  $ to be the element of $\mathbf{T}%
\left(  \mathfrak{g}_{CM}\right)^{\prime}$ determined by%
\begin{equation}
\left\langle \hat{f}\left(  g\right),\beta\right\rangle
=\left\langle \hat{f}_{n}\left(  g\right), \beta\right\rangle \text{
for all }\beta
\in\mathfrak{g}_{CM}^{\otimes n}\text{ and }n\in\mathbb{N}_{0}. \label{e.h6.3}%
\end{equation}

\end{df}

\begin{ex}
\label{ex.h6.3}As a consequence of Eq. (\ref{e.h3.11}),
$\hat{f}\left( g\right)  \in J^{0}$ for all $f\in\mathcal{H}\left(
G_{CM}\right)$ and $g\in G_{CM}$. \end{ex}

In order to put norms on $J^{0}$,  let us equip
$\mathfrak{g}_{CM}^{\otimes n}$ with the usual inner product
determined by
\begin{equation}
\left\langle h_{1}\otimes\dots\otimes h_{n},
k_{1}\otimes\dots\otimes
k_{n}\right\rangle_{\mathfrak{g}_{CM}^{\otimes n}}=\prod_{j=1}^{n}%
\left\langle h_{j}, k_{j}\right\rangle_{\mathfrak{g}_{CM}}\text{ for
any
}h_{i},k_{j}\in\mathfrak{g}_{CM}. \label{e.h6.4}%
\end{equation}
For $n=0$ we let $\left\langle z,w\right\rangle
_{\mathfrak{g}_{CM}^{\otimes 0}}:=z\bar{w}$ for all
$z,w\in\mathfrak{g}_{CM}^{\otimes0}=\mathbb{C}$. The
inner product $\left\langle \cdot,\cdot\right\rangle_{\mathfrak{g}%
_{CM}^{\otimes n}}$ induces a dual inner product on $\left(  \mathfrak{g}%
_{CM}^{\otimes n}\right)^{\ast}$ which we will denote by
$\left\langle
\cdot,\cdot\right\rangle_{n}$. The associated norm on $\left(  \mathfrak{g}%
_{CM}^{\otimes n}\right)^{\ast}$ will be denoted by $\left\Vert
\cdot\right\Vert_{n}$. We extend $\left\Vert \cdot\right\Vert_{n}$
to all of $\left(  \mathfrak{g}_{CM}^{\otimes n}\right)^{\prime}$ by
setting
$\left\Vert \beta\right\Vert_{n}=\infty$ if $\beta\in\left(  \mathfrak{g}%
_{CM}^{\otimes n}\right)^{\prime}\setminus\left(
\mathfrak{g}_{CM}^{\otimes
n}\right)^{\ast}$. If $\Gamma$ is any orthonormal basis for $\mathfrak{g}%
_{CM}$,  then $\left\Vert \beta\right\Vert_{n}$ may be computed
using
\begin{equation}
\left\Vert \beta\right\Vert_{\mathfrak{g}_{CM}^{\otimes
n}}^{2}:=\sum_{h_{1},\dots,h_{n}\in\Gamma}\left\vert \left\langle
\beta,h_{1}\otimes
\dots\otimes h_{n}\right\rangle \right\vert ^{2}. \label{e.h6.5}%
\end{equation}

\begin{df}
[Non-commutative Fock space]\label{d.h6.4}Given $T>0$ and $\alpha\in
J^{0}\left(  \mathfrak{g}_{CM}\right)$,  let
\begin{equation}
\left\Vert \alpha\right\Vert_{J_{T}^{0}\left(
\mathfrak{g}_{CM}\right)
}^{2}:=\sum_{n=0}^{\infty}\frac{T^{n}}{n!}\left\Vert
\alpha_{n}\right\Vert
_{n}^{2}. \label{e.h6.6}%
\end{equation}
Further let%
\begin{equation}
J_{T}^{0}\left(  \mathfrak{g}_{CM}\right)  :=\left\{  \alpha\in
J^{0}\left( \mathfrak{g}_{CM}\right)  :\left\Vert
\alpha\right\Vert_{J_{T}^{0}\left(
\mathfrak{g}_{CM}\right)  }^{2}<\infty\right\}. \label{e.h6.7}%
\end{equation}

\end{df}

The space, $J_{T}^{0}\left(  \mathfrak{g}_{CM}\right)$,  is then a
Hilbert space when equipped with the inner product
\begin{equation}
\left\langle \alpha,\beta\right\rangle_{J_{T}^{0}\left(  \mathfrak{g}%
_{CM}\right)  }=\sum_{n=0}^{\infty}\frac{T^{n}}{n!}\left\langle \alpha
_{n},\beta_{n}\right\rangle_{n}\text{ for any }\alpha,\beta\in J_{T}%
^{0}\left(  \mathfrak{g}_{CM}\right). \label{e.h6.8}%
\end{equation}

\subsection{The Taylor isomorphism\label{s.h6.2}}

\begin{lem}
\label{l.h6.5}Let $f\in\mathcal{H}\left(  G_{CM}\right)  $ and $T>0$
and suppose that $\left\{  P_{n}\right\}_{n=1}^{\infty}\subset
\operatorname*{Proj}\left(  W\right)  $ is a sequence such that $P_{n}%
|_{\mathfrak{g}_{CM}}\uparrow I_{\mathfrak{g}_{CM}}$ as $n\rightarrow\infty$. Then%
\begin{equation}
\lim_{n\rightarrow\infty}\left\Vert \hat{f}\left(  \mathbf{e}\right)
\right\Vert_{J_{T}^{0}\left(  \mathfrak{g}_{P_{n}}\right)
}=\left\Vert \hat{f}\left(  \mathbf{e}\right)  \right\Vert
_{J_{T}^{0}\left(
\mathfrak{g}_{CM}\right)  }=\left\Vert f\right\Vert_{\mathcal{H}_{T}%
^{2}\left(  G_{CM}\right)  }=\lim_{n\rightarrow\infty}\left\Vert f\right\Vert
_{L^{2}\left(  G_{P_{n}},\nu_{T}^{P_{n}}\right)  }, \label{e.h6.9}%
\end{equation}
where $\left\Vert \cdot\right\Vert_{\mathcal{H}_{T}^{2}\left(
G_{CM}\right) }$ is defined in Eq. (\ref{e.h1.4}).
\end{lem}

\begin{proof}
By Theorem 5.1 of \cite{Driver1997c}, for all
$P\in\operatorname*{Proj}\left( W\right)$,  \begin{equation}
\left\Vert f\right\Vert_{L^{2}\left(  G_{P},\nu_{T}^{P}\right)
}=\left\Vert \hat{f}\left(  \mathbf{e}\right)  \right\Vert
_{J_{T}^{0}\left(
\mathfrak{g}_{P}\right)  }, \label{e.h6.10}%
\end{equation}
where
\begin{equation}
\left\Vert \hat{f}\left(  \mathbf{e}\right)
\right\Vert_{J_{T}^{0}\left(
\mathfrak{g}_{P}\right)  }^{2}=\sum_{n=0}^{\infty}\frac{T^{n}}{n!}%
\sum_{\left\{  h_{j}\right\}_{j=1}^{n}\subset\Gamma_{P}}\left\vert
\left\langle \hat{f}\left(  \mathbf{e}\right)
,h_{1}\otimes\dots\otimes
h_{n}\right\rangle \right\vert ^{2} \label{e.h6.11}%
\end{equation}
and $\Gamma_{P}$ is an orthonormal basis for $\mathfrak{g}_{P}$. In
particular, it follows that
\begin{equation}
\left\Vert f\right\Vert_{\mathcal{H}_{T}^{2}\left(  G_{CM}\right)  }%
=\sup_{P\in\operatorname*{Proj}\left(  W\right)  }\left\Vert
\hat{f}\left( \mathbf{e}\right)  \right\Vert_{J_{T}^{0}\left(
\mathfrak{g}_{P}\right)  }
\label{e.h6.12}%
\end{equation}
and hence we must now show
\begin{equation}
\sup_{P\in\operatorname*{Proj}\left(  W\right)  }\left\Vert
\hat{f}\left( \mathbf{e}\right)  \right\Vert_{J_{T}^{0}\left(
\mathfrak{g}_{P}\right) }=\left\Vert \hat{f}\left( \mathbf{e}\right)
\right\Vert_{J_{T}^{0}\left(
\mathfrak{g}_{CM}\right)  }. \label{e.h6.13}%
\end{equation}
If $\Gamma$ is an orthonormal basis for $\mathfrak{g}_{CM}$
containing $\Gamma_{P}$,  it follows that
\[
\left\Vert \hat{f}\left(  \mathbf{e}\right)
\right\Vert_{J_{T}^{0}\left(
\mathfrak{g}_{P}\right)  }^{2}=\sum_{n=0}^{\infty}\frac{T^{n}}{n!}%
\sum_{\left\{  h_{j}\right\}_{j=1}^{n}\subset\Gamma}\left\vert
\left\langle \hat{f}\left(  \mathbf{e}\right)
,h_{1}\otimes\dots\otimes h_{n}\right\rangle \right\vert
^{2}=\left\Vert \hat{f}\left(  \mathbf{e}\right)  \right\Vert
_{J_{T}^{0}\left(  \mathfrak{g}_{CM}\right)  }^{2},
\]
which shows that $\sup_{P\in\operatorname*{Proj}\left(  W\right)
}\left\Vert \hat{f}\left(  \mathbf{e}\right)
\right\Vert_{J_{T}^{0}\left(
\mathfrak{g}_{P}\right)  }\leqslant\left\Vert \hat{f}\left(  \mathbf{e}%
\right)  \right\Vert_{J_{T}^{0}\left(  \mathfrak{g}_{CM}\right) }$.
We may choose orthonormal bases, $\Gamma_{P_{n}}$,  for
$\mathfrak{g}_{P_{n}}$ such that $\Gamma_{P_{n}}\uparrow\Gamma$ as
$n\uparrow\infty$. Then it is easy to show that
\begin{align*}
\lim_{n\rightarrow\infty}\left\Vert f\right\Vert_{L^{2}\left(  G_{P_{n}}%
,\nu_{T}^{P_{n}}\right)  }  &  =\lim_{n\rightarrow\infty}\left\Vert \hat
{f}\left(  \mathbf{e}\right)  \right\Vert_{J_{T}^{0}\left(  \mathfrak{g}%
_{P_{n}}\right)  }\\
&
=\lim_{n\rightarrow\infty}\sum_{n=0}^{\infty}\frac{T^{n}}{n!}\sum_{\left\{
h_{j}\right\}_{j=1}^{n}\subset\Gamma_{P_{n}}}\left\vert \left\langle
\hat {f}\left(  \mathbf{e}\right),h_{1}\otimes\dots\otimes
h_{n}\right\rangle
\right\vert ^{2}\\
&  =\sum_{n=0}^{\infty}\frac{T^{n}}{n!}\sum_{\left\{  h_{j}\right\}
_{j=1}^{n}\subset\Gamma}\left\vert \left\langle \hat{f}\left(  \mathbf{e}%
\right),h_{1}\otimes\dots\otimes h_{n}\right\rangle \right\vert
^{2}=\left\Vert \hat{f}\left(  \mathbf{e}\right)  \right\Vert_{J_{T}%
^{0}\left(  \mathfrak{g}_{CM}\right)  }%
\end{align*}
from which it follows that $\sup_{P\in\operatorname*{Proj}\left(
W\right) }\left\Vert \hat{f}\left(  \mathbf{e}\right)
\right\Vert_{J_{T}^{0}\left(
\mathfrak{g}_{P}\right)  }\geqslant\left\Vert \hat{f}\left(  \mathbf{e}%
\right)  \right\Vert_{J_{T}^{0}\left(  \mathfrak{g}_{CM}\right) }$.
\end{proof}

For the next corollary, recall that $\mathcal{P}$ and $\mathcal{P}_{CM}$
denote the spaces of holomorphic cylinder polynomials on $G$ and $G_{CM}$
respectively, see Definition \ref{d.h1.6} and Eq. (\ref{e.h1.7}).

\begin{cor}
\label{c.h6.6}If $f:G\rightarrow\mathbb{C}$ is a continuous function
satisfying the bounds in Proposition \ref{p.h4.12} with $p=2$,  then
$f|_{G_{CM}}\in\mathcal{H}_{T}^{2}\left(  G_{CM}\right)$ and
$\hat{f}\left( \mathbf{e}\right)  \in J_{T}^{0}\left(
\mathfrak{g}_{CM}\right)$. In particular, for all $T>0$,
$\mathcal{P}_{CM}\subset\mathcal{H}_{T}^{2}\left(
G_{CM}\right)$ and for any $p\in\mathcal{P}$,  $\hat{p}\left(  \mathbf{e}%
\right)  \in J_{T}^{0}\left(  \mathfrak{g}_{CM}\right)$. This shows
that $\mathcal{H}_{T}^{2}\left(  G_{CM}\right)$ and $J_{T}^{0}\left(
\mathfrak{g}_{CM}\right)$ are non-trivial spaces.
\end{cor}

\begin{df}
\label{d.h6.7}For each $T>0$,  the \textbf{Taylor map} is the linear
map, $\mathcal{T}_{T}:\mathcal{H}_{T}^{2}\left(  G_{CM}\right)
\rightarrow
J_{T}^{0}\left(  \mathfrak{g}_{CM}\right)$,  defined by $\mathcal{T}%
_{T}f:=\hat{f}\left(  \mathbf{e}\right)$. \end{df}

\begin{cor}
\label{c.h6.8}The Taylor map,
$\mathcal{T}_{T}:\mathcal{H}_{T}^{2}\left( G_{CM}\right) \rightarrow
J_{T}^{0}\left(  \mathfrak{g}_{CM}\right)$,  is
injective. Moreover, the function $\left\Vert \cdot\right\Vert_{\mathcal{H}%
_{T}^{2}\left(  G_{CM}\right)}$  is a norm on
$\mathcal{H}_{T}^{2}\left(
G_{CM}\right)$ which is induced by the inner product on $\mathcal{H}_{T}%
^{2}\left(  G_{CM}\right)$ defined by%
\begin{equation}
\left\langle u, v\right\rangle_{\mathcal{H}_{T}^{2}\left(
G_{CM}\right) }:=\left\langle \hat{u}\left(  \mathbf{e}\right)
,\hat{v}\left( \mathbf{e}\right)  \right\rangle_{J_{T}^{0}\left(
\mathfrak{g}_{CM}\right) }\text{ for any }u,
v\in\mathcal{H}_{T}^{2}\left(  G_{CM}\right).
\label{e.h6.14}%
\end{equation}

\end{cor}

\begin{proof}
If $\hat{f}\left(  \mathbf{e}\right)  =0$,  then $\left\Vert
f\right\Vert_{\mathcal{H}_{T}^{2}\left(  G_{CM}\right)  }=0$ which
then implies that $f|_{G_{P}}\equiv0$ for all
$P\in\operatorname*{Proj}\left(  W\right)$. As
$f:G_{CM}\rightarrow\mathbb{C}$ is continuous and $\cup_{P\in
\operatorname*{Proj}\left(  W\right)  }G_{P}$ is dense in $G_{CM}$
(see the end of the proof of Theorem \ref{t.h5.9}), it follows that
$f\equiv0$. Hence we have shown $\mathcal{T}_{T}$ injective. Since
$\left\Vert \cdot\right\Vert_{J_{T}^{0}\left(
\mathfrak{g}_{CM}\right)  }$ is a Hilbert norm and, by Lemma
\ref{l.h6.9}, $\left\Vert f\right\Vert_{\mathcal{H}_{T}^{2}\left(
G_{CM}\right)  }=\left\Vert \mathcal{T}_{T}f\right\Vert
_{J_{T}^{0}\left( \mathfrak{g}_{CM}\right)}$,  it follows that
$\left\Vert \cdot\right\Vert
_{\mathcal{H}_{T}^{2}\left(  G_{CM}\right)}$ is the norm on $\mathcal{H}%
_{T}^{2}\left(  G_{CM}\right)  $ induced by the inner product defined in Eq.
(\ref{e.h6.14}).
\end{proof}

Our next goal is to show that the Taylor map, $\mathcal{T}_{T}$,  is
surjective. The following lemma motivates the construction of the
inverse of the Taylor map.

\begin{lem}
\label{l.h6.9}For every $f\in\mathcal{H}\left(  G_{CM}\right),$%
\begin{equation}
f\left(  g\right)  =\sum_{n=0}^{\infty}\frac{1}{n!}\left\langle \hat{f}%
_{n}\left(  \mathbf{e}\right), g^{\otimes n}\right\rangle \text{ for
any
}g\in G_{CM}, \label{e.h6.15}%
\end{equation}
where the above sum is absolutely convergent. By convention,
$g^{\otimes 0}=1\in\mathbb{C}$. (For a more general version of this
Lemma, see Proposition 5.1 in \cite{Driver1995a}.)
\end{lem}

\begin{proof}
The function $u\left(  z\right)  :=f\left(  zg\right)  $ is a
holomorphic function of $z\in\mathbb{C}$. Therefore,
\[
f\left(  g\right)  =u\left(  1\right)  =\sum_{n=0}^{\infty}\frac{1}%
{n!}u^{\left(  n\right)  }\left(  0\right)
\]
and the above sum is absolutely convergent. In fact, one easily sees that for
all $R>0$ there exists $C\left(  R\right)  <\infty$ such that $\frac{1}%
{n!}\left\vert u^{\left(  n\right)  }\left(  0\right)  \right\vert
\leqslant C\left(  R\right)  R^{-n}$ for all $n\in\mathbb{N}$. The
proof is now completed upon observing
\begin{align*}
u^{\left(  n\right)  }\left(  0\right)   &  =\left(
\frac{d}{dt}\right)^{n}u\left(  t\right)\left|_{t=0}\right.=\left(
\frac{d}{dt}\right)^{n}f\left(
tg\right)\left|_{t=0}\right. \\
&  =\left(  \frac{d}{dt}\right)^{n}f\left(  e^{tg}\right)
\left|_{t=0}\right.=\left(
\tilde{g}^{n}f\right)  \left(  \mathbf{e}\right)  =\left\langle \hat{f}%
_{n}\left(  \mathbf{e}\right), g^{\otimes n}\right\rangle .
\end{align*}

\end{proof}

The next theorem is a more precise version of Theorem \ref{t.h1.5}.

\begin{thm}
[Taylor isomorphism theorem]\label{t.h6.10}For all $T>0$,  the space
$\mathcal{H}_{T}^{2}\left(  G_{CM}\right)  $ equipped with the inner
product $\left\langle \cdot,\cdot\right\rangle
_{\mathcal{H}_{T}^{2}\left(
G_{CM}\right)  }$ is a Hilbert space, $\mathcal{T}\left(  \mathcal{H}_{T}%
^{2}\left(  G_{CM}\right)  \right)  \subset J_{T}^{0}\left(  \mathfrak{g}%
_{CM}\right)$,  and $\mathcal{T}_{T}:=\mathcal{T}|_{\mathcal{H}_{T}%
^{2}\left(  G_{CM}\right)  }:\mathcal{H}_{T}^{2}\left(  G_{CM}\right)
\rightarrow J_{T}^{0}\left(  \mathfrak{g}_{CM}\right)  $ is a unitary transformation.
\end{thm}

\begin{proof}
Given Corollary \ref{c.h6.8}, it only remains to prove
$\mathcal{T}_{T}$ is surjective. So let $\alpha\in J_{T}^{0}\left(
\mathfrak{g}_{CM}\right)$. By Lemma \ref{l.h6.9}, if
$f=\mathcal{T}_{T}^{-1}\alpha$ exists it must be given by
\begin{equation}
f\left(  g\right)  :=\sum_{n=0}^{\infty}\frac{1}{n!}\left\langle \alpha
_{n},g^{\otimes n}\right\rangle \text{ for any }g \in G_{CM}. \label{e.h6.16}%
\end{equation}
We now have to check that the sum is convergent, the resulting function $f$ is
in $\mathcal{H}\left(  G_{CM}\right)$,  and $\hat{f}\left(  \mathbf{e}%
\right)  =\alpha$. Once this is done, we may apply Lemma
\ref{l.h6.5} to conclude that $\left\Vert f\right\Vert
_{\mathcal{H}_{T}^{2}\left( G_{CM}\right)  }=\left\Vert
\alpha\right\Vert_{J_{T}^{0}\left( \mathfrak{g}_{CM}\right)
}<\infty$ and hence we will have shown that
$f\in\mathcal{H}_{T}^{2}\left(  G_{CM}\right)  $ and $\mathcal{T}_{T}%
f=\alpha$. For each $n\in\mathbb{N}\cup\left\{  0\right\}$,  the
function $u_{n}\left( g\right)  :=\frac{1}{n!}\left\langle
\alpha_{n},g^{\otimes n}\right\rangle $ is a continuous complex
$n$--linear form in $g\in G_{CM}$ and therefore holomorphic. Since
$\left\vert \left\langle \alpha_{n},g^{\otimes n}\right\rangle
\right\vert \leqslant\left\Vert \alpha_{n}\right\Vert_{n}\left\Vert
g\right\Vert_{\mathfrak{g}_{CM}}^{n}$,  then for $R>0$
\[
\sup\left\{  \left\vert u_{n}\left(  g\right)  \right\vert
:\left\Vert g\right\Vert_{\mathfrak{g}_{CM}}\leqslant R\right\}
\leqslant\left\Vert \alpha_{n}\right\Vert_{n}R^{n}.
\]
Therefore it follows that%
\begin{align}
\sum_{n=0}^{\infty}\sup\left\{  \left\vert u_{n}\left(  g\right)
\right\vert :\left\Vert g\right\Vert_{\mathfrak{g}_{CM}}\leqslant
R\right\}   & \leqslant\sum_{n=0}^{\infty}\frac{T^{n}}{n!}\left\Vert
\alpha_{n}\right\Vert
_{n}\frac{R^{n}}{T^{n}}\nonumber\\
&  \leqslant\sqrt{\sum_{n=0}^{\infty}\frac{T^{n}}{n!}\left\Vert
\alpha
_{n}\right\Vert_{n}^{2}}\sqrt{\sum_{n=0}^{\infty}\frac{T^{n}}{n!}\left(
\frac{R^{n}}{T^{n}}\right)^{2}}\nonumber\\
&  =\left\Vert \alpha\right\Vert_{J_{T}^{0}\left(
\mathfrak{g}_{CM}\right)
} e^{R^{2}/\left(  2T\right)  }<\infty. \label{e.h6.17}%
\end{align}
This shows $f\left(  g\right)  =\lim_{N\rightarrow\infty}\sum_{n=0}^{N}%
u_{n}\left(  g\right)  $ with the limit being uniform over $g$ in
bounded subsets of $\mathfrak{g}_{CM}$. Hence, the sum in Eq.
(\ref{e.h6.16}) is convergent and (see \cite[Theorem 3.18.1]{HP74})
the resulting function, $f$,  is in $\mathcal{H}\left(
G_{CM}\right)$.
Since%
\[
f\left(  zh\right)  =\sum_{n=0}^{\infty}\frac{z^{n}}{n!}\left\langle
\alpha_{n},h^{\otimes n}\right\rangle \text{ for any }z\in\mathbb{C}\text{ and
}h\in\mathfrak{g}_{CM},
\]
it follows that
\[
\left\langle \alpha_{n},h^{\otimes n}\right\rangle =\left(  \frac{d}%
{dz}\right)^{n}f\left(  zh\right)  |_{z=0}=\left(
\frac{d}{dt}\right) ^{n}f\left(  e^{th}\right)  |_{t=0}=\left\langle
\hat{f}_{n}\left( \mathbf{e}\right),h^{\otimes n}\right\rangle .
\]
This is true for all $n$ and $h\in\mathfrak{g}_{CM}$,  so we may use
the argument following Eq. (6.13) in \cite{Driver1995a} (or see the
proof of Theorem 2.5 in \cite{D-G-SC07b}) to show $\hat{f}\left(
\mathbf{e}\right) =\alpha$. \end{proof}

As a consequence of Eq. (\ref{e.h6.17}) we see that if $f\in\mathcal{H}%
_{T}^{2}\left(  G_{CM}\right)  $ then%
\begin{equation}
\left\vert f\left(  g\right)  \right\vert \leqslant\left\Vert
f\right\Vert _{\mathcal{H}_{T}^{2}\left(  G_{CM}\right)  }
e^{\left\Vert g\right\Vert _{\mathfrak{g}_{CM}}^{2}/\left( 2T\right)
}\text{ for any }g\in G_{CM}.
\label{e.h6.18}%
\end{equation}
The next theorem, which is an analogue of Bargmann's pointwise
bounds (see \cite[Eq. (1.7)]{Bargmann61} and \cite[Eq.
(5.4)]{Driver1997c}), improves upon the estimate in Eq.
(\ref{e.h6.18}).

\begin{thm}
[Pointwise bounds]\label{t.h6.11}If $f\in\mathcal{H}_{T}^{2}\left(
G_{CM}\right)$ and $g\in G_{CM}$,  then for all $g\in G_{CM},$%
\begin{equation}
\left\vert f\left(  g\right)  \right\vert \leqslant\left\Vert
f\right\Vert _{\mathcal{H}_{T}^{2}\left(  G_{CM}\right)  }
e^{d_{CM}^{2}\left(
\mathbf{e}, g\right)  /\left(  2T\right)}, \label{e.h6.19}%
\end{equation}
where $d_{CM}^{2}\left(  \cdot,\cdot\right)$ is the distance
function on $G_{CM}$ defined in Eq. (\ref{e.h4.7}).
\end{thm}

\begin{proof}
Let $P_{n}\in\operatorname*{Proj}\left(  W\right)  $ be chosen so
that $P_{n}|_{\mathfrak{g}_{CM}}\uparrow I_{\mathfrak{g}_{CM}}$ as
$n\rightarrow \infty$ and recall that
$G_{0}:=\cup_{n=1}^{\infty}G_{P_{n}}$ is a dense subgroup of
$G_{CM}$ as explained in the proof of Theorem \ref{t.h5.9}. Let
$g\in G_{P_{m}}$ for some $m\in\mathbb{N}$ and let $\sigma:\left[
0,1\right] \rightarrow G_{CM}$ be a $C^{1}$--curve such that
$\sigma\left(  0\right) =\mathbf{e}$ and $\sigma\left(  1\right)
=g$. Then for $n\geqslant m$,  $\sigma_{n}\left(  t\right)
:=\pi_{P_{n}}\left(  \sigma\left(  t\right) \right)  $ is a $C^{1}$
curve in $G_{n}$ such that $\sigma_{n}\left( 0\right)  =\mathbf{e}$
and $\sigma_{n}\left(  1\right)  =g$. Therefore by
\cite[Eq. (5.4)]{Driver1997c}, we have%
\begin{equation}
\left\vert f\left(  g\right)  \right\vert \leqslant\left\Vert f|_{G_{P_{n}}%
}\right\Vert_{L^{2}\left(  G_{P_{n}},\nu_{T}^{P_{n}}\right)  }\cdot
e^{d_{G_{P_{n}}}^{2}\left(  \mathbf{e},g\right)  /\left(  2T\right)
}\leqslant\left\Vert f\right\Vert_{\mathcal{H}_{T}^{2}\left(
G_{CM}\right) }\cdot e^{\ell_{G_{CM}}^{2}\left(  \sigma_{n}\right)
/\left(  2T\right)  },
\label{e.h6.20}%
\end{equation}
where $\ell_{G_{CM}}\left(  \sigma_{n}\right)  $ is the length of
$\sigma_{n}$ as in Eq. (\ref{e.h4.6}). In the proof \cite[Theorem
8.1]{DG07b}, it was shown that
$\lim_{n\rightarrow\infty}\ell_{G_{CM}}\left(  \sigma_{n}\right)
=\ell_{G_{CM}}\left(  \sigma\right)$. Hence we may pass to the limit
in Eq. (\ref{e.h6.20}) to find, $\left\vert f\left(  g\right)
\right\vert \leqslant\left\Vert f\right\Vert
_{\mathcal{H}_{T}^{2}\left(  G_{CM}\right) }\cdot
e^{\ell_{G_{CM}}^{2}\left(  \sigma\right)  /\left(  2T\right)  }$.
Optimizing this last inequality over all $\sigma$ joining
$\mathbf{e}$ to $g$ then shows that Eq. (\ref{e.h6.19}) holds for
all $g\in G_{0}$. This suffices to prove Eq. (\ref{e.h6.19}) as both
sides of this inequality are continuous in $g\in G_{CM}$ and $G_{0}$
is dense in $G_{CM}$. \end{proof}

\section{Density theorems\label{s.h7}}

The following density result is the main theorem of this section and is
crucial to the next section. Techniques similar to those used in this section
have appeared in Cecil \cite{Cecil2008} to prove an analogous result for path
groups over stratified Lie groups.

\begin{thm}
[Density theorem]\label{t.h7.1} For all $T>0$,  $\mathcal{P}_{CM}$
defined by Eq. (\ref{e.h1.7}) is a dense subspace of
$\mathcal{H}_{T}^{2}\left( G_{CM}\right)$.
\end{thm}

\begin{proof}
This theorem is a consequence of Corollary \ref{c.h7.4} and Proposition
\ref{p.h7.12} below.
\end{proof}

The remainder of this section will be devoted to proving the results
used in the proof of the theorem. We will start by constructing some
auxiliary dense subspaces of $J_{T}^{0}\left(
\mathfrak{g}_{CM}\right)$ and $\mathcal{H}_{T}^{2}\left(
G_{CM}\right)$.
\subsection{Finite rank subspaces\label{s.h7.1}}

\begin{df}
\label{d.h7.2}A tensor, $\alpha\in J^{0}\left(
\mathfrak{g}_{CM}\right)$,  is said to have \textbf{finite rank} if
$\alpha_{n}=0$ for all but finitely many $n\in\mathbb{N}$.
\end{df}

The next lemma is essentially a special case of \cite[Lemma 3.5]{D-G-SC07b}.

\begin{lem}
[Finite Rank Density Lemma]\label{l.h7.3}The finite rank tensors in $J_{T}%
^{0}\left(  \mathfrak{g}_{CM}\right)  $ are dense in
$J_{T}^{0}\left( \mathfrak{g}_{CM}\right)$. \end{lem}

\begin{proof}
For $\theta\in\mathbb{R}$,  let $\varphi_{\theta}:\mathfrak{g}_{CM}%
\rightarrow\mathfrak{g}_{CM}$ be defined by
\[
\varphi_{\theta}\left(  A, a\right)  =\left( e^{i\theta}A,
e^{i2\theta }a\right).
\]
Since
\begin{align*}
\left[  \varphi_{\theta}\left(  A, a\right), \varphi_{\theta}\left(
B, b\right)  \right]   &  =\left[  \left( e^{i\theta}A,
e^{i2\theta}a\right)
, \left(  e^{i\theta}B, e^{i2\theta}b\right)  \right] \\
&  =\left(  0, \omega\left(  e^{i\theta}A, e^{i\theta}B\right)
\right) =\left(  0, e^{i2\theta}\omega\left(  A, B\right)  \right)
=\varphi_{\theta }\left[  \left(  A, a\right), \left(  B, b\right)
\right]
\end{align*}
we see that $\varphi_{\theta}$ is a Lie algebra homomorphism.

Now let $\Phi_{\theta}:\mathbf{T}\left(  \mathfrak{g}_{CM}\right)
\rightarrow\mathbf{T}\left(  \mathfrak{g}_{CM}\right)$ be defined by
$\Phi_{\theta}1=1$ and
\[
\Phi_{\theta}\left(  h_{1}\otimes\dots\otimes h_{n}\right)  =\varphi_{\theta
}h_{1}\otimes\dots\otimes\varphi_{\theta}h_{n}\text{ for all }h_{i}%
\in\mathfrak{g}_{CM}\text{ and }n\in\mathbb{N}.
\]
If we write $\xi\wedge\eta$ for $\xi\otimes\eta-\eta\otimes\xi$,  then%
\begin{align*}
\Phi_{\theta}(\xi\wedge\eta-[\xi,\eta])  &  =(\varphi_{e^{i\theta}}\xi
)\wedge(\varphi_{e^{i\theta}}\eta)-\varphi_{e^{i\theta}}[\xi,\eta]\\
&  =(\varphi_{e^{i\theta}}\xi)\wedge(\varphi_{e^{i\theta}}\eta)-[\varphi
_{e^{i\theta}}\xi,\varphi_{e^{i\theta}}\eta].
\end{align*}
From this it follows that $\Phi_{\theta}\left(  J\right)  \subset J$
and therefore if $\alpha\in J^{0}\left(  \mathfrak{g}_{CM}\right)$,
then $\alpha\circ\Phi_{\theta}\in J^{0}\left(
\mathfrak{g}_{CM}\right)$. Letting $\Gamma$ be an orthonormal basis
as in Eq. (\ref{e.h5.8}), we have $\varphi_{\theta}h=e^{i2\theta}h$
or $\varphi_{\theta}h=e^{i\theta}h$ for all $h\in\Gamma$. Therefore
it follows that
\begin{align*}
\left\vert \langle\alpha\circ\Phi_{\theta},k_{1}\otimes k_{2}\otimes
\dots\otimes k_{n}\rangle\right\vert ^{2}  &  =\left\vert \langle
\alpha,\varphi_{\theta}k_{1}\otimes\varphi_{\theta}k_{2}\otimes\dots
\otimes\varphi_{\theta}k_{n}\rangle\right\vert ^{2}\\
&  =\left\vert \langle\alpha, k_{1}\otimes k_{2}\otimes\dots\otimes
k_{n}\rangle\right\vert ^{2}%
\end{align*}
and hence that
\begin{align*}
\left\Vert \alpha\circ\Phi_{\theta}\right\Vert_{J_{T}^{0}\left(
\mathfrak{g}_{CM}\right)  }^{2}  &  =\sum_{n=0}^{\infty}\frac{T^{n}}{n!}%
\sum_{k_{1},k_{2},\dots, k_{n}\in\Gamma}\left\vert
\langle\alpha\circ \Phi_{\theta}, k_{1}\otimes
k_{2}\otimes\dots\otimes k_{n}\rangle\right\vert
^{2} \\
&  =\sum_{n=0}^{\infty}\frac{T^{n}}{n!}\sum_{k_{1},k_{2},\dots,
k_{n}\in\Gamma
}\left\vert \langle\alpha, k_{1}\otimes k_{2}\otimes\dots\otimes k_{n}%
\rangle\right\vert ^{2}=\left\Vert
\alpha\right\Vert_{J_{T}^{0}\left( \mathfrak{g}_{CM}\right)  }^{2}.
\end{align*}
So the map $\alpha\in J_{T}^{0}\left(  \mathfrak{g}_{CM}\right)
\rightarrow\alpha\circ\Phi_{\theta}\in J_{T}^{0}\left(  \mathfrak{g}%
_{CM}\right)  $ is unitary. Moreover, since%
\[
\left\vert \langle\alpha,\varphi_{\theta}k_{1}\otimes\varphi_{\theta}%
k_{2}\otimes\dots\otimes\varphi_{\theta}k_{n}\rangle-\langle\alpha
,k_{1}\otimes k_{2}\otimes\dots\otimes k_{n}\rangle\right\vert ^{2}%
\leqslant2\left\vert \langle\alpha, k_{1}\otimes
k_{2}\otimes\dots\otimes
k_{n}\rangle\right\vert ^{2}%
\]
we may apply the dominated convergence theorem to conclude
\begin{align*}
&  \lim_{\theta\rightarrow0}\left\Vert \alpha\circ\Phi_{\theta}-\alpha
\right\Vert_{J_{T}^{0}\left(  \mathfrak{g}_{CM}\right)  }^{2}\\
&  =\sum_{n=0}^{\infty}\frac{T^{n}}{n!}\sum_{k_{1}, k_{2}, \dots,
k_{n}\in\Gamma
}\lim_{\theta\rightarrow0}\left\vert \langle\alpha, \varphi_{\theta}%
k_{1}\otimes\varphi_{\theta}k_{2}\otimes\dots\otimes\varphi_{\theta}%
k_{n}\rangle-\langle\alpha, k_{1}\otimes k_{2}\otimes\dots\otimes k_{n}%
\rangle\right\vert ^{2}\\
&  =0,
\end{align*}
so that $\alpha\rightarrow\alpha\circ\Phi_{\theta}$ is continuous.
(Notice that $\Phi_{\theta}\circ\Phi_{\alpha}=\Phi_{\theta+\alpha}$,
so it suffices to check continuity at $\theta=0$.)

Let
\[
F_{n}(\theta)=\frac{1}{2\pi n}\ \sum_{k=0}^{n-1}\ \sum_{\ell=-k}^{k}%
e^{i\ell\theta}=\frac{1}{2\pi n}\ \frac{\sin^{2}(k\theta/2)}{\sin^{2}%
(\theta/2)}%
\]
denote Fejer's kernel \cite[p. 143]{T}. Then $\int_{-\pi}^{\pi}F_{n}%
(\theta)d\theta=1$ for all $n$ and
\[
\lim_{n\rightarrow\infty}\int_{-\pi}^{\pi}F_{n}(\theta)u(\theta)d\theta
=u\left(  0\right)  \text{ for all }u\in C\left(  [-\pi, \pi],\mathbb{C}%
\right).
\]
We now let
\[
\alpha\left(  n\right)  :=\int_{-\pi}^{\pi}\alpha\circ\Phi_{\theta}%
F_{n}\left(  \theta\right)  d\theta.
\]
Then
\begin{align*}
\limsup_{n\rightarrow\infty}\left\Vert \alpha-\alpha\left(  n\right)
\right\Vert_{J_{T}^{0}\left(  \mathfrak{g}_{CM}\right)  }^{2}  &
\leqslant\limsup_{n\rightarrow\infty}\left\Vert
\int_{-\pi}^{\pi}\left[ \alpha-\alpha\circ\Phi_{\theta}\right]
F_{n}\left(  \theta\right)
d\theta\right\Vert_{J_{T}^{0}\left(  \mathfrak{g}_{CM}\right)  }\\
&  \leqslant\limsup_{n\rightarrow\infty}\int_{-\pi}^{\pi}\left\Vert
\alpha-\alpha\circ\Phi_{\theta}\right\Vert_{J_{T}^{0}\left(  \mathfrak{g}%
_{CM}\right)  }F_{n}\left(  \theta\right)  d\theta=0.
\end{align*}

Moreover if $\beta:=k_{1},\dots, k_{m}\in\mathfrak{g}_{CM}$ with
$m>n$,  then
there exits $\beta_{l}\in\mathfrak{g}_{CM}^{\otimes m}$ such that%
\[
\Phi_{\theta}\beta=\sum_{l=m}^{2m}e^{il\theta}\beta_{l}.
\]
From this it follows that
\[
\left\langle \alpha\left(  n\right), \beta\right\rangle
:=\int_{-\pi}^{\pi }\left\langle
\alpha,\Phi_{\theta}\beta\right\rangle F_{n}\left(
\theta\right)  d\theta=\sum_{l=m}^{2m}\left\langle \alpha,\beta_{l}%
\right\rangle \int_{-\pi}^{\pi}e^{il\theta}F_{n}\left(  \theta\right)
d\theta=0
\]
from which it follows that $\alpha\left(  n\right)_{m}\equiv0$ for
all $m>n$. Thus $\alpha\left(  n\right)  $ is a finite rank tensor
for all $n\in\mathbb{N}$ and $\limsup_{n\rightarrow\infty}\left\Vert
\alpha
-\alpha\left(  n\right)  \right\Vert_{J_{T}^{0}\left(  \mathfrak{g}%
_{CM}\right)}^{2}=0$. \end{proof}

\begin{cor}
\label{c.h7.4}The vector space,%
\begin{equation}
\mathcal{H}_{T,\text{fin}}^{2}\left(  G_{CM}\right)  :=\left\{  u\in
\mathcal{H}_{T}^{2}\left(  G_{CM}\right)  :\hat{u}\left(  \mathbf{e}\right)
\in J_{T}^{0}\left(  \mathfrak{g}_{CM}\right)  \text{ has a finite
rank}\right\}  \label{e.h7.1}%
\end{equation}
is a dense subspace of $\mathcal{H}_{T}^{2}\left(  G_{CM}\right)$.
\end{cor}

\begin{proof}
This follows directly from Lemma \ref{l.h7.3} and the Taylor isomorphism
Theorem \ref{t.h6.10}.
\end{proof}

\subsection{Polynomial approximations\label{s.h7.2}}

To prove Theorem \ref{t.h7.1}, it suffices to show that every
element $u\in\mathcal{H}_{T, \text{fin}}^{2}\left(  G_{CM}\right)  $
may be well approximated by an element from
$\mathcal{H}_{T}^{2}\left(  G\right)$. In order to do this, let
$\left\{  e_{j}:j=1,2,\right\}  \subset H_{\ast}$ be an
orthonormal basis for $H$ and for $N\in\mathbb{N}$,  define $P_{N}%
\in\operatorname*{Proj}\left(  W\right)$ as in Eq. (\ref{e.h2.17}), i.e.%
\begin{equation}
P_{N}\left(  w\right)  =\sum_{j=1}^{N}\left\langle w,e_{j}\right\rangle
_{H}e_{j}\text{ for all }w\in W. \label{e.h7.2}%
\end{equation}
Let us further define $\pi_{N}:=\pi_{P_{N}}$ and
\begin{equation}
u_{N}:=u\circ\pi_{N}\text{ for all }N\in\mathbb{N}. \label{e.h7.3}%
\end{equation}
We are going to prove Theorem \ref{t.h7.1} by showing
$u_{N}\in\mathcal{P}$ and $u_{N}\rightarrow u$ in
$\mathcal{H}_{T}^{2}\left(  G_{CM}\right)$.
\begin{rem}
\label{r.h7.5}A complicating factor in showing $u_{N}|_{G_{CM}}\rightarrow u$
in $\mathcal{H}_{T}^{2}\left(  G_{CM}\right)  $ is the fact that for general
$\omega$ and $P\in\operatorname*{Proj}\left(  W\right)$,  $\pi_{P}%
:G\rightarrow G_{P}\subset G_{CM}$ is \textbf{not } a group homomorphism. In
fact we have,
\begin{equation}
\pi_{P}\left[  \left(  w, c\right)  \cdot\left( w^{\prime},
c^{\prime}\right) \right]  -\pi_{P}\left(  w, c\right)
\cdot\pi_{P}\left(  w^{\prime}, c^{\prime
}\right)  =\Gamma_{P}\left(  w, w^{\prime}\right)  \label{e.h7.4}%
\end{equation}
where%
\begin{equation}
\Gamma_{P}\left(  w, w^{\prime}\right)  =\frac{1}{2}\left(
0,\omega\left( w, w^{\prime}\right)  -\omega\left(
Pw,Pw^{\prime}\right)  \right)
\label{e.h7.5}%
\end{equation}
So unless $\omega$ is \textquotedblleft supported\textquotedblright\
on the range of $P$,  $\pi_{P}$ is not a group homomorphism. Since,
$\left( w,b\right)  +\left(  0,c\right)  =\left(  w,b\right)
\cdot\left(  0, c\right)$ for all $w\in W$ and $b, c\in\mathbf{C}$,
we may also write equation \ref{e.h7.4} as
\begin{equation}
\pi_{P}\left[  \left(  w, c\right)  \cdot\left( w^{\prime},
c^{\prime}\right) \right]  =\pi_{P}\left(  w, c\right)
\cdot\pi_{P}\left(  w^{\prime},c^{\prime
}\right)  \cdot\Gamma_{P}\left(  w, w^{\prime}\right). \label{e.h7.6}%
\end{equation}

\end{rem}

\begin{lem}
\label{l.h7.6}To each $k:=\left(  A, a\right) \in\mathfrak{g}_{CM}$,
$g=\left(  w, c\right)  \in G$,  and $P\in\operatorname*{Proj}\left(
W\right)$,  let
\begin{equation}
k^{P}\left(  g\right)  =k^{P}\left(  w, c\right)
:=\pi_{P}k+\Gamma_{P}\left(
w,A\right)  \in\mathfrak{g}_{P} \label{e.h7.7}%
\end{equation}
where $\Gamma_{P}$ is defined in Eq. (\ref{e.h7.5}) above. If $u:G_{CM}%
\rightarrow\mathbb{C}$ is a holomorphic function and $g\in G$,  then%
\begin{equation}
\left(  \tilde{k}\left(  u\circ\pi_{P}\right)  \right)  \left(
g\right) =\left\langle Du\left(  \pi_{P}\left(  g\right) \right),
k^{P}\left(
g\right)  \right\rangle \label{e.h7.8}%
\end{equation}
or equivalently put,%
\begin{equation}
\left\langle \widehat{u\circ\pi_{P}}\left(  g\right),k\right\rangle
=\left\langle D\left(  u\circ\pi_{P}\right)  \left(  g\right)
,k\right\rangle =\left\langle Du\left(  \pi_{P}\left(  g\right)
\right), k^{P}\left(
g\right)  \right\rangle . \label{e.h7.9}%
\end{equation}

\end{lem}

\begin{proof}
By direct computation,%
\begin{align*}
\left(  \tilde{k}\left(  u\circ\pi_{P}\right)  \right)  \left(  g\right)   &
=\frac{d}{dt}\Big|_{0}u\left(  \pi_{P}\left(  g\cdot e^{tk}\right)  \right) \\
&  =\frac{d}{dt}\Big|_{0}\left\langle Du\left(  \pi_{P}\left(
g\right) \right), \left[  \pi_{P}\left(  g\right)
\right]^{-1}\cdot\pi_{P}\left( g\cdot e^{tk}\right)  \right\rangle
\end{align*}
where by Eq. (\ref{e.h7.6}),
\begin{align*}
\frac{d}{dt}\Big|_{0}\left(  \left[  \pi_{P}\left(  g\right)  \right]
^{-1}\cdot\pi_{P}\left(  g\cdot e^{tk}\right)  \right)   &  =\frac{d}%
{dt}\Big|_{0}\left(  P\left(  tA\right), a+\frac{1}{2}\omega\left(
w, tA\right)  -\omega\left(  Pw, tPA\right)  \right) \\
&  =\left(  PA, a+\frac{1}{2}\omega\left(  w, A\right)
-\omega\left(
Pw, PA\right)  \right) \\
&  =\pi_{P}k+\Gamma_{P}\left(  w, A\right).
\end{align*}

\end{proof}

\begin{nota}
\label{n.h7.7}Given $P\in\operatorname*{Proj}\left(  W\right)$ and
$k_{j}=\left(  A_{j}, c_{j}\right)  \in\mathfrak{g}_{CM}$,  let $K_{j}%
:=k_{j}^{P}: G_{CM}\rightarrow\mathfrak{g}_{CM}$ and $\kappa_{n}:G_{CM}%
\rightarrow\oplus_{j=1}^{n}\mathfrak{g}_{CM}^{\otimes j}$ be defined by
\begin{align}
\kappa_{n}  &  =\left(  \tilde{k}_{n}+K_{n}\otimes\right)  \left(  \tilde
{k}_{n-1}+K_{n-1}\otimes\right)  \dots\left(  \tilde{k}_{1}+K_{1}%
\otimes\right)  1\nonumber\\
&  =\left(  \tilde{k}_{n}+K_{n}\otimes\right)  \left(  \tilde{k}_{n-1}%
+K_{n-1}\otimes\right)  \dots\left(  \tilde{k}_{2}+K_{2}\otimes\right)  K_{1}.
\label{e.h7.10}%
\end{align}
In these expressions, $K_{j}\otimes$ denotes operation of left
tensor multiplication by $K_{j}$. \end{nota}

\begin{ex}
\label{ex.h7.8}The functions $\kappa_{n}$ are determined recursively by
$\kappa_{1}=K_{1}$ and then
\begin{equation}
\kappa_{n}=\left(  K_{n}\otimes+\tilde{k}_{n}\right)  \kappa_{n-1}%
=K_{n}\otimes\kappa_{n-1}+\tilde{k}_{n}\kappa_{n-1}\text{ for all }%
n\geqslant2. \label{e.h7.11}%
\end{equation}
The first four $\kappa_{n}$ are easily seen to be given by, $\kappa_{1}%
=K_{1}$,
\[ \kappa_{2}=K_{2}\otimes
K_{1}+\tilde{k}_{2}K_{1}=K_{2}\otimes K_{1}+\Gamma_{P}\left(
A_{2},A_{1}\right),
\]%
\begin{align*}
\kappa_{3}  &  =\left(  K_{3}\otimes+\tilde{k}_{3}\right)  \left(
K_{2}\otimes K_{1}+\Gamma_{P}\left(  A_{2},A_{1}\right)  \right) \\
&  =K_{3}\otimes K_{2}\otimes K_{1}+K_{3}\otimes\Gamma_{P}\left(  A_{2}%
,A_{1}\right)  +\Gamma_{P}\left(  A_{3},A_{2}\right)  \otimes K_{1}%
+K_{2}\otimes\Gamma_{P}\left(  A_{3},A_{1}\right),
\end{align*}
and%
\begin{align*}
\kappa_{4}  &  =K_{4}\otimes K_{3}\otimes K_{2}\otimes K_{1}\\
&  +\left(
\begin{array}
[c]{c}%
K_{4}\otimes K_{3}\otimes\Gamma_{P}\left(  A_{2}, A_{1}\right)  +K_{4}%
\otimes\Gamma_{P}\left(  A_{3}, A_{2}\right)  \otimes
K_{1}+K_{4}\otimes
K_{2}\otimes\Gamma_{P}\left(  A_{3}, A_{1}\right) \\
+\Gamma_{P}\left(  A_{4}, A_{3}\right)  \otimes K_{2}\otimes K_{1}+K_{3}%
\otimes\Gamma_{P}\left(  A_{4}, A_{2}\right)  \otimes
K_{1}+K_{3}\otimes K_{2}\otimes\Gamma_{P}\left(  A_{4}, A_{1}\right)
\end{array}
\right) \\
&  +\Gamma_{P}\left(  A_{4}, A_{3}\right)  \otimes\Gamma_{P}\left(
A_{2}, A_{1}\right)  +\Gamma_{P}\left(  A_{3}, A_{2}\right)
\otimes\Gamma_{P}\left( A_{4}, A_{1}\right)  +\Gamma_{P}\left(
A_{4}, A_{2}\right)  \otimes\Gamma_{P}\left(  A_{3}, A_{1}\right).
\end{align*}
At the end we will only use $\kappa_{n}$ evaluated at $\mathbf{e}\in
G_{CM}$. Evaluating the above expressions at $\mathbf{e}$ amounts to
replacing $K_{j}$ by $\pi_{P}k_{j}$ in all of the previous formulas.
\end{ex}

\begin{prop}
\label{p.h7.9}If $u\in\mathcal{H}\left(  G_{CM}\right)$,  then, with
the
setup in Notation \ref{n.h7.7}, we have%
\begin{equation}
\left\langle \widehat{u\circ\pi_{P}},k_{n}\otimes\dots\otimes k_{1}%
\right\rangle =\left\langle \hat{u}\circ\pi_{P},\kappa_{n}\right\rangle \text{
for any }n\in\mathbb{N}, \label{e.h7.12}%
\end{equation}
where both sides of this equation are holomorphic functions on
$G_{CM}$. \end{prop}

\begin{proof}
The proof is by induction with the case $n=1$ already completed via
Equation (\ref{e.h7.9}). To proceed with the induction argument,
suppose that Eq. (\ref{e.h7.12}) holds for some $n\in\mathbb{N}$.
Then by induction and the product rule
\begin{align}
\left\langle \widehat{u\circ\pi_{P}}, k_{n+1}\otimes
k_{n}\otimes\dots\otimes
k_{1}\right\rangle  &  =\tilde{k}_{n+1}\left\langle \widehat{u\circ\pi_{P}%
},k_{n+1}\otimes k_{n}\otimes\dots\otimes k_{1}\right\rangle \nonumber\\
&  =\tilde{k}_{n+1}\left\langle \hat{u}\circ\pi_{P},\kappa_{n}\right\rangle
\nonumber\\
&  =\left\langle
\hat{u}\circ\pi_{P},\tilde{k}_{n+1}\kappa_{n}\right\rangle
+\left\langle \tilde{k}_{n+1}\left[
\hat{u}\circ\pi_{P}\right],\kappa
_{n}\right\rangle . \label{e.h7.13}%
\end{align}
To evaluate $\tilde{k}_{n+1}\left[  \hat{u}\circ\pi_{P}\right]  $
let $v\in\mathbf{T}\left(  \mathfrak{g}_{CM}\right)  $ and let
$\tilde{v}$ denote the corresponding left invariant differential
operator on $G_{CM}$. Then
\begin{align}
\left\langle \tilde{k}_{n+1}\left[  \hat{u}\circ\pi_{P}\right],
v\right\rangle \left(  g\right)   &  =\left(
\tilde{k}_{n+1}\left\langle \left[  \hat{u}\circ\pi_{P}\right],
v\right\rangle \right)  \left(  g\right)
\nonumber\\
&  =\left(  \tilde{k}_{n+1}\left[  \left(  \tilde{v}u\right)  \circ\pi
_{P}\right]  \right)  \left(  g\right) \nonumber\\
&  =\left\langle D\left(  \tilde{v}u\right)  \left(  \pi_{P}\left(  g\right)
\right),k_{n+1}^{P}\left(  g\right)  \right\rangle \nonumber\\
&  =\left(  \widetilde{k_{n+1}^{P}\left(  g\right)  }\tilde{v}u\right)
\left(  \pi_{P}\left(  g\right)  \right) \nonumber\\
&  =\left\langle \hat{u}\left(  \pi_{P}\left(  g\right)  \right),k_{n+1}%
^{P}\left(  g\right)  \otimes v\right\rangle . \label{e.h7.14}%
\end{align}
Combining Eqs. (\ref{e.h7.13}) and (\ref{e.h7.14}) shows,%
\begin{align*}
\left\langle \widehat{u\circ\pi_{P}},k_{n+1}\otimes k_{n}\otimes\dots\otimes
k_{1}\right\rangle  &  =\left\langle \hat{u}\circ\pi_{P},\tilde{k}_{n+1}%
\kappa_{n}\right\rangle +\left\langle \hat{u}\circ\pi_{P},k_{n+1}^{P}%
\otimes\kappa_{n}\right\rangle \\
&  =\left\langle \hat{u}\circ\pi_{P},\tilde{k}_{n+1}\kappa_{n}+k_{n+1}%
^{P}\otimes\kappa_{n}\right\rangle =\left\langle \hat{u}\circ\pi_{P}%
,\kappa_{n+1}\right\rangle
\end{align*}
wherein we have used Eq. (\ref{e.h7.11}) for the last equality.
\end{proof}

The induction proof of the following lemma will be left to the reader with
Example \ref{ex.h7.8} as a guide.

\begin{lem}
\label{l.h7.10}Let $k_{j}=\left(  A_{j},c_{j}\right)
\in\mathfrak{g}_{CM}$ for $1\leqslant j\leqslant n$,  $\left\lfloor
\frac{n}{2}\right\rfloor =n/2$ if $n$ is even and $\left(
n-1\right)  /2$ if $n$ is odd, and $\kappa_{n}$ be as
in Eq. (\ref{e.h7.10}). Then%
\begin{equation}
\kappa_{n}\left(  \mathbf{e}\right)
=\pi_{P}k_{n}\otimes\dots\otimes\pi
_{P}k_{2}\otimes\pi_{P}k_{1}+R\left(  P:k_{n}, \dots, k_{1}\right),
\label{e.h7.15}%
\end{equation}
where
\begin{equation}
R\left(  P:k_{n},,\dots,k_{1}\right)  =\sum_{j=1}^{\left\lfloor \frac{n}%
{2}\right\rfloor }R_{j}\left(  P:k_{n},,\dots,k_{1}\right)  \label{e.h7.16}%
\end{equation}
with $R_{j}\left(  P:k_{1},\dots,k_{n}\right)  \in\mathfrak{g}_{CM}%
^{\otimes\left(  n-j\right)  }$. Each remainder term, $R_{j}\left(
P:k_{1},\dots,k_{n}\right)$,  is a linear combination (with
coefficients coming from $\left\{  \pm1,0\right\}  )$ of homogenous
tensors which are permutations of the indices and order of the terms
in the tensor product of
the form%
\begin{equation}
\Gamma_{P}\left(  A_{1},A_{2}\right)  \otimes\dots\otimes\Gamma_{P}\left(
A_{2j-1},A_{2j}\right)  \otimes k_{2j+1}\otimes\dots\otimes k_{n}.
\label{e.h7.17}%
\end{equation}

\end{lem}

\begin{prop}
\label{p.h7.11}Let $P_{N}\in\operatorname*{Proj}\left(  W\right)  $ and
$\pi_{N}:=\pi_{P_{N}}$ be as in Notation \ref{n.h1.1} and suppose that
$u\in\mathcal{H}\left(  G_{CM}\right)  $ satisfies $\left\Vert \hat{u}%
_{n}\left(  \mathbf{e}\right)  \right\Vert_{n}<\infty$ for all $n$.
Then
\begin{equation}
\lim_{N\rightarrow\infty}\left\Vert \hat{u}_{n}\left(  \mathbf{e}\right)
-\left[  \widehat{u\circ\pi_{N}}\left(  \mathbf{e}\right)  \right]
_{n}\right\Vert_{n}=0\text{ for }n=0,1,2,\dots.. \label{e.h7.18}%
\end{equation}

\end{prop}

\begin{proof}
To simplify notation, let $\alpha_{n}:=\hat{u}_{n}\left(
\mathbf{e}\right)  $ and $\alpha_{n}\left(  N\right)  :=\left[
\widehat{u\circ\pi_{N}}\left( \mathbf{e}\right)  \right]_{n}$. Let
$\Gamma$ be an orthonormal basis for
$\mathfrak{g}_{CM}$ of the form in Eq. (\ref{e.h5.8}) and let $\mathbf{k}%
:=\left(  k_{1},k_{2},\dots,k_{n}\right)  \in\Gamma^{n}$. Then%
\[
\left\langle \alpha-\alpha\left(  N\right),k_{1}\otimes\dots\otimes
k_{n}\right\rangle =\left\langle \alpha,k_{1}\otimes\dots\otimes
k_{n}-\pi _{N}k_{1}\otimes\dots\otimes\pi_{N}k_{n}\right\rangle
+\left\langle \alpha,R\left(  P_{N}:\mathbf{k}\right)  \right\rangle
\]
where $R\left(  P_{N}:\mathbf{k}\right)  $ is as in Lemma \ref{l.h7.10}.
Therefore, $\left\Vert \alpha_{n}-\alpha_{n}\left(  N\right)  \right\Vert
_{n}\leqslant C_{N}+D_{N}$ where%
\begin{align*}
C_{N}  &  :=\sqrt{\sum_{\mathbf{k}\in\Gamma^{n}}\left\vert \left\langle
\alpha,R\left(  P_{N}:\mathbf{k}\right)  \right\rangle \right\vert ^{2}}\text{
and}\\
D_{N}  &  :=\sqrt{\sum_{\mathbf{k}\in\Gamma^{n}}\left\vert \left\langle
\alpha_{n},k_{1}\otimes\dots\otimes k_{n}-\pi_{N}k_{1}\otimes\dots\otimes
\pi_{N}k_{n}\right\rangle \right\vert ^{2}}.
\end{align*}
We will complete the proof by showing that, $\lim_{N\rightarrow\infty}%
C_{N}=0=\lim_{N\rightarrow\infty}D_{N}$. To estimate $C_{N}$,  use
Lemma \ref{l.h7.10} and the triangle inequality for $\ell_{2}\left(
\Gamma^{n}\right)  $ to find,
\[
C_{N}=\sqrt{\sum_{\mathbf{k}\in\Gamma^{n}}\left\vert \sum_{j=1}^{\left\lfloor
\frac{n}{2}\right\rfloor }\left\langle \alpha,R_{j}\left(  P_{N}%
:\mathbf{k}\right)  \right\rangle \right\vert ^{2}}\leqslant\sum
_{j=1}^{\left\lfloor \frac{n}{2}\right\rfloor }\sqrt{\sum_{\mathbf{k}\in
\Gamma^{n}}\left\vert \left\langle \alpha,R_{j}\left(  P_{N}:\mathbf{k}%
\right)  \right\rangle \right\vert ^{2}}.
\]
But $\sum_{\mathbf{k}\in\Gamma^{n}}\left\vert \left\langle \alpha,R_{j}\left(
P_{N}:\mathbf{k}\right)  \right\rangle \right\vert ^{2}$ is bounded by a sum
of terms (the number of these terms depends only on $j$ and $n$ and
\textbf{not }$N)$ of which a typical term (see Eq. (\ref{e.h7.17})) is;%
\begin{equation}
\sum_{\mathbf{k}\in\Gamma^{n}}\left\vert \left\langle \alpha_{n-j}%
,\Gamma_{P_{N}}\left(  A_{1},A_{2}\right)  \otimes\dots\otimes\Gamma_{P_{N}%
}\left(  A_{2j-1},A_{2j}\right)  \otimes k_{2j+1}\otimes\dots\otimes
k_{n}\right\rangle \right\vert ^{2}. \label{e.h7.19}%
\end{equation}
The sum in Eq. (\ref{e.h7.19}) may be estimated by,%
\[
\left\Vert
\alpha_{n-j}\right\Vert_{n-j}\sum_{l_{1},\dots,l_{2j}=1}^{\infty
}\left\Vert \Gamma_{P_{N}}\left(  e_{l_{1}},e_{l_{2}}\right)
\right\Vert
_{\mathfrak{g}_{CM}}^{2}\dots\left\Vert \Gamma_{P_{N}}\left(  e_{l_{2j-1}%
},e_{l_{2j}}\right)  \right\Vert_{\mathfrak{g}_{CM}}^{2}=\left\Vert
\alpha_{n-j}\right\Vert_{n-j}^{2}\varepsilon_{N}^{j},
\]
where%
\begin{align*}
\varepsilon_{N}  &  =\frac{1}{4}\sum_{k,l=1}^{\infty}\left\Vert \omega\left(
e_{k},e_{l}\right)  -\omega\left(  P_{N}e_{k},P_{N}e_{l}\right)  \right\Vert
_{\mathbf{C}}^{2}\\
&  =\frac{1}{4}\sum_{\max\left(  k,l\right)  >N}^{\infty}\left\Vert
\omega\left(  e_{k},e_{l}\right)  -\omega\left(  P_{N}e_{k},P_{N}e_{l}\right)
\right\Vert_{\mathbf{C}}^{2}\\
&  \leqslant\frac{1}{2}\sum_{\max\left(  k,l\right)
>N}^{\infty}\left\Vert \omega\left(  e_{k},e_{l}\right)
\right\Vert_{\mathbf{C}}^{2}\rightarrow 0\text{ and
}N\rightarrow\infty.
\end{align*}
Thus we have shown $\lim_{N\rightarrow\infty}C_{N}=0$

For $N\in\mathbb{N}$,  let $\Gamma_{N}=\left\{  \left(
0,f_{j}\right)
\right\}_{j=1}^{d}\cup\left\{  \left(  e_{j},0\right)  \right\}_{j=1}%
^{N}$. Since $k_{1}\otimes\dots\otimes
k_{n}=\pi_{N}k_{1}\otimes\dots \otimes\pi_{N}k_{n}$ if
$\mathbf{k}:=\left(  k_{1},k_{2},\dots,k_{n}\right)
\in\Gamma_{N}^{n}$,  it follows that%
\begin{align}
D_{N}^{2}  &  =\sum_{\mathbf{k}\in\Gamma^{n}\setminus\Gamma_{N}^{n}}\left\vert
\left\langle \alpha_{n},k_{1}\otimes\dots\otimes k_{n}-\pi_{N}k_{1}%
\otimes\dots\otimes\pi_{N}k_{n}\right\rangle \right\vert ^{2}\nonumber\\
&  \leqslant2\sum_{\mathbf{k}\in\Gamma^{n}\setminus\Gamma_{N}^{n}}\left\vert
\left\langle \alpha_{n},k_{1}\otimes\dots\otimes k_{n}\right\rangle
\right\vert ^{2}. \label{e.h7.20}%
\end{align}
Because
\[
\sum_{\mathbf{k}\in\Gamma^{n}}\left\vert \left\langle \alpha_{n},k_{1}%
\otimes\dots\otimes k_{n}\right\rangle \right\vert ^{2}=\left\Vert
\alpha _{n}\right\Vert_{n}^{2}<\infty
\]
and $\Gamma_{N}^{n}\uparrow\Gamma_{N}$ as $N\uparrow\infty$,  the
sum in Eq. (\ref{e.h7.20}) tends to zero as $N\rightarrow\infty$.
Thus $\lim_{N\rightarrow\infty}D_{N}=0$ and the proof is complete.
\end{proof}

\begin{prop}
\label{p.h7.12}If $u\in\mathcal{H}_{T,\text{fin}}^{2}\left(
G_{CM}\right)  $ and $u_{N}:=u\circ\pi_{N}$ as in Eq.
(\ref{e.h7.3}), then $u_{N}\in \mathcal{P}$ and
$u_{N}|_{G_{CM}}\rightarrow u$ in $\mathcal{H}_{T}^{2}\left(
G_{CM}\right)$. \end{prop}

\begin{proof}
Suppose $m\in\mathbb{N}$ is chosen so that $\hat{u}_{n}\left(  \mathbf{e}%
\right)  =0$ if $n>m$. According to Proposition \ref{p.h7.9},
\[
\left\langle \hat{u}_{N}\left(
\mathbf{e}\right),k_{n}\otimes\dots\otimes k_{1}\right\rangle
=\left\langle \hat{u}\left(  \mathbf{e}\right),\kappa _{n}\left(
\mathbf{e}\right)  \right\rangle
\]
where $\kappa_{n}\left(  \mathbf{e}\right)
\in\bigoplus_{j=1}^{\left\lfloor \frac{n}{2}\right\rfloor
}\mathfrak{g}_{CM}^{\otimes\left(  n-j\right)  }$. From this it
follows that $\left\langle \hat{u}_{N}\left(  \mathbf{e}\right)
,k_{n}\otimes\dots\otimes k_{1}\right\rangle =0$ if
$n\geqslant2m+2$. Therefore, $u_{N}$ restricted to
$P_{N}H\times\mathbf{C}$ is a holomorphic polynomial and since
$u_{N}=u_{N}|_{P_{N}H\times\mathbf{C}}\circ\pi_{N}$,  it follows
that $u_{N}\in\mathcal{P}$. Moreover,
\[
\lim_{N\rightarrow\infty}\left\Vert \hat{u}\left(  \mathbf{e}\right)
-\hat {u}_{N}\left(  \mathbf{e}\right)  \right\Vert_{J_{T}^{0}\left(
\mathfrak{g}_{CM}\right)  }^{2}=\lim_{N\rightarrow\infty}\sum_{n=0}%
^{2m+2}\frac{T^{n}}{n!}\left\Vert \hat{u}_{n}\left(
\mathbf{e}\right) -\left[  \hat{u}_{N}\left(  \mathbf{e}\right)
\right]_{n}\right\Vert_{n}^{2}=0,
\]
wherein we have used Proposition \ref{p.h7.11} to conclude $\lim
_{N\rightarrow\infty}\left\Vert \hat{u}_{n}\left(  \mathbf{e}\right)
-\left[ \hat{u}_{N}\left(  \mathbf{e}\right)  \right]
_{n}\right\Vert_{n}=0$ for all $n$. It then follows by the Taylor
isomorphism Theorem \ref{t.h6.10} that
$\lim_{N\rightarrow\infty}\left\Vert u-u_{N}\right\Vert_{\mathcal{H}_{T}%
^{2}\left(  G_{CM}\right)  }=0$. \end{proof}

\section{The skeleton isomorphism\label{s.h8}}

This section is devoted to the proof of the skeleton Theorem \ref{t.h1.8}. Let
us begin by gathering together a couple of results that we have already proved.

\begin{prop}
\label{p.h8.1}If $f:G\rightarrow\mathbb{C}$ is a continuous function such that
$f|_{G_{CM}}$ is holomorphic, then%
\begin{equation}
\left\Vert f\right\Vert_{L^{2}\left(  \nu_{T}\right)
}\leqslant\left\Vert f|_{G_{CM}}\right\Vert
_{\mathcal{H}_{T}^{2}\left(  G_{CM}\right) }=\left\Vert
\hat{f}\left(  \mathbf{e}\right)  \right\Vert_{J_{T}^{0}\left(
\mathfrak{g}_{CM}\right)  }. \label{e.h8.1}%
\end{equation}
If $\left\Vert f|_{G_{CM}}\right\Vert_{\mathcal{H}_{T}^{2}\left(
G_{CM}\right)  }^{2}<\infty$,  then $S_{T}f=f$ and $f$ satisfies the
Gaussian pointwise bounds in Eq. (\ref{e.h6.19}). (See Corollary
\ref{c.h8.3} for a more sophisticated version of this proposition.)
\end{prop}

\begin{proof}
See Theorems \ref{t.h5.9} and \ref{t.h6.11}.
\end{proof}

\begin{lem}
\label{l.h8.2}Let $f:G\rightarrow\mathbb{C}$ be a continuous function such
that $f|_{G_{CM}}$ is holomorphic and let $\delta>0$ be as in Theorem
\ref{t.h4.11}. If there exists an $\varepsilon\in\left(  0,\delta\right)  $
such that $\left\vert f\left(  \cdot\right)  \right\vert \leqslant
Ce^{\varepsilon\rho^{2}\left(  \cdot\right)  /\left(  2T\right)  }$ on $G$,  then%
\begin{equation}
\left\Vert f\right\Vert_{L^{2}\left(  \nu_{T}\right)  }=\left\Vert
f\right\Vert_{\mathcal{H}_{T}^{2}\left(  G_{CM}\right) }=\left\Vert
\hat
{f}\left(  \mathbf{e}\right)  \right\Vert_{J_{T}^{0}\left(  \mathfrak{g}%
_{CM}\right)  }<\infty. \label{e.h8.2}%
\end{equation}
(It will be shown in Corollary \ref{c.h8.4} that $f$ is actually in
$\mathcal{H}_{T}^{2}\left(  G\right)$.) In particular, Eq.
(\ref{e.h8.2}) holds for all $f\in\mathcal{P}$. \end{lem}

\begin{proof}
Let $\left\{  P_{n}\right\}
_{n=1}^{\infty}\subset\operatorname*{Proj}\left( W\right)  $ be a
sequence such that $P_{n}|_{\mathfrak{g}_{CM}}\uparrow
I_{\mathfrak{g}_{CM}}$ as $n\rightarrow\infty$. Then, by Lemma
\ref{l.h6.5} and Proposition \ref{p.h4.12} with $h=0$,  \[
\infty>\left\Vert f\right\Vert_{L^{2}\left(  \nu_{T}\right)  }=\lim
_{n\rightarrow\infty}\left\Vert f\right\Vert_{L^{2}\left(
G_{P_{n}}\nu_{T}^{P_{n}}\right)  }=\left\Vert f\right\Vert
_{\mathcal{H}_{T}^{2}\left( G_{CM}\right)  }=\left\Vert
\hat{f}\left(  \mathbf{e}\right)  \right\Vert_{J_{T}^{0}\left(
\mathfrak{g}_{CM}\right)  }.
\]
\end{proof}

We are now ready to complete the proof of the Skeleton isomorphism Theorem
\ref{t.h1.8}.

\subsection{Proof of Theorem \ref{t.h1.8}\label{s.h8.1}}

\begin{proof}
By Corollary \ref{c.h5.10}, $S_{T}f=f|_{G_{CM}}$ for all $f\in\mathcal{P}$ and
hence by Lemma \ref{l.h8.2}, $\left\Vert S_{T}f\right\Vert_{\mathcal{H}%
_{T}^{2}\left(  G_{CM}\right)  }=\left\Vert f\right\Vert
_{L^{2}\left( \nu_{T}\right)  }$. It therefore follows that
$S_{T}|_{\mathcal{P}}$ extends uniquely to an isometry,
$\bar{S}_{T}$,  from $\mathcal{H}_{T}^{2}\left( G\right)  $ to
$\mathcal{H}_{T}^{2}\left(  G_{CM}\right)$ such that $\bar
{S}_{T}\left(  \mathcal{P}\right)  =\mathcal{P}_{CM}$. Since
$\bar{S}_{T}$ is isometric and $\mathcal{P}_{CM}$ is dense in
$\mathcal{H}_{T}^{2}\left( G_{CM}\right)$, it follows that
$\bar{S}_{T}$ is surjective, i.e. $\bar
{S}_{T}:\mathcal{H}_{T}^{2}\left(  G\right)  \rightarrow\mathcal{H}_{T}%
^{2}\left(  G_{CM}\right)$ is a unitary map. To finish the proof we
only need to show $S_{T}f=\bar{S}_{T}f$ for all
$f\in\mathcal{H}_{T}^{2}\left( G\right)$. Let $p_{n}\in\mathcal{P}$
such that $p_{n}\rightarrow f$ in $L^{2}\left( \nu_{T}\right)$. Then
$p_{n}=S_{T}p_{n}\rightarrow\bar{S}_{T}f$ in
$\mathcal{H}_{T}^{2}\left(  G_{CM}\right)  $ and hence by the
Gaussian pointwise bounds in Eq. (\ref{e.h6.19}),
$\bar{S}_{T}f\left(  g\right) =\lim_{n\rightarrow\infty}p_{n}\left(
g\right)  $ for all $g\in G_{CM}$. Similarly, using the Gaussian
bounds in Corollary \ref{c.h4.8}, it follows that
\begin{align}
\left\vert S_{T}f\left(  g\right)  -p_{n}\left(  g\right)  \right\vert  &
=\left\vert S_{T}\left(  f-p_{n}\right)  \left(  g\right)  \right\vert
\nonumber\\
&  \leqslant\left\Vert f-p_{n}\right\Vert_{L^{2}\left(  \nu_{T}\right)  }%
\exp\left(  \frac{c\left(  k\left(  \omega\right)  T/2\right)  }{T}d_{G_{CM}%
}^{2}\left(  \mathbf{e},g\right)  \right)  \label{e.h8.3}%
\end{align}
and hence we also have, $S_{T}f\left(  g\right)
=\lim_{n\rightarrow\infty }p_{n}\left(  g\right)$ for all $g\in
G_{CM}$. Therefore, $S_{T}f=\bar {S}_{T}f$ as was to be shown.
\end{proof}

\begin{cor}
\label{c.h8.3}If $f:G\rightarrow\mathbb{C}$ is a continuous function
such that $f|_{G_{CM}}\in\mathcal{H}_{T}^{2}\left(  G_{CM}\right) $,
then $f\in\mathcal{H}_{T}^{2}\left(  G\right)$,
$S_{T}f=f|_{G_{CM}}$,  and $\left\Vert f\right\Vert_{L^{2}\left(
\nu_{T}\right)  }=\left\Vert f\right\Vert
_{\mathcal{H}_{T}^{2}\left(  G_{CM}\right)  }$. \end{cor}

\begin{proof}
By Proposition \ref{p.h8.1} we already know that
$S_{T}f=f|_{G_{CM}}$. By Theorem \ref{t.h1.8}, there exists
$u\in\mathcal{H}_{T}^{2}\left(  G\right)  $ such that
$f|_{G_{CM}}=S_{T}u$. Let $p_{n}\in\mathcal{P}$ be chosen so that
$p_{n}\rightarrow u$ in $L^{2}\left(  \nu_{T}\right)  $ and hence
$p_{n}|_{G_{CM}}=S_{T}p_{n}\rightarrow S_{T}u=S_{T}f$ in $\mathcal{H}_{T}%
^{2}\left(  G_{CM}\right)  $ as $n\rightarrow\infty$. Hence it
follows from
Proposition \ref{p.h8.1} that%
\[
\left\Vert f-p_{n}\right\Vert_{L^{2}\left(  \nu_{T}\right)
}\leqslant
\left\Vert \left(  f-p_{n}\right)  |_{G_{CM}}\right\Vert_{\mathcal{H}_{T}%
^{2}\left(  G_{CM}\right)  }=\left\Vert S_{T}\left(  f-p_{n}\right)
\right\Vert_{\mathcal{H}_{T}^{2}\left(  G_{CM}\right)  },
\]
and therefore, $\lim_{n\rightarrow\infty}\left\Vert
f-p_{n}\right\Vert_{L^{2}\left(  \nu_{T}\right)  }=0$,  i.e.
$p_{n}\rightarrow f$ in $L^{2}\left(  \nu_{T}\right)$. Since
$p_{n}\rightarrow u$ in $L^{2}\left(
\nu_{T}\right)  $ as well, we may conclude that $f=u\in\mathcal{H}_{T}%
^{2}\left(  G\right)$. \end{proof}

\begin{cor}
\label{c.h8.4}Suppose that $f:G\rightarrow\mathbb{C}$ is a continuous function
such that $\left\vert f\right\vert \leqslant Ce^{\varepsilon\rho^{2}/\left(
2T\right)  }$ and $f|_{G_{CM}}$ is holomorphic, then $f\in\mathcal{H}_{T}%
^{2}\left(  G\right)  $ and $S_{T}f=f$. \end{cor}

\begin{proof}
This is a consequence of Lemma \ref{l.h8.2} and Corollary \ref{c.h8.3}.
\end{proof}

\section{The holomorphic chaos expansion\label{s.h9}}

This section is devoted to the proof of the holomorphic chaos expansion
Theorem \ref{t.h1.9} (or equivalently Theorem \ref{t.h9.10}). Before going to
the proof we will develop the machinery necessary in order to properly define
the right side of Eq. (\ref{e.h1.8}).

\subsection{Generalities about multiple It\^{o} integrals\label{s.h9.1}}

Let $\left(  \mathbb{H},\mathbb{W}\right)  $ be a complex abstract Wiener
space. Analogous to the notation used in Subsection \ref{s.h6.1} we will
denote the norm on $\mathbb{H}^{\ast\otimes n}$ by $\Vert\cdot\Vert_{n}$.

\begin{nota}
\label{n.h9.1}For $\alpha\in\mathbb{H}^{\ast\otimes n}$ and $P\in
\operatorname*{Proj}\left(  \mathbb{W}\right)$,  let
$\alpha_{P}:=\alpha\circ P^{\otimes n}\in\mathbb{H}^{\ast\otimes
n}$. \end{nota}

\begin{prop}
\label{p.h9.2}Let $n\in\mathbb{N}$ and $\alpha\in\mathbb{H}^{\ast\otimes n}$
and $P_{k}\in\operatorname*{Proj}\left(  \mathbb{W}\right)  $ with
$P_{k}|_{\mathbb{H}}\uparrow I|_{\mathbb{H}}$. Then $\alpha_{P_{k}}%
\rightarrow\alpha$ in $\mathbb{H}^{\ast\otimes n}$. \end{prop}

\begin{proof}
Let $\Lambda:=\cup_{k}\Lambda_{k}$ be an orthonormal basis for $\mathbb{H}$
where $\Lambda_{k}$ is chosen to be an orthonormal basis for
$\operatorname*{Ran}\left(  P_{k}\right)  $ such that $\Lambda_{k}%
\subset\Lambda_{k+1}$ for all $k$. Since $P_{k}u=u$ or $P_{k}u=0$
for all
$u\in\Lambda$ and $k\in\mathbb{N}$,  we have%
\[
\left\vert \left\langle \alpha,u_{1}\otimes\dots\otimes u_{n}-P_{k}%
u_{1}\otimes\dots\otimes P_{k}u_{n}\right\rangle \right\vert ^{2}%
\leq\left\vert \left\langle \alpha,u_{1}\otimes\dots\otimes u_{n}\right\rangle
\right\vert ^{2}%
\]
where%
\[
\sum_{u_{1},\dots,u_{n}\in\Lambda}\left\vert \left\langle \alpha,u_{1}%
\otimes\dots\otimes u_{n}\right\rangle \right\vert ^{2}=\left\Vert
\alpha\right\Vert_{n}^{2}<\infty.
\]
An application of the dominated convergence theorem then implies,
\begin{align*}
\lim_{k\rightarrow\infty}\left\Vert
\alpha-\alpha_{P_{k}}\right\Vert_{n}^{2} &
=\lim_{k\rightarrow\infty}\sum_{u_{1},\dots,u_{n}\in\Lambda}\left\vert
\left\langle \alpha,u_{1}\otimes\dots\otimes u_{n}-P_{k}u_{1}\otimes
\dots\otimes P_{k}u_{n}\right\rangle \right\vert ^{2}\\
&  =\sum_{u_{1},\dots,u_{n}\in\Lambda}\lim_{k\rightarrow\infty}\left\vert
\left\langle \alpha,u_{1}\otimes\dots\otimes u_{n}-P_{k}u_{1}\otimes
\dots\otimes P_{k}u_{n}\right\rangle \right\vert ^{2}=0.
\end{align*}

\end{proof}

\begin{lem}
\label{l.h9.3}Suppose that $\left\{  b\left(  t\right)
\right\}_{t\geq0}$
is a $\mathbb{W}$--valued Brownian motion normalized by%
\begin{equation}
\mathbb{E}\left[  \ell_{1}\left(  b\left(  t\right)  \right)
\ell_{2}\left( b\left(  s\right)  \right)  \right]
=\frac{1}{2}s\wedge t\left(  \ell
_{1},\ell_{2}\right)_{\mathbb{H}_{\operatorname{Re}}^{\ast}}\text{
for all
}\ell_{1},\ell_{2}\in\mathbb{W}_{\operatorname{Re}}^{\ast}. \label{e.h9.1}%
\end{equation}
If $P\in\operatorname*{Proj}\left(  \mathbb{W}\right)$,  $T>0$, and
$\left\{ f_{s}\right\}_{s\geq0}$ is a $\left( P\mathbb{H}\right)
^{\ast}$--valued continuous adapted process, such that
$\mathbb{E}\int_{0}^{T}\left\vert f_{s}\right\vert_{\left(
P\mathbb{H}\right)^{\ast}}^{2}ds<\infty$,  then
\begin{equation}
\mathbb{E}\left\vert \int_{0}^{T}\left\langle f_{s},d\left(  Pb\right)
\left(  s\right)  \right\rangle \right\vert ^{2}=\int_{0}^{T}\mathbb{E}%
\left\vert f_{s}\right\vert_{\left(  P\mathbb{H}\right)
^{\ast}}^{2}ds.
\label{e.h9.2}%
\end{equation}

\end{lem}

\begin{proof}
Let $\left\{  e_{j}\right\}_{j=1}^{d}$ be an orthonormal basis for
$P\mathbb{H}$ and write
\[
Pb\left(  s\right)  =\sum_{j=1}^{d}\left[  X_{j}\left(  s\right)  e_{j}%
+Y_{j}\left(  s\right)  ie_{j}\right]
\]
where $X_{j}\left(  s\right)  =\operatorname{Re}\left(  Pb\left(
s\right) ,e_{j}\right)  $ and $Y_{j}\left(  s\right)
=\operatorname{Im}\left( Pb\left(  s\right),e_{j}\right)$. From the
normalization in Eq. (\ref{e.h9.1}) it follows that $\left\{
\sqrt{2}X_{j},\sqrt{2}Y_{j}\right\}_{j=1}^{d}$ is a sequence of
independent standard Brownian motions, and
therefore the quadratic covariations of these processes are given by:%
\begin{equation}
dX_{j}dY_{k}=0\text{ and }dX_{j}dX_{k}=dY_{j}dY_{k}=\frac{1}{2}\delta
_{jk}dt\text{ for all }j,k=1,\dots,d. \label{e.h9.3}%
\end{equation}
Using Eq. (\ref{e.h9.3}) along with the identity,%
\begin{equation}
\left\langle f_{s},d\left(  Pb\right)  \left(  s\right)  \right\rangle
=\sum_{j=1}^{d}\left[  \left\langle f_{s},e_{j}\right\rangle dX_{j}\left(
s\right)  +\left\langle f_{s},ie_{j}\right\rangle dY_{j}\left(  s\right)
\right], \label{e.h9.4}%
\end{equation}
it follows by the basic isometry property of the stochastic integral that%
\begin{align*}
\mathbb{E}\left\vert \int_{0}^{T}\left\langle f_{s},d\left(  Pb\right)
\left(  s\right)  \right\rangle \right\vert ^{2}  &  =\frac{1}{2}\sum
_{j=1}^{d}\mathbb{E}\left[  \int_{0}^{T}\left\vert \left\langle f_{s}%
,e_{j}\right\rangle \right\vert ^{2}ds+\int_{0}^{T}\left\vert \left\langle
f_{s},ie_{j}\right\rangle \right\vert ^{2}ds\right] \\
&  =\mathbb{E}\int_{0}^{T}\sum_{j=1}^{d}\left\vert \left\langle f_{s}%
,e_{j}\right\rangle \right\vert
^{2}ds=\int_{0}^{T}\mathbb{E}\left\vert f_{s}\right\vert_{\left(
P\mathbb{H}\right)^{\ast}}^{2}ds.
\end{align*}

\end{proof}

\begin{df}
\label{d.h9.4}For $P\in\operatorname*{Proj}\left(  \mathbb{W}\right)$,  $n\in\mathbb{N}$,  and $T>0$,  let%
\[
M_{n}^{P}\left(  T\right)  :=\int_{0\leq s_{1}\leq
s_{2}\leq\dots\leq s_{n}\leq T}dPb\left(  s_{1}\right)  \otimes
dPb\left(  s_{2}\right) \otimes\dots\otimes dPb\left(  s_{n}\right).
\]
Alternatively put, $M_{0}^{P}\left(  T\right)  \equiv1$ and
$M_{n}^{P}\left( t\right)  \in\left(  P\mathbb{H}\right)^{\otimes
n}$ is defined inductively
by%
\begin{equation}
M_{n}^{P}\left(  t\right)  :=\int_{0}^{t}M_{n-1}^{P}\left(  s\right)  \otimes
dPb\left(  s\right)  \text{ for all }t\geq0. \label{e.h9.5}%
\end{equation}

\end{df}

\begin{cor}
\label{c.h9.5}Suppose that $T>0$,  $\alpha\in\mathbb{H}^{\ast\otimes
n}$,  and $P\in\operatorname*{Proj}\left(  \mathbb{W}\right)$, then
$\left\langle \alpha,M_{n}^{P}\left(  T\right)  \right\rangle $ is a
square integrable
random variable and%
\[
\mathbb{E}\left\vert \left\langle \alpha,M_{n}^{P}\left(  T\right)
\right\rangle \right\vert ^{2}=\frac{T^{n}}{n!}\left\Vert \alpha
_{P}\right\Vert_{n}^{2}.
\]

\end{cor}

\begin{proof}
The proof is easily carried out by induction with the case $n=1$ following
directly from Lemma \ref{l.h9.3}. Similarly from Lemma \ref{l.h9.3}, Eq.
(\ref{e.h9.5}), and induction we have%
\begin{align*}
\mathbb{E}\left\vert \tilde{\alpha}_{P}\right\vert ^{2}  &  =\mathbb{E}%
\left\vert \int_{0}^{T}\left\langle \alpha,M_{n-1}^{P}\left(  s\right)
\otimes dPb\left(  s\right)  \right\rangle \right\vert ^{2}\\
&  =\int_{0}^{T}\sum_{j=1}^{d}\mathbb{E}\left\vert \left\langle \alpha
,M_{n-1}^{P}\left(  s\right)  \otimes e_{j}\right\rangle \right\vert ^{2}ds\\
&  =\sum_{j=1}^{d}\int_{0}^{T}\frac{s^{n-1}}{\left(  n-1\right)  !}\left\Vert
\left\langle \alpha,\left(  \cdot\right)  \otimes e_{j}\right\rangle
\right\Vert_{n-1}^{2}ds=\frac{T^{n}}{n!}\left\Vert \alpha\right\Vert_{n}%
^{2}.
\end{align*}

\end{proof}

\begin{nota}
\label{n.h9.6}We now fix $T>0$ and for $P\in\operatorname*{Proj}\left(
\mathbb{W}\right)$,  let $\tilde{\alpha}_{P}=\left\langle \alpha,M_{n}%
^{P}\left(  T\right)  \right\rangle $,  i.e.%
\[
\tilde{\alpha}_{P}=\left\langle \alpha,\int_{0\leq s_{1}\leq s_{2}\leq
\dots\leq s_{n}\leq T}dPb\left(  s_{1}\right)  \otimes dPb\left(
s_{2}\right)  \otimes\dots\otimes dPb\left(  s_{n}\right)  \right\rangle .
\]

\end{nota}

\begin{lem}
\label{l.h9.7}If $P,Q\in\operatorname*{Proj}\left(  \mathbb{W}\right)$,  then%
\[
\left\Vert \tilde{\alpha}_{P}-\tilde{\alpha}_{Q}\right\Vert_{L^{2}}%
^{2}:=\mathbb{E}\left\vert
\tilde{\alpha}_{P}-\tilde{\alpha}_{Q}\right\vert
^{2}=\frac{T^{n}}{n!}\left\Vert
\alpha_{P}-\alpha_{Q}\right\Vert_{n}^{2}.
\]

\end{lem}

\begin{proof}
Let $R\in\operatorname*{Proj}\left(  \mathbb{W}\right)  $ be the orthogonal
projection onto $\operatorname*{Ran}\left(  P\right)  +\operatorname*{Ran}%
\left(  Q\right)$. We then have $\left(  \alpha_{P}\right)
_{R}=\alpha_{P}$ and $\left(  \alpha_{Q}\right)_{R}=\alpha_{Q}$ and
therefore, by Corollary
\ref{c.h9.5},%
\begin{align*}
\mathbb{E}\left\vert
\tilde{\alpha}_{P}-\tilde{\alpha}_{Q}\right\vert ^{2}  &
=\mathbb{E}\left\vert \left(
\alpha_{P}\right)_{R}^{\symbol{126}}-\left( \alpha_{Q}\right)
_{R}^{\symbol{126}}\right\vert ^{2}=\mathbb{E}\left\vert
\left(  \alpha_{P}-\alpha_{P}\right)_{R}^{\symbol{126}}\right\vert ^{2}\\
&  =\frac{T^{n}}{n!}\left\Vert \left(  \alpha_{P}-\alpha_{P}\right)
_{R}\right\Vert_{n}^{2}=\frac{T^{n}}{n!}\left\Vert \alpha_{P}-\alpha
_{P}\right\Vert_{n}^{2}.
\end{align*}

\end{proof}

\begin{prop}
\label{p.h9.8}Let $\alpha\in\mathbb{H}^{\ast\otimes n}$ and $P_{k}%
\in\operatorname*{Proj}\left(  \mathbb{W}\right)  $ with $P_{k}|_{\mathbb{H}%
}\uparrow I|_{\mathbb{H}}$,  then $\left\{
\tilde{\alpha}_{P_{k}}\right\}_{k=1}^{\infty}$ is an
$L^{2}$--convergent series. We denote the limit by $\tilde{\alpha}$.
This limit is independent of the choice of orthogonal projections
used in constructing $\tilde{\alpha}$. \end{prop}

\begin{proof}
For $k,l\in\mathbb{N}$,  by Lemma \ref{l.h9.7},
\[
\left\Vert \tilde{\alpha}_{P_{l}}-\tilde{\alpha}_{P_{k}}\right\Vert_{L^{2}%
}=\left\Vert
\alpha_{P_{l}}-\alpha_{P_{k}}\right\Vert_{n}\rightarrow0\text{ as
}l,k\rightarrow\infty,
\]
because, as we have already seen, $\alpha_{P_{l}}\rightarrow\alpha$
in $\mathbb{H}^{\ast\otimes n}$. Therefore $\tilde{\alpha}:=L^{2}$--
$\lim_{k\rightarrow\infty}\tilde{\alpha}_{P_{k}}$ exists.

Now suppose that $Q_{l}\in\operatorname*{Proj}\left(
\mathbb{W}\right)  $ also increases to $I|_{\mathbb{H}}$. By Lemma
\ref{l.h9.7} and the fact that both $\alpha_{P_{l}}$ and
$\alpha_{Q_{l}}$ converge to $\alpha$ in $\mathbb{H}^{\ast\otimes
n}$,  we have
\[
\left\Vert \tilde{\alpha}_{P_{l}}-\tilde{\alpha}_{Q_{l}}\right\Vert_{L^{2}%
}=\left\Vert \alpha_{P_{l}}-\alpha_{Q_{l}}\right\Vert_{\mathbb{H}%
^{\ast\otimes n}}\rightarrow0\text{ as }l\rightarrow\infty.
\]

\end{proof}

By polarization of the identity, $\left\Vert \tilde{\alpha}\right\Vert
_{L^{2}}^{2}=T^{n}\left\Vert \alpha\right\Vert_{n}^{2}/n!$,  it follows that%
\[
\left(
\tilde{\alpha},\tilde{\beta}\right)_{L^{2}}=\frac{T^{n}}{n!}\left(
\alpha,\beta\right)_{\mathbb{H}^{\ast\otimes n}}\text{ for all
}\alpha ,\beta\in\mathbb{H}^{\ast\otimes n}.
\]
Moreover, if $\alpha\in\mathbb{H}^{\ast\otimes n}$ and $\beta\in
\mathbb{H}^{\ast\otimes m}$ with $m\neq n$,  by the orthogonality of
the finite
dimensional approximations, $\tilde{\alpha}_{P_{l}}$ and $\tilde{\beta}%
_{P_{l}}$, we have that $\left(  \tilde{\alpha},\tilde{\beta}\right)_{L^{2}%
}=0$.
\begin{cor}
[It\^{o}'s isometry]\label{c.h9.9}Suppose that $\alpha=\left\{  \alpha
_{n}\right\}_{n=0}^{\infty}\in\bigoplus\limits_{n=0}^{\infty}\frac{T^{n}%
}{n!}\mathbb{H}^{\ast\otimes n}$,  i.e.
$\alpha_{n}\in\mathbb{H}^{\ast\otimes n}$ for all $n$ such that
\[
\left\Vert \alpha\right\Vert_{T}^{2}=\sum_{n=0}^{\infty}\frac{T^{n}}%
{n!}\left\Vert \alpha_{n}\right\Vert_{n}^{2}<\infty.
\]
Then $\tilde{\alpha}:=\sum_{n=0}^{\infty}\tilde{\alpha}_{n}$ is $L^{2}\left(
\mathbf{P}\right)  $--convergent and the map,%
\[
\bigoplus\limits_{n=0}^{\infty}\frac{T^{n}}{n!}\mathbb{H}^{\ast\otimes n}%
\ni\alpha\mapsto\tilde{\alpha}\in L^{2}\left(  \mathbf{P}\right),
\]
is an isometry, where $\mathbf{P}$ is the probability measure used
in describing the law of $\left\{  b\left(  t\right)  \right\}
_{t\geq0}$. \end{cor}

\subsection{The stochastic Taylor map\label{s.h9.2}}

Let $b\left(  t\right)  =\left(  B\left(  t\right),B_{0}\left(
t\right) \right)  \in\mathfrak{g}$ and $g\left(  t\right)  \in G$ be
the Brownian motions introduced at the start of Section \ref{s.h4}.
We are going to use the results of the previous subsection with
$\mathbb{H}=\mathfrak{g}_{CM}$,  $\mathbb{W}=\mathfrak{g,}$ and
$b\left(  t\right)  =\left(  B\left(  t\right) ,B_{0}\left( t\right)
\right)$. Let $f\in\mathcal{H}_{T}^{2}\left( G\right)  $ and
$\alpha_{f}:=\mathcal{T}_{T}S_{T}f\in J_{T}^{0}\left(
\mathfrak{g}_{CM}\right)  $. The following theorem is a (precise)
restatement of Theorem \ref{t.h1.9}.

\begin{thm}
\label{t.h9.10}For any $f\in\mathcal{H}_{T}^{2}\left(  G\right)  $
\begin{equation}
f\left(  g\left(  T\right)  \right)  =\tilde{\alpha}_{f}, \label{e.h9.6}%
\end{equation}
where $\tilde{\alpha}_{f}$ was introduced in Corollary \ref{c.h9.9}. (The
right hand side of Eq. (\ref{e.h1.8}) is to be interpreted as $\tilde{\alpha
}_{f}.$)
\end{thm}

\begin{proof}
First suppose that $f$ is a holomorphic polynomial and $P\in
\operatorname*{Proj}\left(  W\right)  $ so that $\pi_{P}\in
\operatorname*{Proj}\left(  \mathfrak{g}\right)$. Then by It\^{o}'s formula,%
\[
f\left(  g_{P}\left(  T\right)  \right)  =f\left(  \mathbf{e}\right)
+\int_{0}^{T}\left\langle Df\left(  g_{P}\left(  t\right)
\right),d\pi _{P}b\left(  t\right)  \right\rangle .
\]
Iterating this equation as in the proof of \cite[Proposition 5.2]%
{Driver1995a}, if $N\in\mathbb{N}$ is sufficiently large, then%
\begin{align*}
f\left(  g_{P}\left(  T\right)  \right)   &  =f\left(
\mathbf{e}\right) +\sum_{n=1}^{N}\int_{0\leq s_{1}\leq
s_{2}\leq\dots\leq s_{n}\leq T}\left\langle D^{n}f\left(
\mathbf{e}\right),d\pi_{P}b\left( s_{1}\right)  \otimes\dots\otimes
d\pi_{P}b\left(  s_{n}\right)  \right\rangle
\\
&  =f\left(  \mathbf{e}\right)  +\sum_{n=1}^{N}\left[  D^{n}f\left(
\mathbf{e}\right)  \right]_{\pi_{P}}^{\symbol{126}}.
\end{align*}
We now replace $P$ by $P_{k}\in\operatorname*{Proj}\left(  W\right)  $ with
$P_{k}\uparrow I$ in this identity. Using Propositions \ref{p.h4.12} and
\ref{p.h9.8}, we may now pass to the limit as $k\rightarrow\infty$ in order to
conclude,
\begin{equation}
f\left(  g\left(  T\right)  \right)  =f\left(  \mathbf{e}\right)  +\sum
_{n=1}^{N}\left[  D^{n}f\left(  \mathbf{e}\right)  \right]^{\symbol{126}%
}=\tilde{\alpha}_{f}. \label{e.h9.7}%
\end{equation}

Now suppose that $f\in\mathcal{H}_{T}^{2}\left(  G\right)$. By
Theorem \ref{t.h7.1} we can find a sequence of holomorphic
polynomials $\left\{ f_{n}\right\}
_{n=1}^{\infty}\subset\mathcal{P}$ such that
\[
\mathbb{E}\left\vert f\left(  g\left(  T\right)  \right)  -f_{n}\left(
g\left(  T\right)  \right)  \right\vert ^{2}=\left\Vert f-f_{n}\right\Vert
_{L^{2}\left(  \nu_{T}\right)  }^{2}\rightarrow0\text{ as }n\rightarrow
\infty.
\]
The isometry property of the Taylor and skeleton maps (Theorem \ref{t.h6.10}
and Corollary \ref{c.h8.3}), shows that $\alpha_{f_{n}}\rightarrow\alpha_{f}$
in $J_{T}^{0}$ and therefore by Corollary \ref{c.h9.9} $\tilde{\alpha}_{f_{n}%
}\rightarrow\tilde{\alpha}_{f}$ as $n\rightarrow\infty$. Hence we
may pass to the limit in Eq. (\ref{e.h9.7}) applied to the sequence
$f_{n}\left(  g\left( T\right)  \right)  =\tilde{\alpha}_{f_{n}}$,
to complete the proof of Eq. (\ref{e.h9.6}).
\end{proof}

\section{Future directions and questions\label{s.h10}}

In this last section we wish to speculate on a number of ways that the results
in this paper might be generalized.

\begin{enumerate}
\item It should be possible to remove the restriction on $\mathbf{C}$ being
finite dimensional, i.e. we expect much of what have done in this
paper to go through when $\mathbf{C}$ is replaced by a separable
Hilbert space. In doing so one would have to modify the finite
dimensional approximations used in our construction to truncate
$\mathbf{C}$ as well.

\item We also expect that the level of non-commutativity of $G$ may be
increased. To be more precise, under suitable hypothesis it should be possible
to handle more general graded nilpotent Lie groups.

\item Open questions:

\begin{enumerate}
\item as we noted in Remark \ref{r.h5.13} we do not know if
$\mathcal{A}_{T}^{p}=\mathcal{H}_{T}^{p}\left(  G\right)$. It might
be easier to try to answer this question for $p=2$.

\item give an intrinsic characterization of
$\mathcal{H}_{T}^{2}\left(  G\right)$ as in Shigekawa
\cite{Shigekawa1991} in terms of functions in $L^{2}\left(
\nu_{T}\right)$ solving a weak form of the Cauchy--Riemann
equations.

\end{enumerate}
\end{enumerate}

\begin{acknowledgement}
We are grateful to Professor Malliavin whose question during a workshop
at the Hausdorff Institute (Bonn, Germany) led us to include a section on a holomorphic chaos expansion.
\end{acknowledgement}

\def\cprime{$'$}
\providecommand{\bysame}{\leavevmode\hbox to3em{\hrulefill}\thinspace}
\providecommand{\MR}{\relax\ifhmode\unskip\space\fi MR }
\providecommand{\MRhref}[2]{%
  \href{http://www.ams.org/mathscinet-getitem?mr=#1}{#2}
}
\providecommand{\href}[2]{#2}

\end{document}